\documentclass[a4paper,10pt]{article}
\usepackage{graphicx}
\usepackage{amsfonts}
\usepackage{amsmath}
\usepackage{amssymb}
\usepackage{amsthm}
\usepackage{fancyhdr}
\usepackage{mathtools}
\usepackage{indentfirst}
\usepackage{placeins}
\usepackage{listings}
\usepackage{caption}
\usepackage{subfigure}
\usepackage{xcolor}
\usepackage{comment}
\usepackage{url}
\usepackage[numbers]{natbib}
\usepackage{bm}
\usepackage{hyperref}

\def\longrightharpoonup{\relbar\joinrel\rightharpoonup}
\def\longleftharpoondown{\leftharpoondown\joinrel\relbar}

\def\longrightleftharpoons{
  \mathop{
    \vcenter{
       \hbox{
       \ooalign{
          \raise1pt\hbox{$\longrightharpoonup\joinrel$}\crcr
  	  \lower1pt\hbox{$\longleftharpoondown\joinrel$}
	}
      }
    }
  }
}

\newtheorem{remark}{Remark}
\newtheorem{prop}{Proposition}
\newtheorem{thm}{Theorem}
\newtheorem{lemma}{Lemma}
\newtheorem{corollary}{Corollary}
\newtheorem{assump}{Assumption}

\DeclareMathOperator*{\argmin}{arg\,min}

\allowdisplaybreaks
\providecommand{\keywords}[1]{\textbf{\small Keywords } #1}

\begin{document}
\title{Ergodic SDEs on submanifolds and related numerical sampling schemes}

\date{}
\author{
  Wei Zhang\,\textsuperscript{1} 
}
\footnotetext[1]{Zuse Institute Berlin, Takustrasse 7, 14195 Berlin,
Germany.\quad Email: wei.zhang@fu-berlin.de}

\maketitle
\begin{abstract}
  In many applications, it is often necessary to 
   sample the mean value of certain quantity with respect to a probability
   measure $\mu$ on the level set of a smooth function $\xi:\mathbb{R}^d
   \rightarrow \mathbb{R}^k$, $1 \le k < d$.
  A specially interesting case is the so-called conditional probability measure, which is
  useful in the study of free energy calculation and model reduction of diffusion processes. 
  By Birkhoff's ergodic theorem, one approach to estimate the mean value is to
  compute the time average along an infinitely long trajectory of an ergodic diffusion process on the level set whose invariant
  measure is $\mu$. 
Motivated by the previous work of~\citet*{projection_diffusion}, as well as
the work of~\citet*{Tony-constrained-langevin2012}, in this paper we construct a family of ergodic diffusion processes on 
  the level set of $\xi$ whose invariant measures coincide with the given one.
  For the conditional measure,  we propose a consistent numerical scheme which samples the conditional measure asymptotically.
  The numerical scheme doesn't require computing the second derivatives of $\xi$ and the
  error estimates of its long time sampling efficiency are obtained.
\end{abstract}
%AMS  60J60, 53C17

\section{Introduction}
\label{sec-intro}
Many stochastic dynamical systems in real-world applications in
physics, chemistry, and biology often involve a large number of
degrees of freedom which evolve on vastly different time scales. Understanding
the behavior of these systems can be highly challenging due to the
high dimensionality and the existence of multiple time scales. 
To tackle these difficulties, the terminology \textit{reaction coordinate}, or
\textit{collective variable}, is often introduced
to help describe the essential dynamical behavior of complex systems~\cite{givon2004emd,mimick_sde,effective_dynamics,Majda2008_perspective,pavliotis2008_multiscale}.

In various research topics, in particular those related to molecular dynamics, 
one often encounters the problem of computing the mean value of certain quantity on the level set 
\begin{align}
  \Sigma=\xi^{-1}(\bm{0})=\Big\{ x \in \mathbb{R}^d ~\Big|~ \xi(x) = \bm{0} \in \mathbb{R}^k\Big\} 
\label{levelset-sigma-intro}
\end{align}
of a reaction coordinate function $\xi : \mathbb{R}^d \rightarrow \mathbb{R}^k$, $1 \le k <d$.  
Among different probability measures on $\Sigma$, the one defined by
\begin{align}
d\mu_1 = \frac{1}{Z} e^{-\beta U} \big[\mbox{det}(\nabla\xi^T
\nabla\xi)\big]^{-\frac{1}{2}} d\nu
\label{mu1-intro}
\end{align}
is especially relevant in applications and is sometimes called the conditional probability measure
on $\Sigma$. In (\ref{mu1-intro}), the parameter $\beta > 0$, $U: \mathbb{R}^d \rightarrow
\mathbb{R}$ is a smooth function, $Z$ is the normalization constant, 
$\nabla\xi$ denotes the $d \times k$ Jacobian matrix of the map $\xi$, 
and $\nu$ is the surface measure on $\Sigma$ induced from the Lebesgue measure on $\mathbb{R}^d$.  
The probability measure $\mu_1$ has a probabilistic interpretation, and the numerical computation of the mean value 
\begin{align}
  \overline{f} = \int_\Sigma f(x)\, d\mu_1(x)
  \label{mean-f-intro}
\end{align}
for a function $f$ on the level set is involved in various contexts, such as 
free energy calculations based on the thermodynamics integration
formula~\cite{lelievre2010free,Tony-constrained-langevin2012,non-equilibrium-2018}.

Applying Birkhoff's ergodic
theorem, the mean value $\overline{f}$ can be approximated by the time average 
$\frac{1}{T} \int_0^T f(X_s) ds$
along a long trajectory of the process $X_s$ which evolves on the level set
$\Sigma$ and has the invariant measure $\mu_1$. 
For this purpose, it is helpful to construct a diffusion process on the level set
with the correct invariant measure $\mu_1$, i.e., to write down
the stochastic differential equation (SDE) of $X_s$ in $\mathbb{R}^d$. 
While finding such a SDE is trivial in the linear reaction
coordinate case~\cite{vanden-eijnden2003}, it is not obvious when the reaction
coordinate $\xi$ is a nonlinear function of system's state. 

In the literature, the problem finding SDEs on the level set of the reaction coordinate
function with a given invariant measure has been considered in the 
study of free energy calculations~\cite{blue-moon,lelievre2010free,projection_diffusion,Tony-constrained-langevin2012}. Given a smooth function
$U : \mathbb{R}^d \rightarrow \mathbb{R}$, 
the authors in~\cite{projection_diffusion} constructed a diffusion process
$Y_s$ on $\Sigma$ whose unique invariant measure is $\mu_2$, given by 
\begin{align}
d\mu_2 = \frac{1}{Z} e^{-\beta U} d\nu\,.
\label{mu2-intro}
\end{align}
It is also shown in~\cite{projection_diffusion} that this process $Y_s$ can be obtained by projecting 
the dynamics 
\begin{align}
  d \widetilde{Y}_s & = -\nabla U (\widetilde{Y}_s)\, ds +
  \sqrt{2\beta^{-1}} dW_s 
\label{dynamics-id}
\end{align}
from $\mathbb{R}^d$ onto the level set $\Sigma$, where
 $W_s = (W_s^1, \cdots, W_s^d)^T$ is a $d$-dimensional Brownian motion.
 The dynamics $Y_s$ can be used to sample $\mu_2$, and therefore to sample the
 conditional measure $\mu_1$ in (\ref{mu1-intro}) as well, by either modifying
 the potential $U$ or reweighting the
function $f$ according to the factor $\big[\mbox{det}(\nabla\xi^T \nabla\xi)\big]^{-\frac{1}{2}}$.
In a more recent work~\cite{Tony-constrained-langevin2012}, 
the authors studied the constrained Langevin dynamics, which evolves on the submanifold of the entire phase space 
including both position and momentum.
It is shown in~\cite{Tony-constrained-langevin2012} that 
 the position components of the constrained
Langevin dynamics has the marginal invariant measure which coincides with $\mu_2$.
Therefore, it can also be used to compute the average $\overline{f}$ with
respect to the conditional measure $\mu_1$ (by either modifying the potential
or reweighting $f$ according to 
$\big[\mbox{det}(\nabla\xi^T \nabla\xi)\big]^{-\frac{1}{2}}$).
Detailed studies on the numerical schemes as well as applications of the constrained Langevin dynamics have been carried out in~\cite{Tony-constrained-langevin2012}. 

The same conditional probability measure $\mu_1$ in (\ref{mu1-intro}), as well as 
the average $\overline{f}$ in (\ref{mean-f-intro}), also plays an important role
in the study of the effective dynamics of diffusion processes
\cite{kevrekidis2003,froyland2014,effective_dynamics,effective_dyn_2017}.
As a generalization of the dynamics (\ref{dynamics-id}), 
 the diffusion process 
\begin{equation}\label{dynamics-v}
\begin{aligned}
  d \widetilde{Y}^i_s & = -\Big(a_{ij} \frac{\partial U}{\partial x_j}\Big)
  (\widetilde{Y}_s)\,ds +
  \frac{1}{\beta} \frac{\partial a_{ij}}{\partial x_j}(\widetilde{Y}_s)\, ds +
  \sqrt{2\beta^{-1}} \sigma_{ij}(\widetilde{Y}_s)\, dW^j_s\,, \quad 1 \le i \le d\,,
\end{aligned}
\end{equation}
and its effective dynamics have been considered in~\cite{effective_dyn_2017}, where the matrix-valued coefficients $\sigma, a : \mathbb{R}^d
\rightarrow \mathbb{R}^{d\times d}$ are related by $a = \sigma\sigma^T$, such
that $a$ is uniformly positive definite. Notice that, 
(\ref{dynamics-v}) is written in component-wise form with Einstein's summation
convention (the same Einstein's summation convention will be used throughout this paper, whenever no ambiguity will arise), and it reduces to (\ref{dynamics-id}) when $\sigma=a = \mbox{id}$.
The infinitesimal generator of (\ref{dynamics-v}) can be written as 
\begin{align}
  \mathcal{L} = \frac{e^{\beta U}}{\beta}  \frac{\partial}{\partial x_i}\Big(e^{-\beta U}a_{ij}\frac{\partial}{\partial x_j}\Big)\,. \label{overdamp-l}
\end{align}
Under mild conditions on $U$, it is known that,
for any (smooth, uniformly positive definite) coefficient $a$, the dynamics
(\ref{dynamics-v}) has the common unique invariant measure whose probability
density is $\frac{1}{Z}e^{-\beta U}$ with respect to the Lebesgue measure on
$\mathbb{R}^d$.

Motivated by these previous work, in this paper we try to answer the following
two questions.
\vspace{0.05cm}

{
(Q1) \hspace{0.1cm} \textit{Besides the process constructed in
  \cite{projection_diffusion} that is closely related to (\ref{dynamics-id}),
  can we obtain other diffusion processes  on $\Sigma$, which are probably
  related to (\ref{dynamics-v}) involving the coefficients $\sigma, a$, and have the same invariant measure? 
In particular, can we construct SDEs on $\Sigma$ whose invariant measure is
$\mu_1$? 
}
\vspace{0.1cm}

(Q2) \hspace{0.1cm} \textit{Numerically, instead of sampling $\mu_2$,
can we directly estimate the mean value in (\ref{mean-f-intro}) with respect to $\mu_1$, preferably with a numerical algorithm that is easy to implement? }
\vspace{0.3cm}

The main contributions of the current work are related to the above questions
and are summarized below. First, concerning Question (Q1), 
in Theorem~\ref{thm-mu1-mu2} of Section~\ref{sec-n-rn}, we will construct a family of diffusion processes on $\Sigma$ which sample
either $\mu_1$ or $\mu_2$. In particular, we show that the diffusion process
\begin{align}
 dX_s^i 
= & - (Pa)_{ij} 
\frac{\partial U}{\partial x_j}\,ds + \frac{1}{\beta} \frac{\partial
(Pa)_{ij}}{\partial x_j}\,ds + \sqrt{2\beta^{-1}}\, P_{j,i}\, dW_s^j\,, \quad 1 \le i \le d\,,
\label{dynamics-1-submanifold-intro}
\end{align}
evolves on $\Sigma$ and the invariant measure is the conditional probability
measure $\mu_1$ in (\ref{mu1-intro}), where the projection map $P$ and the invertible $k \times k$ symmetric matrix $\Psi$
are given by 
\begin{align}
P=\mbox{id}- a\nabla\xi\Psi^{-1}\nabla\xi^T,\quad  \Psi = \nabla\xi^T a \nabla \xi.
\label{p-ij-brief-def}
\end{align}
Correspondingly, the infinitesimal generator of (\ref{dynamics-1-submanifold-intro}) is 
\begin{align}
  \mathcal{L} = \frac{e^{\beta U}}{\beta} \frac{\partial}{\partial x_i}\Big(e^{-\beta
  U}(Pa)_{ij}\frac{\partial}{\partial x_j}\Big)\,, 
  \label{generator-l-mu1-intro}
\end{align}
which should be compared to the infinitesimal generator in (\ref{overdamp-l}).
Second, concerning Question (Q2), in Section~\ref{sec-numerical-scheme-mu1} we study a numerical algorithm which
estimates the mean value $\overline{f}$ in (\ref{mean-f-intro}). Specifically,
we propose to use the numerical scheme
\begin{align}
  \begin{split}
    x^{(l + \frac{1}{2})}_i =& x^{(l)}_i + \Big(-a_{ij}\frac{\partial U}{\partial
  x_j} + \frac{1}{\beta}\frac{\partial a_{ij}}{\partial x_j}\Big)(x^{(l)})\, h
  + \sqrt{2 \beta^{-1}h}\, \sigma_{ij}(x^{(l)})\, \eta^{(l)}_j\,, \quad 1 \le i \le
  d \,,\\
  x^{(l+1)} =& \Theta\big(x^{(l+\frac{1}{2})}\big)\,,
\end{split}
\label{scheme-mu1-intro}
\end{align}
with $x^{(0)} \in \Sigma$, and to approximate $\overline{f}$ by $\widehat{f}_n = \frac{1}{n} \sum\limits_{l=0}^{n-1} f(x^{(l)})$.
In (\ref{scheme-mu1-intro}), $h$ is the step-size, 
$\bm{\eta}^{(l)}=(\eta_1^{(l)}, \eta_2^{(l)}, \cdots,
\eta_d^{(l)})^T$ are independent $d$-dimensional standard Gaussian random
variables, and 
$\Theta(x) = \lim\limits_{s\rightarrow +\infty}\varphi(x,s)$ is the limit of the flow map 
\begin{align}
  \begin{split}
    \frac{d\varphi(x,s)}{ds} =& - (a\nabla F) \big(\varphi(x,s)\big)\,,\quad
    \varphi(x,0) = x, \qquad  \forall~x\in \mathbb{R}^d\,, 
\end{split}
  \label{phi-map-intro}
\end{align}
with $F(x)=\frac{1}{2}|\xi(x)|^2=\frac{1}{2}\sum\limits_{\alpha=1}^{k}\xi_\alpha^2(x)$.
Following the approach developed in~\cite{conv-time-averaging},
in Theorem~\ref{thm-estimate-scheme-on-submanifold}, we obtain the
estimates of the approximation error between $\widehat{f}_n$ and
$\overline{f}$. While different
constraint approaches have been proposed in the
literature~\cite{Leimkuhler-geodesic-2016,Tony-constrained-langevin2012,goodman-submanifold}, to the best of the
author's knowledge, constraint using the flow map $\varphi$ has not been studied yet. 

Let us comment on the two contributions mentioned above.
First, knowing the SDE (\ref{dynamics-1-submanifold-intro}) and the expression
(\ref{generator-l-mu1-intro}) of its infinitesimal generator $\mathcal{L}$ is helpful for analysis. In fact, in Section~\ref{sec-numerical-scheme-mu1}, 
the analysis of sampling error estimate of the scheme \eqref{scheme-mu1-intro} 
relies on Poisson equation on $\Sigma$ related to $\mathcal{L}$ in (\ref{generator-l-mu1-intro}). 
Furthermore, (\ref{generator-l-mu1-intro}) plays a role in the work~\cite{zhang-lelievre-pathwise2018} in analyzing the approximation quality of the effective dynamics, while SDE (\ref{dynamics-1-submanifold-intro})
has been used in~\cite{non-equilibrium-2018} to study fluctuation relations and Jarzynski's equality for nonequilibrium systems.
Second, we emphasize that $\Theta(x)$ in the scheme (\ref{scheme-mu1-intro}) can be evaluated by solving the ODE (\ref{phi-map-intro}) starting from $x$. 
Although $\Theta$ is defined as the limit when $s \rightarrow +\infty$,
in many cases the computational cost is not large, due to the exponential
convergence of the (gradient) flow (\ref{phi-map-intro}) to its limit,
particularly for the initial state $x=x^{(l+\frac{1}{2})}$ that is close to $\Sigma$. 
Furthermore, comparing to the direct (Euler-Maruyama) discretization of SDE
(\ref{dynamics-1-submanifold-intro}) which may deviate from $\Sigma$ and
requires second order derivatives of $\xi$, the scheme (\ref{scheme-mu1-intro})
satisfies $x^{(l)} \in \Sigma$ for all $l \ge 0$, and 
it doesn't require computing the second order derivatives of $\xi$. 
Therefore, we expect the numerical scheme
(\ref{scheme-mu1-intro})--(\ref{phi-map-intro}) is both stable and relatively easy to implement.
Readers are referred to Remark~\ref{rmk-on-numerical-scheme-1}--\ref{rmk-on-numerical-scheme-about-ode}
in Section~\ref{sec-numerical-scheme-mu1} and Example $1$ in Section~\ref{sec-example} for further algorithmic discussions. 

In the following, we briefly explain the approach that we will use to study Question~(Q1), as well as the idea behind the scheme (\ref{scheme-mu1-intro})--(\ref{phi-map-intro}).
Concerning Question~(Q1), we take the manifold point of view by considering
$\mathbb{R}^d$ as a Riemannian manifold $\mathcal{M}=(\mathbb{R}^d, g)$ with the metric $g=a^{-1}$, defined by 
\begin{align}
  g(\bm{u}, \bm{v}) = \langle \bm{u}, \bm{v}\rangle_g = u_i (a^{-1})_{ij}
  v_j\,, \quad \forall~\bm{u}, \bm{v} \in \mathbb{R}^d\,.
  \label{ip-a-intro}
\end{align}
A useful observation is that, for $\mathcal{L}$ in \eqref{overdamp-l}, we have~\cite{effective_dyn_2017}
\begin{align*}
  \mathcal{L} f &= \Big[-\mbox{grad}^{\mathcal{M}} \Big(U + \frac{1}{2\beta} \ln G\Big) + \frac{1}{\beta} \Delta^{\mathcal{M}}\Big] f\,,
  \quad \forall~\mbox{smooth}~ f : \mathbb{R}^d \rightarrow \mathbb{R}\,,
\end{align*}
where $G=\mbox{det}g$, and $\mbox{grad}^{\mathcal{M}}$, $\Delta^{\mathcal{M}}$ denote the gradient and the Laplacian-Beltrami operator on $\mathcal{M}$, respectively. 
Accordingly, (\ref{dynamics-v}) can be written as a SDE on $\mathcal{M}$ as
\begin{align}
  d\widetilde{Y}_s = -\mbox{grad}^{\mathcal{M}}\Big(U + \frac{1}{2\beta} \ln G\Big)\, ds + \sqrt{2\beta^{-1}}
  d\widetilde{B}_s\,,
  \label{sde-x-manifold}
\end{align}
where $\widetilde{B}_s$ is the Brownian motion on $\mathcal{M}$~\cite{analysis_manifold}. 
Conversely, SDE (\ref{dynamics-v}) can be seen as the equation of (\ref{sde-x-manifold})
under the (global) coordinate chart of $\mathcal{M}$. This equivalence allows us to study (\ref{dynamics-v}) on $\mathbb{R}^d$ by
the corresponding SDE (\ref{sde-x-manifold}) on manifold $\mathcal{M}$. 
Comparing to (\ref{dynamics-v}), one advantage of working with
the abstract equation (\ref{sde-x-manifold}) is that the invariant measure
of (\ref{sde-x-manifold}) can be recognized as
easily as in (\ref{dynamics-id}), provided that we apply integration by parts
formula on the manifold $\mathcal{M}$. 

A family of ergodic SDEs on $\Sigma$ (i.e., Question (Q1)) is obtained by taking the same manifold
point of view. Specifically, consider $\Sigma$ as a submanifold of $\mathcal{M}$ and
denote by $\mbox{grad}^{\Sigma}$, $\Delta^{\Sigma}$, $B_s$ the gradient operator, the Laplacian and the Brownian motion 
(with generator $\frac{1}{2} \Delta^{\Sigma}$~\cite{analysis_manifold}) on $\Sigma$, respectively.
Since the infinitesimal generator of the SDE 
\begin{align}
  dY_s = - \mbox{grad}^{\Sigma} U ds + \sqrt{2 \beta^{-1}} dB_s
  \label{sde-manifold-n-intro}
\end{align}
is $\mathcal{L} = -\mbox{grad}^{\Sigma} U + \frac{1}{\beta} \Delta^{\Sigma}$, 
under mild assumptions on $U$, it is straightforward to verify that dynamics (\ref{sde-manifold-n-intro}) 
evolves on $\Sigma$ and has the unique invariant measure $\frac{1}{Z}e^{-\beta U}d\nu_g$,
where $\nu_g$ is the surface measure on $\Sigma$ induced from the metric $g=a^{-1}$ on $\mathcal{M}=(\mathbb{R}^d, g)$.
 Therefore, answering Question (Q1) boils down to calculating the
 expression of (\ref{sde-manifold-n-intro}) under the coordinate
 chart of $\mathcal{M}$ (not $\Sigma$). 
 This will be achieved by calculating the expressions
 of $\mbox{grad}^{\Sigma}$, $\Delta^{\Sigma}$ under the coordinate chart of
 $\mathcal{M}$ and then figuring out the
 relation between the two measures $\nu$ and $\nu_g$.

Concerning the idea behind the numerical scheme (\ref{scheme-mu1-intro})--(\ref{phi-map-intro}),
we recall that one way to (approximately) sample $\mu_1$ on $\Sigma$ is to constrain
the dynamics (\ref{dynamics-v}) in the neighborhood of $\Sigma$ by adding an extra potential to it. 
This is often termed as softly constrained dynamics~\cite{projection_diffusion,Maragliano2006}
and has been widely used in applications.
In this context, one consider the dynamics
\begin{equation}\label{dynamics-v-soft-intro}
\begin{aligned}
  d X^{\epsilon,i}_s & = \Big[- a_{ij} \frac{\partial U}{\partial
x_j}\, -
\frac{1}{\epsilon}a_{ij}\frac{\partial }{\partial x_j}\Big(\frac{1}{2}\sum\limits_{\alpha=1}^k
  \xi_\alpha^2\Big)\,+
\frac{1}{\beta} \frac{\partial a_{ij}}{\partial x_j}\Big]\, ds +
  \sqrt{2\beta^{-1}} \sigma_{ij}\, dW^j_s\,, 
\end{aligned}
\end{equation}
where $\epsilon > 0$, $1 \le i \le d$, based on the fact that the invariant measure
of (\ref{dynamics-v-soft-intro}) converges weakly to $\mu_1$, as $\epsilon \rightarrow 0$.
The dynamics (\ref{dynamics-v-soft-intro}) stays close to $\Sigma$ most of the time, thanks to the existence of the extra constraint force.
Furthermore, only the first order
derivatives of $\xi$ are involved. In spite of these nice properties, however, direct simulation of
(\ref{dynamics-v-soft-intro}) is inefficient when $\epsilon$ is small, 
because the time step-size in numerical simulations becomes severely
limited due to the strong stiffness in the dynamics. 
Indeed, our numerical scheme is motivated in order to overcome the aforementioned drawback of the softly constrained dynamics (\ref{dynamics-v-soft-intro}),  
and the scheme (\ref{scheme-mu1-intro})--(\ref{phi-map-intro}) can be viewed as a multiscale numerical method for
(\ref{dynamics-v-soft-intro}), where the stiff and non-stiff terms in 
(\ref{dynamics-v-soft-intro}) are handled separately~\cite{vanden-eijnden2003}. 
In contrast to the previous work~\cite{katzenberger1991,fatkullin2010,projection_diffusion}, where the
convergence of (\ref{dynamics-v-soft-intro}) was studied on a finite time interval, our result concerns the long time sampling efficiency of the discretized numerical scheme.

Before concluding this introduction, we compare the current work with several previous ones. Generally speaking, Monte Carlo samplers (based on
ergodicity) either on $\mathbb{R}^d$ or on its submanifolds can be
classified into Metropolis-adjusted samplers and samplers without Metropolis step (unadjusted).
For Metropolis-adjusted methods, in particular, exploiting
Riemannian geometry structure to develop MCMC methods has been studied in~\cite{riemann-mcmc}.
The authors there demonstrated that incorporating the geometry of the space into numerical methods can lead
to significant improvement of the sampling efficiency. 
In line with this development, in Section~\ref{sec-example} we will consider a concrete example
where a non-constant matrix $a$ can help remove the stiffness in the sampling task.
On the other hand, despite of the common Riemannian manifold point of view in the current work and in~\cite{riemann-mcmc}, 
the main difference is that the current work deals with sampling on the
submanifold $\Sigma$ instead of the entire $\mathbb{R}^d$ (or its domain).
The derivations in the current work are more involved mainly due to this difference.
Besides sampling on the entire space, Metropolis-adjusted samplers on submanifolds,
using either MCMC or Hybrid Monte Carlo, have been considered in several recent
work~\cite{pmlr-v22-brubaker12,Tony-constrained-langevin2012,goodman-submanifold,tony-gabriel-hmc-submanifolds}. Reversible Metropolis random walk on submanifolds has been
constructed in~\cite{goodman-submanifold}, which is then extended in~\cite{tony-gabriel-hmc-submanifolds} by allowing non-zero gradient forces in the proposal move. 
In contrast to these Metropolis-adjusted samplers, the numerical scheme
(\ref{scheme-mu1-intro})--(\ref{phi-map-intro}) in the current work is unadjusted (without Metropolis-step) and samples the conditional probability measure $\mu_1$
when the step-size $h\rightarrow 0$. This means that in practice the step-size $h$ should be chosen properly such that the discretization error is tolerable.
In this direction, we point out that unadjusted samplers on $\mathbb{R}^d$,
which naturally arise from discretizations of SDEs, have been well studied in
the
literature~\cite{talay-tubaro-1990,gabriel-ben-ima-2016,bou-rabee-2010,debussche_weak_backward2012,abdulle_high-order-2014,conv-time-averaging}.
The current work can be thought as a further step along this direction for sampling schemes on submanifolds,
by applying the machinery developed in~\cite{conv-time-averaging}. 
Comparison between the scheme (\ref{scheme-mu1-intro}) and the Metropolis-adjusted algorithm
in~\cite{goodman-submanifold} can be found in
Remark~\ref{rmk-cmp-metropolis-or-not}, as well as in Example $2$ in Section~\ref{sec-example}.
We also refer to~\cite{gabriel-ben-ima-2016} for related discussions.

The rest of the paper is organized as follows. In Section~\ref{sec-n-rn}, we construct ergodic SDEs on $\Sigma$ which
sample either $\mu_1$ or $\mu_2$. In Section~\ref{sec-numerical-scheme-mu1}, we study the numerical scheme
(\ref{scheme-mu1-intro})--(\ref{phi-map-intro}) and quantify its approximation error in estimating the mean value in
(\ref{mean-f-intro}). In Section~\ref{sec-example}, we demonstrate our results through concrete examples.
Conclusions and further discussions are made in Section~\ref{sec-conclusion}.
Technical details related to the Riemannian manifold $\mathcal{M}$ in Section~\ref{sec-n-rn} are included in Appendix~\ref{app-sec-manifold}.
Proofs of the results in Section~\ref{sec-numerical-scheme-mu1} are collected in Appendix~\ref{app-sec-numerical}.

Finally, we conclude this introduction with the assumptions which will be
made (implicitly) throughout this paper.
\begin{assump}
  The matrix $\sigma: \mathbb{R}^d \rightarrow \mathbb{R}^{d \times d}$ is
  both smooth and invertible at each $x \in \mathbb{R}^d$. The matrix
  $a=\sigma\sigma^{T}$ is uniformly positive definite with uniformly bounded
  inverse $a^{-1}$. 
 \label{assump-1}
\end{assump}
\begin{assump}
  The function $\xi: \mathbb{R}^d \rightarrow \mathbb{R}^k$ is $C^2$ smooth and the level set $\Sigma$ is both connected and compact, such that 
  $\mbox{rank}(\nabla\xi) = k$ at each $x \in \Sigma$. 
  \label{assump-2}
\end{assump}
\section{SDEs of ergodic diffusion processes on $\Sigma$}
\label{sec-n-rn}
In this section, we construct SDEs of ergodic processes on $\Sigma$ that sample a given invariant measure. The main result of this section is Theorem~\ref{thm-mu1-mu2}, which shows that the invariant measure of the SDE \eqref{dynamics-1-submanifold-intro} in Introduction is the conditional probability measure
$\mu_1$ in \eqref{mu1-intro}. Readers who are mainly interested in numerical
algorithms can read Theorem~\ref{thm-mu1-mu2} and then directly jump to Section~\ref{sec-numerical-scheme-mu1}. 

First of all, let us point out that, the semigroup approach based on
functional inequalities on Riemannian manifolds is well developed 
to study the solution of Fokker-Planck equation towards equilibrium.  
One sufficient condition for the exponential convergence of the
Fokker-Planck equation (and therefore the ergodicity of the corresponding dynamics) is the famous Bakry-Emery criterion~\cite{Bakry1994}. In particular, concrete conditions
are given in~\cite{STURM2005} which guarantee the exponential convergence to
the unique invariant measure. In the following, we will always assume that the
potential $U \in C^\infty(\Sigma)$ and the Bakry-Emery condition in \cite{STURM2005} is satisfied.

Recall that $\mathcal{M}=(\mathbb{R}^d, g)$, where $g=a^{-1}$ and $\nu_g$ is the surface measure on $\Sigma$ induced from
$\mathcal{M}$. Matrices $P$, $P_{j,i}$ are given in \eqref{p-ij}, \eqref{p-i-j} in Appendix~\ref{app-sec-manifold}, respectively.
For $1 \le i \le d$, $\bm{e}_i$ denotes the vector whose $i$th component equals to $1$ while all the other $d-1$ components equal to $0$.
We refer the reader to Appendix~\ref{app-sec-manifold} for further details.
Let us first consider the 
probability measure $\mu$ on $\Sigma$ given by $d\mu = \frac{1}{Z} e^{-\beta
U} d\nu_g$, where $\beta > 0$ and $Z$ is a normalization constant. 
The following proposition is a direct application of
Proposition~\ref{laplaceN-state} in Appendix~\ref{app-sec-manifold}.
\begin{prop}
  Consider the dynamics on $\mathbb{R}^d$ which satisfies the Ito SDE 
\begin{align}
  \begin{split}
    dY_s^i =& - (Pa)_{ij} 
    \frac{\partial \big[U - \frac{1}{2\beta}\ln \big((\det
    a)^{-1}\det(\nabla\xi^T a \nabla\xi)\big)\big]}{\partial x_j}\,ds + \frac{1}{\beta} \frac{\partial
(Pa)_{ij}}{\partial x_j}\,ds + \sqrt{2\beta^{-1}}\, P_{j,i}\, dW_s^j
\end{split}
\label{sde-mu}
\end{align}
for $1 \le i \le d$, where $W_s = (W_s^1, W_s^2, \cdots, W_s^d)^T$ is a $d$-dimensional Brownian motion.
Suppose $Y_0 \in \Sigma$, then $Y_s \in \Sigma$ almost surely for $s\ge 0$. 
Furthermore, it has a unique invariant measure $\mu$ given by $d\mu =
\frac{1}{Z} e^{-\beta U} d\nu_g$.
\label{sde-inv-mu}
\end{prop}
\begin{proof}
  Using (\ref{p-i-j}) and Proposition~\ref{laplaceN-state} in Appendix~\ref{app-sec-manifold}, we know that the infinitesimal generator 
  of SDE (\ref{sde-mu}) is 
  \begin{align}
    \mathcal{L}f = -\langle\mbox{grad}^{\Sigma} U, \mbox{grad}^{\Sigma} f\rangle_g + \frac{1}{\beta}
  \Delta^{\Sigma} f\,, \quad \forall~f: \Sigma \rightarrow \mathbb{R}\,,
    \label{l-prop-1}
    \end{align}
    where $\mbox{grad}^{\Sigma}$, $\Delta^{\Sigma}$ are the gradient and Laplace-Beltrami operators on $\Sigma$ of $\mathcal{M}$, respectively.
    Applying Ito's formula to $\xi_\alpha(Y_s)$, we have 
    \begin{align*}
      d\xi_\alpha(Y_s) = \mathcal{L}\xi_\alpha(Y_s)\,ds + \sqrt{2\beta^{-1}} \frac{\partial
      \xi_\alpha(Y_s)}{\partial x_i} P_{j, i} dW_s^j\,, \quad 1 \le
      \alpha \le k\,.
    \end{align*}
    Using (\ref{l-prop-1}), \eqref{p-i-j} in Appendix~\ref{app-sec-manifold}, and the fact that $\mbox{grad}^{\Sigma}\xi_\alpha =
    P\mbox{grad}^{\mathcal{M}}\xi_\alpha=0$, it is straightforward to verify that
    \begin{align*}
      &\mathcal{L} \xi_\alpha = -\langle \mbox{grad}^{\Sigma} U, \mbox{grad}^{\Sigma}\xi_\alpha\rangle_g + \frac{1}{\beta}
      \mbox{div}^{\Sigma}(\mbox{grad}^{\Sigma}\xi_\alpha) = 0 \,,\\
      &\frac{\partial \xi_\alpha}{\partial x_i} P_{j, i} = 0 \,,
      \quad 1 \le j \le d\,,
    \end{align*}
    on $\Sigma$, which implies $d\xi_\alpha(Y_s) = 0$, $\forall s \ge 0$.
    Since $Y_0 \in \Sigma$, we conclude that $\xi_\alpha(Y_s) = \xi_\alpha(Y_0) = 0$ 
    a.s. $s \ge 0$, for $1 \le \alpha \le k$, and
    therefore $Y_s \in \Sigma$ for  $s \ge 0$, almost surely. 

    Using the expression (\ref{l-prop-1}) of $\mathcal{L}$ and the integration
    by parts formula (\ref{integrate-by-part-sigma}) in Appendix~\ref{app-sec-manifold}, it is easy to see that $\mu$ is an invariant
    measure of the dynamics (\ref{sde-mu}).  The uniqueness is implied by the
    exponential convergence result established in \cite[Remark 1.1 and Corollary 1.5]{STURM2005}, since we assume Bakry-Emery condition is satisfied.
\end{proof}

In the above, we have considered the level set $\Sigma$ as a submanifold of $\mathcal{M}=(\mathbb{R}^d, g)$. 
In applications, on the other hand, it is natural to view $\Sigma$ as a submanifold of the
standard Euclidean space $\mathbb{R}^d$, with the surface measure $\nu$ on $\Sigma$ that is induced
from the Euclidean metric on $\mathbb{R}^d$. 
As already mentioned in the Introduction, the following two probability measures 
\begin{align}
d\mu_1 = \frac{1}{Z} e^{-\beta U} \big[\mbox{det}(\nabla\xi^T \nabla\xi)\big]^{-\frac{1}{2}} d\nu\,,\qquad d\mu_2 = \frac{1}{Z} e^{-\beta U} d\nu, 
\label{mu-1-2}
\end{align}
where $Z$ denotes possibly different normalization constants,
are often interesting and arise in many
situations~\cite{blue-moon,projection_diffusion,effective_dynamics,effective_dyn_2017}.
In particular, $\mu_1$ has a probabilistic interpretation and 
often appears in the study of free energy calculation and model reduction of
stochastic dynamics~\cite{lelievre2010free,effective_dynamics}. 
In order to construct processes which sample $\mu_1$ or $\mu_2$, we need to 
figure out the relations between the two surface measures $\nu_g$ and $\nu$ on
$\Sigma$. 
\begin{lemma}
  Let $\nu_g$, $\nu$ be the surface measures on $\Sigma$ induced
  from the metric $g=a^{-1}$ and the Euclidean metric on $\mathbb{R}^d$,
  respectively. We have 
  \begin{align*}
    d\nu_g =
    (\det a)^{-\frac{1}{2}}
\bigg[\frac{\mbox{\normalfont{det}}(\nabla\xi^T
    a\nabla\xi)}{\mbox{\normalfont{det}}(\nabla\xi^T\nabla\xi)}\bigg]^{\frac{1}{2}}
d\nu\, .
\end{align*}
\label{lemma-sigma-relation}
\end{lemma}
\begin{proof}
  Let $x \in \Sigma$ and $\bm{v}_1, \bm{v}_2, \cdots, \bm{v}_{d-k}$ be a basis
  of $T_x\Sigma$. 
  Assume that $\bm{v}_i = c_{ij} \bm{e}_j$, where
  $c = (c_{ij})$ is a
  $(d-k) \times d$ matrix whose rank is $d-k$.
  Using the fact $\langle \bm{v}_i, \mbox{grad}^{\mathcal{M}}\xi_\alpha\rangle_g = 0$ for $1\le i \le d-k$, $1 \le
   \alpha\le k$, we can deduce that
   $c\,\nabla\xi = 0$. 
  Calculating the surface measures $\nu_g$ and $\nu$ under this basis, we
  obtain
  \begin{align}
    d\nu_g =
    \bigg[\frac{\mbox{det}(ca^{-1}c^T)}{\mbox{det}(cc^T)}\bigg]^{\frac{1}{2}}
    d\nu\,.
    \label{sigma-g-sigma}
  \end{align}
  To simplify the right hand side of (\ref{sigma-g-sigma}), we use the following equality 
  \begin{align*}
    \begin{pmatrix}
      c \\
      \nabla\xi^T\, a 
    \end{pmatrix}
    \begin{pmatrix}
      c^T & \nabla\xi
    \end{pmatrix}
    = 
    \begin{pmatrix}
      cc^T & 0 \\
      \nabla\xi^T a c^T & \nabla\xi^T a \nabla\xi 
    \end{pmatrix}
    = 
    \begin{pmatrix}
      ca^{-1} \\
      \nabla\xi^T
    \end{pmatrix}
    a
    \begin{pmatrix}
      c^T & \nabla\xi
    \end{pmatrix}\,.
  \end{align*}
  After computing the determinants of the last two matrices above, we obtain
  \begin{align*}
    \det (cc^T) \,\det(\nabla\xi^T a \nabla\xi) 
    =
    (\det a)\,\det \bigg[
        \begin{pmatrix}
      ca^{-1} \\
      \nabla\xi^T
    \end{pmatrix}
    \begin{pmatrix}
      c^T & \nabla\xi
    \end{pmatrix}
    \bigg]  
    = (\det a)\,~ \det 
    \begin{pmatrix}
      ca^{-1}c^T & ca^{-1}\nabla\xi \\
      0 & \nabla\xi^T\nabla\xi
    \end{pmatrix}\,.
  \end{align*}
  The conclusion follows after we substitute the above relation into
  (\ref{sigma-g-sigma}).
\end{proof}

Applying Lemma~\ref{lemma-sigma-relation} and
Proposition~\ref{sde-inv-mu}, we can obtain ergodic processes whose invariant
measures are given in (\ref{mu-1-2}). 
\begin{thm}
  Let $\mu_1, \mu_2$ be the two probability measures on $\Sigma$
  defined in (\ref{mu-1-2}).
  Consider the dynamics $X_s$, $Y_s$ on $\mathbb{R}^d$ which satisfy the Ito SDEs 
\begin{align}
  \begin{split}
 dX_s^i 
= & - (Pa)_{ij} 
\frac{\partial U}{\partial x_j}\,ds + \frac{1}{\beta} \frac{\partial
(Pa)_{ij}}{\partial x_j}\,ds + \sqrt{2\beta^{-1}}\, P_{j,i}\, dW_s^j\,,
  \end{split}
\label{dynamics-1-submanifold}
\end{align}
and
\begin{align}
  \begin{split}
  dY_s^i = & - (Pa)_{ij} 
\frac{\partial \big[U - \frac{1}{2\beta}\ln \det(\nabla\xi^T \nabla\xi)\big]}{\partial x_j}\,ds + \frac{1}{\beta} \frac{\partial
(Pa)_{ij}}{\partial x_j}\,ds + \sqrt{2\beta^{-1}}\, P_{j,i}\, dW_s^j\,,
  \end{split}
\label{dynamics-2-submanifold}
\end{align}
for $1 \le i \le d$, where $\beta > 0$ and $W_s = (W_s^1, W_s^2,\cdots, W_s^d)^T$ is a
$d$-dimensional Brownian motion. Suppose that $X_0 , Y_0 \in \Sigma$, then $X_s,
Y_s \in \Sigma$ almost surely for $s\ge 0$. 
Furthermore, the unique invariant probability measures of the dynamics $X_s$ and $Y_s$ 
are $\mu_1$ and $\mu_2$, respectively.
\label{thm-mu1-mu2}
\end{thm}
\begin{proof}
  Applying Lemma~\ref{lemma-sigma-relation}, we can rewrite the probability measures
  $\mu_1, \mu_2$ as
  \begin{align}
    \begin{split}
    &d\mu_1 = \frac{1}{Z} e^{-\beta
    U} \big[\mbox{det}(\nabla\xi^T\nabla\xi)\big]^{-\frac{1}{2}} d\nu
    = \frac{1}{Z} \mbox{exp}\Big[-\beta \Big(U + \frac{1}{2\beta} \ln
	\frac{\det(\nabla\xi^T a \nabla\xi)}{\det a} \Big)\Big]
    d\nu_g\,,\\
    &d\mu_2 = \frac{1}{Z} e^{-\beta
    U}d \nu= \frac{1}{Z} \mbox{exp}\Big[-\beta \Big(U + \frac{1}{2\beta}
    \ln\frac{\det\big(\nabla\xi^T a \nabla\xi\big)}{(\det a) \det(\nabla\xi^T\nabla\xi)}\Big)\Big]
    d\nu_g\,,
  \end{split}
  \label{mu1-mu2-g}
\end{align}
where again $Z$ denotes different normalization constants.
Applying Proposition~\ref{sde-inv-mu} to the two probability measures expressed in
(\ref{mu1-mu2-g}), we can conclude that both the dynamics $X_s$ in
(\ref{dynamics-1-submanifold}) and $Y_s$ in (\ref{dynamics-2-submanifold}) evolve on
the submanifold $\Sigma$, and their 
invariant probability measures are given by $\mu_1$ and $\mu_2$, respectively.
\end{proof}
\begin{remark}
  Under Assumptions~\ref{assump-1}--\ref{assump-2}, we can find a neighborhood $\mathcal{O}$ of $\Sigma$, such
  that $P$ can be extended to $\mathcal{O}$. Furthermore, the relations in
   (\ref{p-fact}) in Appendix~\ref{app-sec-manifold} are still satisfied in $\mathcal{O}$. Due to this fact, 
   in \eqref{dynamics-2-submanifold}--\eqref{dynamics-2-submanifold} we can talk about the derivatives of $P$ at states $x\in \Sigma$. 
   \label{rmk-1-near-omega}
\end{remark}
\begin{remark}
  \begin{enumerate}
    \item
      Notice that, similar to (\ref{overdamp-l}), the infinitesimal generator
      of $X_s$ in (\ref{dynamics-1-submanifold}) can be written as 
\begin{align}
  \mathcal{L} = \frac{e^{\beta U}}{\beta}  \frac{\partial}{\partial x_i}\Big(e^{-\beta
  U}(Pa)_{ij}\frac{\partial}{\partial x_j}\Big)\,. 
  \label{generator-l-of-ys}
\end{align}
      Using (\ref{mu1-mu2-g}) and (\ref{integrate-by-part-sigma}) in Appendix~\ref{app-sec-manifold}, we can also verify the integration by parts formula
      \begin{align}
	\int_{\Sigma} (\mathcal{L}f) f'\, d\mu_1 = 
	\int_{\Sigma} (\mathcal{L}f') f\, d\mu_1 = 
	- \frac{1}{\beta} \int_{\Sigma} (Pa \nabla f)\cdot \nabla f' \, d\mu_1\,, 
	\label{integration-by-part-mu1}
      \end{align}
      for any two $C^2$ smooth functions $f, f': \Sigma \rightarrow \mathbb{R}$.
\item
   Using Jacobi's formula \cite{matrix-book} $\frac{\partial
      \ln\det (\nabla\xi^T\nabla\xi)}{\partial x_j} =
      (\nabla\xi^T\nabla\xi)^{-1}_{\alpha\eta}\frac{\partial
      (\nabla\xi_\alpha^T\nabla\xi_\eta)}{\partial x_j}$ and $(Pa)_{ij}\partial_j\xi_\alpha = 0$, 
      the equation (\ref{dynamics-2-submanifold}) can be simplified as 
      \begin{align*}
  dY_s^i = & - (Pa)_{ij} 
	\frac{\partial U}{\partial x_j}\,ds + \frac{1}{\beta} Q_{jl}\frac{\partial
(Pa)_{ij}}{\partial x_l}\,ds + \sqrt{2\beta^{-1}}\, P_{j,i}\,
	dW_s^j\,,
      \end{align*}
      where the matrix $Q = \mbox{id} - \nabla\xi
      (\nabla\xi^T\nabla\xi)^{-1}\nabla\xi^T$. 
  In the special case when $g=a = \mbox{\normalfont{id}}$, we have $\nu_g = \nu$
      and $P_{j, i} = P_{ji}=Q_{ji}$ from (\ref{p-i-j}). 
      Accordingly, 
  we can write the dynamics (\ref{dynamics-2-submanifold}) as
  \begin{align}
    \begin{split}
    dY_s^i =& 
- P_{ij} \frac{\partial U}{\partial x_j}\,ds 
+ \frac{1}{\beta} P_{lj}\frac{\partial P_{li}}{\partial x_j}\,ds +
\sqrt{2\beta^{-1}}\, P_{ji}\, dW_s^j\\
=& 
- P_{ij} \frac{\partial U}{\partial x_j}\,ds 
-\frac{1}{\beta} (\Psi^{-1})_{\alpha\gamma}  P_{lj} (\partial^2_{lj}\xi_\alpha) \partial_i\xi_\gamma\,ds
+ \sqrt{2\beta^{-1}}\, P_{ji}\, dW_s^j\,,\\
=& 
- P_{ij} \frac{\partial U}{\partial x_j}\,ds 
+ \frac{1}{\beta} H_i ds + \sqrt{2\beta^{-1}}\, P_{ji}\, dW_s^j\,,
\end{split}
    \label{dynamics-2-submanifold-id}
\end{align}
for $1 \le i \le d$, where $H=H_i\bm{e}_i$ is the mean curvature vector of
$\Sigma$ (see Proposition~\ref{prop-mean} in Appendix~\ref{app-sec-manifold}).
In Stratonovich form, (\ref{dynamics-2-submanifold-id}) can be written as 
  \begin{align}
    dY_s^i =& -P_{ij} \frac{\partial U}{\partial x_j} ds  
    +
  \sqrt{2\beta^{-1}} P_{ji} \circ\,dW_s^j\,, \hspace{0.2cm}  1 \le
i \le d\,.
\label{dynamics-1-submanifold-id-stratonovic}
\end{align}
      In this case, our results are accordant with those in~\cite{projection_diffusion}.
\end{enumerate}
\label{rmk-3}
\end{remark}

The dynamics constructed in Proposition~\ref{sde-inv-mu} and Theorem~\ref{thm-mu1-mu2} are
reversible on $\Sigma$, in the sense that their infinitesimal generators are
self-adjoint with respect to their invariant measures. In fact, using the same
idea, we can construct non-reversible ergodic SDEs on $\Sigma$ as well. 
We will only consider the conditional probability measure $\mu_1$, since it is
more relevant in applications and the result is also simpler.
\begin{corollary}
  Let $\mu_1$ be the conditional probability measure on $\Sigma$ defined in (\ref{mu-1-2}).
  The vector field $\bm{J}=(J_1, J_2, \cdots, J_d)^{T}=J_i\bm{e}_i$, defined
  on $x \in \Sigma$, satisfies 
  \begin{align}
    \begin{split}
    \bm{J}(x) \in T_x \Sigma, \qquad \forall~x \in \Sigma\,,\\
    P_{ij}\frac{\partial J_j}{\partial x_i} + J_{j} \frac{\partial P_{ij}}{\partial x_i} - \beta J_i \frac{\partial U}{\partial x_i} = 0\,.
    \end{split}
    \label{j-assumption}
  \end{align}
  Consider the dynamics $X_s$ on $\mathbb{R}^d$ which satisfies the Ito SDE 
\begin{align}
  \begin{split}
 dX_s^i 
= & J_i\,ds - (Pa)_{ij} 
\frac{\partial U}{\partial x_j}\,ds + \frac{1}{\beta} \frac{\partial
(Pa)_{ij}}{\partial x_j}\,ds + \sqrt{2\beta^{-1}}\, P_{j,i}\, dW_s^j\,,
  \end{split}
\label{dynamics-1-submanifold-j}
\end{align}
for $1 \le i \le d$, where $\beta > 0$ and $W_s = (W_s^1, W_s^2,\cdots, W_s^d)^T$ is a
$d$-dimensional Brownian motion.
  Suppose that $X_0 \in \Sigma$, then $X_s \in \Sigma$ almost surely for $s\ge 0$. 
Furthermore, the unique invariant probability measure of $X_s$ is $\mu_1$.
\label{corollary-on-non-reversible}
\end{corollary}
The proof can be found in Appendix~\ref{app-sec-manifold}.
\begin{remark}
  We make two remarks regarding the non-reversible vector $\bm{J}$. 
  \begin{enumerate}
    \item
Notice that, as tangent vectors acting on functions, we have $P\bm{e}_j =
  P_{ij}\frac{\partial}{\partial x_i} \in T_x \Sigma$. Therefore, the
  condition (\ref{j-assumption}) indeed only depends on the value of $\bm{J}$ on $\Sigma$.
  Supposing that $\bm{J}$ and $U$ are defined in a neighborhood $\mathcal{O}$
  of $\Sigma$ (see Remark~\ref{rmk-1-near-omega}), the
  condition (\ref{j-assumption}) can be written equivalently as 
  \begin{align}
    \begin{split}
      &\bm{J}(x) \in T_x \Sigma, \qquad \forall~x \in \Sigma\,,\\
      & \frac{\partial }{\partial x_i} \Big[(P_{ij}J_j)e^{-\beta U}\Big]=0
      \,,\qquad \forall~x~\mbox{near}~\Sigma\,.
    \end{split}
    \label{j-assumption-rmk}
  \end{align}
    \item
  Recall that, the non-reversible dynamics on $\mathbb{R}^d$ 
  \begin{equation}\label{dynamics-v-j}
\begin{aligned}
  d \widetilde{Y}^i_s & = \widetilde{J}_i
  (\widetilde{Y}_s)\,ds -\Big(a_{ij} \frac{\partial U}{\partial x_j}\Big)
  (\widetilde{Y}_s)\,ds +
  \frac{1}{\beta} \frac{\partial a_{ij}}{\partial x_j}(\widetilde{Y}_s)\, ds +
  \sqrt{2\beta^{-1}} \sigma_{ij}(\widetilde{Y}_s)\, dW^j_s\,, \quad 1 \le i \le d\,,
\end{aligned}
\end{equation}
has the invariant probability density $\frac{1}{Z} e^{-\beta U}$, if the
vector $\widetilde{\bm{J}}=(\widetilde{J}_1, \widetilde{J}_2, \cdots,
\widetilde{J}_d)^T$ satisfies 
\begin{align}
  \mbox{\textnormal{div}} (\widetilde{\bm{J}} e^{-\beta U}) = \frac{\partial
  \big(\widetilde{J}_i e^{-\beta U}\big)}{\partial x_i} =0 \,,\qquad \forall~x \in \mathbb{R}^d\,.
  \label{j-assumption-rmk-rn}
\end{align}
Comparing (\ref{j-assumption-rmk-rn}) with (\ref{j-assumption-rmk}), it is
clear that $\bm{J}=\widetilde{\bm{J}}|_{\Sigma}$ satisfies the condition
(\ref{j-assumption}) of Corollary~\ref{corollary-on-non-reversible} and can be used to construct non-reversible
SDEs on $\Sigma$, provided that $P\widetilde{\bm{J}} = \widetilde{\bm{J}}$ in the neighborhood
$\mathcal{O}$. Roughly speaking, in this case the vector field
$\widetilde{\bm{J}}$ is tangential to the level sets of $\xi$ in $\mathcal{O}$. In general cases, however, 
we can not simply take $\bm{J} = P\widetilde{\bm{J}}$ to obtain non-reversible
processes on $\Sigma$ which sample $\mu_1$, since (\ref{j-assumption-rmk})
may not be satisfied. 
We refer to Remark~\ref{rmk-conjecture-non-reversible-scheme} in Section~\ref{sec-numerical-scheme-mu1} for an
alternative idea to develop ``non-reversible'' numerical schemes.
\end{enumerate}
\label{rmk-non-reversible}
\end{remark}

\section{Numerical scheme sampling the conditional measure on $\Sigma$}
\label{sec-numerical-scheme-mu1}
Given a smooth function $f : \Sigma \rightarrow \mathbb{R}$ on
the level set $\Sigma$, in this section we study the numerical scheme 
(\ref{scheme-mu1-intro})--(\ref{phi-map-intro}) in the Introduction, which allows us to 
numerically compute the average 
\begin{align}
  \overline{f} = \int_{\Sigma} f(x)\,d\mu_1(x)
  \label{mean-f}
\end{align}
with respect to the conditional probability measure $\mu_1$ in (\ref{mu1-intro}).

To motivate the numerical scheme, let us first introduce the softly constrained dynamics, which satisfies the SDE
\begin{equation}\label{dynamics-v-soft}
\begin{aligned}
  d X^{\epsilon,i}_s & = \Big[- a_{ij} \frac{\partial U}{\partial
x_j}\, -
\frac{1}{\epsilon}a_{ij}\frac{\partial }{\partial x_j}\big(\frac{1}{2}\sum\limits_{\alpha=1}^k
  \xi_\alpha^2\big)\,+
\frac{1}{\beta} \frac{\partial a_{ij}}{\partial x_j}\Big](X^\epsilon_s)\, ds +
  \sqrt{2\beta^{-1}} \sigma_{ij}(X^\epsilon_s)\, dW^j_s\,, 
\end{aligned}
\end{equation}
where $\epsilon > 0$, $1 \le i \le d$. It is straightforward to
verify that (\ref{dynamics-v-soft}) has a unique invariant measure 
\begin{align}
  d\mu^\epsilon(x) = \frac{1}{Z^\epsilon} \exp\Big[-\beta\Big(U(x) +
    \frac{1}{2\epsilon} 
  \sum\limits_{\alpha=1}^k \xi_\alpha^2(x)\Big)\Big]\,dx\,,\quad \forall x \in
  \mathbb{R}^d\,,
  \label{mu-eps}
\end{align}
where $Z^\epsilon$ is the normalization constant. 
As $\epsilon \rightarrow 0$, 
the authors in \cite{projection_diffusion}
studied the convergence of the dynamics (\ref{dynamics-v-soft}) itself on a
finite time horizon in the case when $a=\sigma = \textnormal{id}$ and $k=1$.
Closely related problems have also been studied
in~\cite{fatkullin2010,katzenberger1991,funaki1993}. 
Since we are mainly interested in sampling the invariant measure, 
we record the following known convergence result of the measure
$\mu^\epsilon$ to $\mu_1$. We omit its proof since it is a standard
application of the co-area formula. 
\begin{lemma}
  Let $f : \mathbb{R}^d \rightarrow \mathbb{R}$ be a bounded smooth function. 
  $\mu_\epsilon$ is the probability measure in (\ref{mu-eps}) and $\mu_1$ is the
  conditional probability measures on $\Sigma$ defined in (\ref{mu-1-2}).
  We have 
  \begin{align*}
    \lim_{\epsilon \rightarrow 0} \int_{\mathbb{R}^d} f(x)\, d\mu^\epsilon(x)
    = \frac{1}{Z} \int_{\Sigma} f(x) e^{-\beta U(x)}
    \big[\mbox{\textnormal{det}}(\nabla\xi^T
    \nabla\xi)(x)\big]^{-\frac{1}{2}}\,d\nu(x) =
    \int_{\Sigma} f(x)\,d\mu_1(x)\,,
  \end{align*}
  where $Z$ is the normalization constant given by 
  \begin{align*}
    Z = \int_{\Sigma} e^{-\beta U(x)} \big[\mbox{\textnormal{det}}(\nabla\xi^T
    \nabla\xi)(x)\big]^{-\frac{1}{2}} d\nu(x)\,.
  \end{align*}
  \label{lemma-mu-eps-to-mu1}
\end{lemma}

Lemma~\ref{lemma-mu-eps-to-mu1} suggests that the softly constrained
dynamics (\ref{dynamics-v-soft}) with a small $\epsilon$ is a good candidate to sample $\mu_1$ on $\Sigma$. However, direct simulation of
(\ref{dynamics-v-soft}) is probably inefficient when $\epsilon$ is small, 
because the time step-size in numerical simulations becomes 
limited due to the strong stiffness in the dynamics. 
The numerical scheme we will study below can be viewed as a multiscale numerical method for the
dynamics (\ref{dynamics-v-soft}). To explain the method, let us introduce the flow map 
$\varphi : \mathbb{R}^d \times [0, +\infty) \rightarrow \mathbb{R}^d$, defined
  by
\begin{align}
  \begin{split}
    \frac{d\varphi(x,s)}{ds} =& - (a\nabla F) \big(\varphi(x,s)\big)\,,\quad
    \varphi(x,0) = x, \qquad  \forall~x\in \mathbb{R}^d\,, 
\end{split}
  \label{phi-map}
\end{align}
where the function $F$ is 
\begin{align}
  F(x)=\frac{1}{2}|\xi(x)|^2=\frac{1}{2}\sum_{\alpha=1}^{k}\xi_\alpha^2(x)\,.
\label{fun-cap-f}
\end{align}
Under proper conditions~\cite{katzenberger1991,fatkullin2010}, one can define the limiting map of
$\varphi$ as
\begin{align}
\Theta(x) = \lim\limits_{s\rightarrow +\infty}
  \varphi(x,s)\,, \qquad \forall~ x \in \mathbb{R}^d\,.
  \label{theta-map}
\end{align}
Since $\nabla F|_{\Sigma}=0$ and $\Sigma$ is the set consisting of all global minima of $F$, it is clear
that $\Theta: \mathbb{R}^d \rightarrow \Sigma$ and $\Theta(x) = x$, for $\forall~x \in \Sigma$.

With the map $\Theta$, 
we propose to approximate the average $\overline{f}$ in (\ref{mean-f}) by 
\begin{align}
  \widehat{f}_n = \frac{1}{n} \sum_{l=0}^{n-1} f(x^{(l)}) \,,
  \label{fn-by-trajectory-avearge}
\end{align}
where $n$ is a large number and the states $x^{(l)}$ are sampled from the numerical scheme
\begin{align}
  \begin{split}
    x^{(l + \frac{1}{2})}_i =& x^{(l)}_i + \Big(-a_{ij}\frac{\partial U}{\partial
x_j}   + \frac{1}{\beta}\frac{\partial a_{ij}}{\partial x_j}\Big)\, h
  + \sqrt{2 \beta^{-1}h}\, \sigma_{ij} \eta^{(l)}_j\,, \quad 1 \le i \le
  d \,,\\
  x^{(l+1)} =& \Theta\big(x^{(l+\frac{1}{2})}\big)\,,
\end{split}
\label{micro-scheme-initial-repeat}
\end{align}
starting from $x^{(0)} \in \Sigma$. In (\ref{micro-scheme-initial-repeat}),  
$h>0$ is the time step-size, functions $a, \sigma, U$ are evaluated at
$x^{(l)}$, and $\bm{\eta}^{(l)}=(\eta_1^{(l)}, \eta_2^{(l)}, \cdots,
\eta_d^{(l)})^T$ are independent $d$-dimensional standard Gaussian random
variables, for $0 \le l < n-1$. 
\begin{remark}
  We make two comments about the scheme (\ref{fn-by-trajectory-avearge})--(\ref{micro-scheme-initial-repeat}).
  \begin{enumerate}
    \item
  Since the image of $\Theta$ is on $\Sigma$, the discrete dynamics
  $x^{(l)}$ stays on $\Sigma$ all the time. As in the case of the softly constrained
  dynamics (\ref{dynamics-v-soft}), the numerical scheme has the advantage that only the $1$st order
  derivatives of $\xi$ are needed.
    \item
  When $a=\mbox{id}$, the numerical scheme (\ref{micro-scheme-initial-repeat}) becomes 
\begin{align}
  \begin{split}
    x^{(l + \frac{1}{2})} =& x^{(l)}-\nabla U(x^{(l)})\,h + \sqrt{2 \beta^{-1}h}\, \bm{\eta}^{(l)}\,, \\
  x^{(l+1)} =& \Theta\big(x^{(l+\frac{1}{2})}\big)\,.
\end{split}
\label{micro-scheme-a-id}
\end{align}
    \end{enumerate}
  \label{rmk-on-numerical-scheme-1}
\end{remark}
  At each step $l \ge 0$, one needs to
  compute $\Theta(x^{(l+\frac{1}{2})})$. This can be done by
  solving the ODE (\ref{phi-map}) starting from $x^{(l+\frac{1}{2})}$, using 
  numerical integration methods such as Runge-Kutta methods.
  In the following remark, we discuss issues associated with the computation of the ODE flow map $\Theta$.
  \begin{remark}[Computation of the flow map $\Theta$]
    \begin{enumerate}
    \item
 Exploiting the gradient structure of the ODE (\ref{phi-map}), we can in fact
 establish exponential convergence of the dynamics $\varphi$ to its limit
 $\Theta$, at least in the neighborhood $\mathcal{O}$ of $\Sigma$. For instance, we refer to 
 \cite{katzenberger1991} and \cite[Chapter $4$]{ambrosio2005gradient}. 
 Here, for brevity, we point out that the exponential decay of $F(\varphi(x,s))$ can
 be easily obtained and is therefore a good candidate for the convergence criterion
 in numerical implementations. Actually, under Assumption~\ref{assump-1}--\ref{assump-2}, we can suppose 
 $$z^T\Psi(x)z \ge c_0|z|^2, \quad \forall z \in \mathbb{R}^k\,,\quad
 \forall~x\in \mathcal{O},$$ for some $c_0 >0$, where $\Psi = \nabla\xi^T a \nabla \xi$.
 Direct calculation gives 
 \begin{align}
   \frac{d F(\varphi(x,s))}{ds} = -(\xi_\alpha
   \Psi_{\alpha\eta}\,\xi_\eta)(\varphi(x,s)) \le
   -c_0|\xi(\varphi(x,s))|^2 = -2c_0 F(\varphi(x,s))\,,
   \label{inequality-of-F-decay}
 \end{align}
 which implies that $|\xi(\varphi(x,s))|^2= 2F(\varphi(x,s)) \le e^{-2c_0 s} |\xi(x)|^2$. 
 In practice, suppose that we choose the condition $|\xi(\varphi(x,s))| \le \epsilon_{tol}$ as the stop
	criterion of ODE solvers and set $\Theta(x) = \varphi(x,s_{ode})$ when the
	condition is met at the time $s_{ode}$.
 Then the above analysis indicates that we need to integrate the ODE (\ref{phi-map})
	until the time $s_{ode}=\max\big\{\frac{1}{c_0} \big(\ln |\xi(x^{(l+\frac{1}{2})})| +
 \ln \frac{1}{\epsilon_{tol}}\big),0\big\}$,
 which grows logarithmically as $\epsilon_{tol}\rightarrow 0$.
  Since $x^{(l+\frac{1}{2})}$ is likely to remain close to $\Sigma$ when $h$ is small, we
  can expect that $\Theta(x^{(l+\frac{1}{2})})$ can be computed up to
	sufficient accuracy with affordable numerical effort. 
      \item
  As a complement of the discussion above, 
	we point out that adaptivity techniques (e.g., using adaptive step-sizes) can be used to
	accelerate the computation of the flow map $\Theta$.
	For instance, instead of (\ref{phi-map}), we can consider the ODE 
\begin{align}
  \begin{split}
    \frac{d\bar{\varphi}(x,s)}{ds} =& - (a\nabla |\xi|^{2-\kappa}) \big(\bar{\varphi}(x,s)\big)\,,\quad
    \bar{\varphi}(x,0) = x,
\end{split}
  \label{phi-map-rescale}
\end{align}
where $0 \le \kappa < 1$. In fact, from the identity 
	\begin{align}
	  \nabla |\xi|^{2-\kappa} = (2-\kappa) |\xi|^{-\kappa} \nabla
	  \frac{|\xi|^2}{2} = (2-\kappa) \sum_{\alpha=1}^k
	  \frac{\xi_\alpha}{|\xi|^{\kappa}} \nabla \xi_\alpha, 
	  \label{rescale-identity}
	\end{align}
	we know that ODE (\ref{phi-map-rescale}) is related to ODE (\ref{phi-map})
	by a rescaling of the time $s$. Accordingly, for each $x$, the
	solution $\bar{\varphi}(x,\cdot)$ coincides
	with $\varphi(x,\cdot)$ after a reparametrization and therefore can be
	used to compute the projection $\Theta(x)$ as well. Furthermore, 
	similar to (\ref{inequality-of-F-decay}), in this case we have 
	\begin{align*}
	  \frac{d F(\bar{\varphi}(x,s))}{ds} = -(2-\kappa)\big[|\xi|^{-\kappa} (\xi_\alpha
	  \Psi_{\alpha\eta}\,\xi_\eta)\big](\bar{\varphi}(x,s)) \le
	  -c_0(2-\kappa) \big[2F(\bar{\varphi}(x,s))\big]^{1-\kappa/2}\,,
	\end{align*}
	from which we obtain $|\xi(\bar{\varphi}(x,s))|^{\kappa} \le |\xi(x)|^\kappa
	- 2c_0(2-\kappa)s$, and therefore $\bar{\varphi}(x,s)$ reaches the state
	$\Theta(x) \in \Sigma$ before the finite time $s_{ode} =
	\frac{|\xi(x)|^\kappa}{2c_0(2-\kappa)}$.

In applications, $\Theta(x)$ can be computed by solving the ODE (\ref{phi-map-rescale})
	with a proper $\kappa \in [0, 1)$ (and decreasing step-sizes).
	From the above discussion, in particular the identity
	(\ref{rescale-identity}), we know that 
	this is equivalent to solving the ODE
	(\ref{phi-map}) using adaptive step-sizes. 
	We refer to Examples $1$--$2$ in Section~\ref{sec-example}
	for numerical validation.
	\end{enumerate}
  \label{rmk-on-numerical-scheme-about-ode}
\end{remark}

Our main result of this section concerns the approximation quality of the
mean value $\overline{f}$ by the running average $\widehat{f}_n$ in (\ref{fn-by-trajectory-avearge}), in the case when $h$ is small and $n$ is large. 
For this purpose, it is necessary to study the properties of the limiting
flow map $\Theta$, since it is involved in 
the numerical scheme (\ref{micro-scheme-initial-repeat}).
In fact, we have the following important result, which characterizes the
derivatives of $\Theta$ by the projection map $P$ in \eqref{p-ij-brief-def}.
(We refer the reader to \eqref{p-ij}--\eqref{p-fact} in Appendix~\ref{app-sec-manifold} for properties of $P$.)
    \begin{prop}
      Let $\Theta$ be the limiting flow map in (\ref{theta-map}) and $P$ be the projection map in (\ref{p-ij-brief-def}). At each $x \in \Sigma$, we have 
    \begin{align}
      \begin{split}
      \frac{\partial \Theta_{i}}{\partial x_j} =& P_{ij}\,, \\
      a_{lr}\frac{\partial^2 \Theta_{i}}{\partial x_l\partial x_r} 
      =& \frac{\partial (Pa)_{il}}{\partial x_{l}} - 
      P_{il} \frac{\partial a_{lr}}{\partial x_{r}}\,,  
	\label{1st-2nd-d-theta}
      \end{split}
    \end{align}
for $~1 \le i,\,j \le d$. 
    \label{prop-map-phi-1st-2nd-derivative}
    \end{prop}
    The proof of Proposition~\ref{prop-map-phi-1st-2nd-derivative} can be found in Appendix~\ref{app-sec-numerical}.

Based on the above result, we are ready to quantify the approximation error between the
estimator $\widehat{f}_n$ and the mean value $\overline{f}$.
\begin{thm}
  Suppose that both the step-size $h$ and the number of the total steps $n$ are fixed.
  Assume that $f : \Sigma \rightarrow \mathbb{R}$ is a smooth function on
  $\Sigma$ and $\overline{f}$ is its mean value defined in (\ref{mean-f}) with respect to the measure $\mu_1$.
  Consider the running average $\widehat{f}_n$ in (\ref{fn-by-trajectory-avearge}),
  which is computed by simulating the numerical scheme (\ref{micro-scheme-initial-repeat}) with time step-size
  $h>0$. Let $T=nh$ and $C$ denote a generic positive constant that is independent of
  $h$, $n$. We have the following approximation results. 
  \begin{enumerate}
  \item
    $\big|\mathbf{E}\widehat{f}_n - \overline{f}\big| \le C(h+ \frac{1}{T})$.
  \item
    $\mathbf{E}\big|\widehat{f}_n - \overline{f}\big|^2 \le C(h^2 + \frac{1}{T})$.
  \item
    For any $0 < \epsilon < \frac{1}{2}$, there is an almost surely bounded positive random variable $\zeta(\omega)$, such that
      $|\widehat{f}_n - \overline{f}| \le Ch+ \frac{\zeta(\omega)}{T^{1/2-\epsilon}}$\,, almost surely.
  \end{enumerate}
  \label{thm-estimate-scheme-on-submanifold}
\end{thm}
We present the proof of Theorem~\ref{thm-estimate-scheme-on-submanifold} in
Appendix~\ref{app-sec-numerical}, since it is technical and the idea follows
the standard approach developed in~\cite{conv-time-averaging}, where Poisson
equation played a crucial role. However, let us emphasize that, in contrast to
\cite{conv-time-averaging}, in the current setting we are working on the
submanifold $\Sigma$ and, furthermore, the map $\Theta$ is involved in our
numerical scheme. In particular, in the proof we use the Poisson equation
related to the generator $\mathcal{L}$ in (\ref{generator-l-mu1-intro}) of the process \eqref{dynamics-1-submanifold-intro},
based on the fact that $\mu_1$ is the invariant measure of
\eqref{dynamics-1-submanifold-intro} (This is the place where
Theorem~\ref{thm-mu1-mu2} in Section~\ref{sec-n-rn} is used in order
to establish Theorem~\ref{thm-estimate-scheme-on-submanifold}).
\begin{remark}
  Theorem~\ref{thm-estimate-scheme-on-submanifold} concerns the long time
  behavior of the scheme (\ref{micro-scheme-initial-repeat}) with a small
  step-size, i.e., large $T$ and small $h$. This is often relevant in molecular dynamics simulations.
While the estimates of Theorem~\ref{thm-estimate-scheme-on-submanifold} are
  stated in terms of the variables $h$
  and $T$, we should point out that the time $T=nh$ and therefore it depends on both $h$ and $n$.
  Alternatively (and more precisely), the estimates can be expressed using the independent
  variables $h$ and $n$. For instance, for the mean square error estimate, we have 
  \begin{align}
    \mathbf{E}\big|\widehat{f}_n - \overline{f}\big|^2 \le C\Big(h^2 + \frac{1}{nh}\Big)\,.
    \label{thm-estimate-scheme-by-h-n}
  \end{align}
  Therefore, for a fixed (large) total sample number $n$,
  we can conclude that the optimal upper bound in (\ref{thm-estimate-scheme-by-h-n})
  is $\mathcal{O}(n^{-\frac{2}{3}})$ and is achieved when $h=\mathcal{O}(n^{-\frac{1}{3}})$.
  We refer to~\cite{conv-time-averaging} for related discussions.
  \label{explain-the-estimate-of-thm}
\end{remark}
In applications, the conditional probability measure $\mu_1$ often satisfies the
following Poincar{\'e} inequality~\cite{zhang-lelievre-pathwise2018}
       \begin{align}
	 \mbox{Var}_{\mu_1}(f) := \int_{\Sigma} (f-\overline{f})^2\,d\mu_1 
	 \le -\frac{1}{K} \int_{\Sigma} (\mathcal{L}f) f\,d\mu_1 =
	 \frac{1}{K\beta} \int_{\Sigma} (Pa \nabla f)\cdot \nabla f\,d\mu_1 \,,
	 \label{poincare-inequality-mu1}
       \end{align}
       for all $\,f : \Sigma\rightarrow \mathbb{R}$ such that the right hand
       side of the above inequality is finite, where $K>0$ is the  
Poincar{\'e} constant, $\mathcal{L}$ is the infinitesimal generator 
(\ref{generator-l-mu1-intro}), and the identity (\ref{integration-by-part-mu1}) in
Remark~\ref{rmk-3} has been used. Under this condition, the mean square error
estimate in Theorem~\ref{thm-estimate-scheme-on-submanifold} can be 
improved (i.e., the constant in front of the $\mathcal{O}(T^{-1})$ term 
is small when $K$ is large) and we have the following corollary (The proof is
in Appendix~\ref{app-sec-numerical}).
\begin{corollary}
  Under the same assumptions in
  Theorem~\ref{thm-estimate-scheme-on-submanifold} and further assuming that 
  $\mu_1$ satisfies the Poincar{\'e} inequality
  (\ref{poincare-inequality-mu1}), we have 
  \begin{align*}
    \mathbf{E}\big|\widehat{f}_n - \overline{f}\big|^2 \le  
    \frac{2C_1 \mbox{\textnormal{Var}}_{\mu_1}(f)}{KT} + C_2\Big(h^2 +
    \frac{h}{T}+\frac{1}{T^2}\Big)\,,
  \end{align*}
  where $C_1$ is any constant larger than $1$, the constant $C_2$ 
  depends on the choice of $C_1$ but is independent of both $h$ and $n$.
  \label{corollary-on-spectral-gap}
\end{corollary}
\begin{remark}[Non-reversible schemes]
  The idea of using the map $\Theta$ in the constraint step of the numerical
  scheme (\ref{micro-scheme-initial-repeat}) is motivated by the softly
  constrained (reversible) dynamics (\ref{dynamics-v-soft}). It is natural to
  consider whether certain ``non-reversible'' numerical scheme can
  be obtained using the same idea. In fact, let $A\in \mathbb{R}^{d\times
  d}$ be a constant skew-symmetric matrix such that $A^T=-A$. The softly constrained
  (non-reversible) dynamics 
  \begin{equation}\label{dynamics-v-soft-non-reversible}
\begin{aligned}
  d X^{\epsilon,A, i}_s & = \bigg[A_{ij} \frac{\partial}{\partial
x_j}\Big(U+\frac{1}{2\epsilon}\sum\limits_{\alpha=1}^k\xi_\alpha^2\Big)\,- a_{ij} \frac{\partial U}{\partial
x_j}\, -
\frac{1}{\epsilon}a_{ij}\frac{\partial }{\partial x_j}\Big(\frac{1}{2}\sum\limits_{\alpha=1}^k
  \xi_\alpha^2\Big)\,+
\frac{1}{\beta} \frac{\partial a_{ij}}{\partial x_j}\bigg](X^{\epsilon,A}_s)\, ds
\\
& + \sqrt{2\beta^{-1}} \sigma_{ij}(X^{\epsilon,A}_s)\, dW^j_s\,, 
\end{aligned}
\end{equation}
indeed has the same invariant measure $\mu^{\epsilon}$ in (\ref{mu-eps}).
Based on this fact, a reasonable guess of the ``non-reversible'' numerical
scheme that samples the conditional measure $\mu_1$ is 
the multiscale method of (\ref{dynamics-v-soft-non-reversible}), i.e.,
\begin{align}
  \begin{split}
    x^{(l + \frac{1}{2})}_i =& x^{(l)}_i + \Big(A_{ij}\frac{\partial U}{\partial x_j}  
    -a_{ij} \frac{\partial U}{\partial x_j}   + \frac{1}{\beta}\frac{\partial
    a_{ij}}{\partial x_j}\Big)(x^{(l)})\, h
  + \sqrt{2 \beta^{-1}h}\, \sigma_{ij}(x^{(l)})\, \eta^{(l)}_j\,, \quad 1 \le i \le
  d \,,\\
  x^{(l+1)} =& \Theta^{A}\big(x^{(l+\frac{1}{2})}\big)\,,
\end{split}
\label{scheme-non-reversible}
\end{align}
where $\Theta^{A}(x)=\lim\limits_{s\rightarrow +\infty} \varphi^{A}(x,s)$ is the limit of the (non-gradient) flow map 
\begin{align}
  \begin{split}
    \frac{d\varphi^{A}(x,s)}{ds} =& -((a-A)\nabla F\big)\big(\varphi^{A}(x,s)\big)\,,\quad
    \varphi^{A}(x,0) = x, \qquad  \forall~x\in \mathbb{R}^d\,, 
\end{split}
  \label{phi-map-a}
\end{align}
with the same function $F$ in (\ref{fun-cap-f}).
We expect that the long time sampling error estimates of the numerical scheme
(\ref{scheme-non-reversible}) can be studied following the same approach of this
section as well.
For this purpose, however, it is necessary to handle the non-gradient term in the ODE (\ref{phi-map-a}),
which brings difficulties when calculating the derivatives of the map
$\Theta^{A}$ (cf. Proposition~\ref{prop-map-phi-1st-2nd-derivative} as well as
its proof in Appendix~\ref{app-sec-numerical}). We will postpone the analysis in the future work
and readers are referred to Example $1$ in Section~\ref{sec-example} for numerical
validation of the scheme (\ref{scheme-non-reversible}--(\ref{phi-map-a}).
   \label{rmk-conjecture-non-reversible-scheme}
\end{remark}

In the literature, reversible Metropolis random walk on submanifold~\cite{goodman-submanifold} and Hybrid Monte Carlo algorithm~\cite{tony-gabriel-hmc-submanifolds} 
have been proposed to sample distributions on submanifolds. A comprehensive
comparison between our scheme (without Metropolis step) and these Metropolis-adjusted
approaches is a complicate task and goes beyond the scope of the current paper. 
We only discuss this issue briefly in the following remark. 
\begin{remark}
  [Comparison to the Metropolis-adjusted samplers~\cite{goodman-submanifold,tony-gabriel-hmc-submanifolds} on submanifolds]
  We compare the following three aspects.  
  \begin{enumerate}
    \item
      \underline{Constraint step}.  In the scheme
      (\ref{micro-scheme-initial-repeat}), the map $\Theta$ in (\ref{theta-map}) is
      used to project the state $x^{(l+\frac{1}{2})}$ to the submanifold $\Sigma$.
      Under mild assumptions, the gradient structure of the ODE (\ref{phi-map})
allows to define $\Theta$ for all states at which the map $\xi$ is $C^2$
      smooth (Assumption~\ref{assump-2}).
      Differently, in~\cite{goodman-submanifold,tony-gabriel-hmc-submanifolds}, newly proposed states
      are projected back to the submanifold by solving a nonlinear system. 
      One usually uses Newton's method to find the solution of the
      system, with the hope that the convergence can be achieved within a few
      iteration steps (success), thanks to the quadratic convergence rate of Newton's method.
      In practice, however, it may happen that either the solution does not
      exist or, even if the solution exists, the Newton's method does not
      converge (due to local convergence). In these two cases, the constraint step
      ends without finding a new state on the submanifold (no success). Although the Markov chain samples the correct invariant distribution regardless whether the constraint step is
      successful or not~\cite{goodman-submanifold}, the sampling efficiency is
      affected by the success rate of the constraint step. 
    \item \underline{Computational complexity}.
      Suppose the computational complexity of evaluating the $d\times k$ matrix
      $\nabla\xi$ is $\mathcal{O}(k\cdot d)$ and $n$ states are sampled in total.
      In each step of both the scheme (\ref{micro-scheme-initial-repeat}) and the Metropolis-adjusted
      methods in~\cite{goodman-submanifold,tony-gabriel-hmc-submanifolds}, the
      major computational effort is devoted to the constraint step, i.e.,
      either computing the map $\Theta$ by integrating ODE or solving equations using Newton's method. 
      For the scheme (\ref{micro-scheme-initial-repeat}) with $a=\mbox{id}$, the overall
      computational complexity of the constraint step is therefore $\mathcal{O}(n
      \cdot k \cdot d\cdot s_{ode}/\Delta s)$, where $s_{ode}$  and $\Delta s$
      are the average final time (see
      Remark~\ref{rmk-on-numerical-scheme-about-ode}) and the average
      step-size in the ODE integration, respectively. 
      For the Metropolis-adjusted methods in \cite{goodman-submanifold,tony-gabriel-hmc-submanifolds},
in each Newton iteration it is necessary to compute the matrix product $\nabla\xi^T(x)\nabla\xi(x')$ for two states $x,x'$. 
      Therefore, the overall computational complexity is $\mathcal{O}(n\cdot
      k^2\cdot d\cdot N_{iter})$, where $N_{iter}$ is the average total
      iteration steps of Newton's method.
      In practice, one can implement the method in~\cite{goodman-submanifold} in a way such that Newton's method 
      ends within a few Newton iterations, e.g., $N_{iter} \le 10$.  On the
      other hand, the ODE integration in the scheme (\ref{micro-scheme-initial-repeat}) requires more
      iteration steps (for the examples in Section~\ref{sec-example}, $20-40$ steps are needed), i.e., $s_{ode}/\Delta s \ge N_{iter}$.
      However, when comparing the computational cost of both methods, we
      should keep in mind that an ODE iteration step is generally
      cheaper than a Newton iteration step, since the latter involves 
      both matrix-matrix multiplication and solving linear systems. (The
      cost of solving $k\times k$ linear systems is not major and therefore for
      simplicity is not included in the estimation above.)
      While the computational cost of Newton's method is smaller when $k$ is
      small, the ODE integration becomes faster for medium or large $k$.
      We refer to Example $2$ in Section~\ref{sec-example} for
      numerical comparison on the computational cost of both methods.
    \item \underline{Choice of step-size}.
      To apply the scheme (\ref{micro-scheme-initial-repeat}), one usually
      chooses a suitably small step-size $h$ and runs the scheme for sufficient many steps (large $n$).
      In concrete applications, one often needs to tune the step-size $h$,
      keeping in mind that a large $h$ will lead to large bias, while a unnecessarily small
      $h$ will result in large correlations. See Remark~\ref{explain-the-estimate-of-thm}. 
      For the Metropolis-adjusted
      methods~\cite{goodman-submanifold,tony-gabriel-hmc-submanifolds}, the
      choice of the step-size (in the proposal step) is in fact a more
      delicate issue. Although the Markov chain remains unbiased for large
      step-sizes, the sampling efficiency will be possibly limited due to a low
      acceptance rate in the Metropolis step. This issue has been discussed
      in~\cite{tony-gabriel-hmc-submanifolds}, where the performance 
      with different step-sizes has been numerically investigated.
      Besides the acceptance rate in the Metropolis step, the success rate of the 
      constraint step also depends on the step-size used in the proposal step.
      Taking the $(d-1)$-dimensional unit sphere $\mathbb{S}^{d-1}$ as an
      example, it is not difficult to see that there won't be corresponding
      projected state on the sphere (i.e., the solution of the constraint
      equation does not exist), if
      the norm of the tangent vector generated in the proposal step
      (see~\cite{goodman-submanifold}) is larger than one.
      This implies that the success rate of the constraint step will decrease
      when we increase the step-size in the proposal step.
      To summarize, it is important to choose the step-size in the
      Metropolis-adjusted method in~\cite{goodman-submanifold,tony-gabriel-hmc-submanifolds} properly, such that both the acceptance rate in the
      Metropolis step and the success rate of the constraint step are not too small.
  \end{enumerate}
  \label{rmk-cmp-metropolis-or-not}
\end{remark}

Before concluding, let us point out that the approach used in the above proof
allows us to study other numerical schemes on $\Sigma$ as well. As an example,
we consider the projection from $\mathbb{R}^d$ to $\Sigma$ along geodesic
curves (instead of using the flow map \eqref{phi-map}--\eqref{theta-map}) defined by the metric $g=a^{-1}$ in (\ref{ip-a-intro}), i.e., the metric on $\mathcal{M}=(\mathbb{R}^d, g)$.
Let $\mathrm{d}$ be the distance function on $\mathbb{R}^d$ induced by the
metric $g$ in (\ref{ip-a-intro}). We introduce the projection function
\begin{align}
  \Pi(x) = \Big\{y\,\Big|\, \mathrm{d}(x,y) = \mathrm{d}(x,\Sigma),\, y \in \Sigma
  \Big\}\,,\quad \forall~x \in \mathbb{R}^d\,.
  \label{projection-map-pi}
\end{align}
Clearly, we have $\Pi|_{\Sigma}=\mbox{\textnormal{id}}|_\Sigma$. Given any $x \in \Sigma$, there is a
neighborhood $\Omega \subset \mathbb{R}^d$ of $x$ such that $\Pi|_\Omega$ is a single-valued map.
Furthermore, applying 
inverse function theorem, we can verify that $\Pi$ is smooth on $\Omega$.
Similar to Proposition~\ref{prop-map-phi-1st-2nd-derivative}, we need the following result which connects the derivatives of $\Pi$ to the
projection map $P$ in (\ref{p-ij-brief-def}). (Note that, comparing to the
derivatives of the map $\Theta$ in \eqref{1st-2nd-d-theta}, there is an extra term in the second equation of
\eqref{1st-2nd-d-pi}.)  Its proof is given in Appendix~\ref{app-sec-numerical}.
\begin{prop}
  Let $\Pi = (\Pi_1, \Pi_2, \cdots, \Pi_d)^T: \mathbb{R}^d \rightarrow \Sigma$ be the projection function in
  (\ref{projection-map-pi}), where $\Pi_i : \mathbb{R}^d \rightarrow
 \mathbb{R}$ are smooth functions, $1\le i \le d$. 
 For $x \in \Sigma\cap \Omega$, we have 
 \begin{align}
   \begin{split}
   & \frac{\partial \Pi_i}{\partial x_j} = P_{ij}\,, \\
   &a_{lr} \frac{\partial^2 \Pi_i}{\partial x_l\partial x_r}  
     = 
      - P_{il} \frac{\partial a_{lr}}{\partial x_r}
      + \frac{\partial (Pa)_{il}}{\partial x_l}
      +
      \frac{1}{2} (Pa)_{il} \frac{\partial \ln \mbox{\textnormal{det}}\Psi}{\partial
	x_l}\,,
   \end{split}
   \label{1st-2nd-d-pi}
 \end{align}
for $1 \le i,j \le d$, where $\Psi = \nabla\xi^T a \nabla \xi$.
 \label{prop-map-pi}
\end{prop}
Now we are ready to study the numerical scheme 
\begin{align}
  \begin{split}
    x^{(l + \frac{1}{2})}_i =& x^{(l)}_i + \Big(-a_{ij}\frac{\partial U}{\partial
  x_j}   + \frac{1}{\beta}\frac{\partial a_{ij}}{\partial x_j}\Big)(x^{(l)})\, h
  + \sqrt{2 \beta^{-1}h}\, \sigma_{ij}(x^{(l)})\, \eta^{(l)}_j\,, \quad 1 \le i \le
  d \,,\\
  x^{(l+1)} =& \Pi\big(x^{(l+\frac{1}{2})}\big)\,,
\end{split}
\label{micro-scheme-pi}
\end{align}
where $x^{(0)} \in \Sigma$, and the map $\Pi$ (instead of $\Theta$) is used in each step to project the states $x^{(l+\frac{1}{2})}$ back to $\Sigma$.
\begin{thm}
  Assume that $f : \Sigma \rightarrow \mathbb{R}$ is a smooth function on
  $\Sigma$ and $\overline{\overline{f}}$ is its mean value 
  \begin{align*}
    \overline{\overline{f}} = \int_{\Sigma} f(x)\,d\mu(x)\,,
  \end{align*}
  with respect to the probability measure 
  \begin{align} 
    d\mu=\frac{1}{Z} e^{-\beta U} \sqrt{\frac{\det (\nabla\xi^T a
    \nabla\xi)}{\det (\nabla\xi^T \nabla \xi)}}\, d\nu\,.
    \label{mu3-pi}
\end{align}
  Consider the running average $\widehat{f}_n$ in (\ref{fn-by-trajectory-avearge}),
  which is computed by simulating the numerical scheme (\ref{micro-scheme-pi}) with time step-size
  $h>0$. Let $T=nh$ and $C$ denote a generic positive constant that is independent of
  $h$, $T$. We have the following approximation results. 
  \begin{enumerate}
  \item
    $\big|\mathbf{E}\widehat{f}_n - \overline{\overline{f}}\big| \le C(h+ \frac{1}{T})$.
  \item
    $\mathbf{E}\big|\widehat{f}_n - \overline{\overline{f}}\big|^2 \le C(h^2 + \frac{1}{T})$.
  \item
    For any $0 < \epsilon < \frac{1}{2}$, there is an almost surely bounded positive random variable $\zeta(\omega)$, such that
    $|\widehat{f}_n - \overline{\overline{f}}| \le Ch+ \frac{\zeta(\omega)}{T^{1/2-\epsilon}}$\,, almost surely.
  \end{enumerate}
  \label{thm-estimate-scheme-on-submanifold-pi}
\end{thm}
We omit the proof since it resembles the proof of Theorem~\ref{thm-estimate-scheme-on-submanifold}.
\begin{remark}
  For the projection map $\Pi$ induced by a general metric $g=a^{-1}$
  or, equivalently, by a general (positive definite) matrix $a$,
  implementing the numerical scheme (\ref{micro-scheme-pi}) is not as easy as
  the numerical scheme (\ref{micro-scheme-initial-repeat}). We decide to omit
  the algorithmic discussions, due to the fact that the probability measure
  (\ref{mu3-pi}) seems less relevant in applications.
  However, it is meaningful to point out that, when $a=\mbox{id}$, the above result is relevant to the one in~\cite{projection_diffusion}. In this case, the probability
  measure in (\ref{mu3-pi}) reduces to $\mu_2=\frac{1}{Z} e^{-\beta U}d\nu$ in (\ref{mu-1-2}) and the numerical scheme
  (\ref{micro-scheme-pi}) can be formulated equivalently using
  Lagrange multiplier. We refer to~\cite{projection_diffusion,Tony-constrained-langevin2012} for comprehensive numerical details. 
\end{remark}
\section{Numerical examples}
\label{sec-example}
In this section, we study three concrete examples. In the first example, we investigate the different schemes in Section~\ref{sec-numerical-scheme-mu1}. In particular, the sampling performance of the constrained schemes using different maps $\Theta$, $\Theta^A$, and $\Pi$, as well as the performance of the unconstrained Euler-Maruyama discretization of the SDE (\ref{dynamics-1-submanifold}), will be compared.
In the second example, we compare the computational costs of the scheme (\ref{micro-scheme-initial-repeat})
and the Metropolis-adjusted algorithm introduced in~\cite{goodman-submanifold}.
In the last example, we show that in some cases it is helpful to consider non-constant matrices $\sigma$ and $a$.
The C/C++ code used for producing the numerical results in the following examples is available at: \url{https://github.com/zwpku/sampling-on-levelset}.
\subsubsection*{Example 1: Comparison of schemes using different projection maps}
Let us define $\xi: \mathbb{R}^2\rightarrow \mathbb{R}$ by 
\begin{align*}
  \xi(x) = \frac{1}{2}\Big(\frac{x_1^2}{c^2}+x_2^2 -1\Big)\,, \quad \forall~x = (x_1, x_2)^T \in \mathbb{R}^2\,,
\end{align*}
with the constant $c=3$. The level set $\Sigma=\big\{(x_1,
x_2)^T~|~\frac{x_1^2}{c^2}+x_2^2=1\big\}$ is an ellipse in $\mathbb{R}^2$.
We have $\nabla\xi = (\frac{x_1}{c^2},x_2)^T$ and therefore
$\det(\nabla\xi^T\nabla\xi) = |\nabla\xi|^2 = \frac{x_1^2}{c^4} + x_2^2$.
For simplicity, we choose the potential $U=0$ and the matrices
$a=\sigma=\mbox{id}\in\mathbb{R}^{2\times 2}$.
The two probability measures in (\ref{mu-1-2}) on $\Sigma$ are 
\begin{align*}
  d\mu_1 = \frac{1}{Z}\,\Big(\frac{x_1^2}{c^4} + x_2^2\Big)^{-\frac{1}{2}} d\nu\,,\quad 
  d\mu_2 = \frac{1}{Z}\, d\nu\,,
\end{align*}
where $Z$ denotes two different normalization constants and $\nu$ is the
surface measure on $\Sigma$. Since $\Sigma$ is a one-dimensional manifold, it
is helpful to consider the parametrization of $\Sigma$ by
\begin{align}
  x_1 = c\cos\theta\,,\quad x_2 = \sin\theta\,,
  \label{ex1-sigma-by-angle-theta}
\end{align}
where the angle $\theta \in [0, 2\pi]$.
 Applying the chain rule $\frac{\partial}{\partial \theta}
= -c\sin\theta\frac{\partial}{\partial x_1} + \cos\theta \frac{\partial}{\partial
x_2}$, we can obtain the expressions of $\mu_1$, $\mu_2$  under this coordinate as 
\begin{align}
  d\mu_1 = \frac{1}{Z} d\theta\,,\quad d\mu_2 = \frac{1}{Z} \big(c^2\sin^2\theta +
  \cos^2\theta\big)^{\frac{1}{2}}\,d\theta\,.
  \label{ex1-mu1-mu2-theta}
\end{align}
With these preparations, we proceed to study the following four numerical approaches.
\begin{enumerate}
  \item
    \underline{Numerical scheme (\ref{micro-scheme-initial-repeat}) using $\Theta$}. Since $U\equiv 0$ and $a=\mbox{id}$, 
(\ref{micro-scheme-initial-repeat}) becomes  
\begin{align}
  \begin{split}
    x^{(l + \frac{1}{2})} =& x^{(l)} + \sqrt{2 \beta^{-1}h}\,  \bm{\eta}^{(l)}\,, \\
  x^{(l+1)} =& \Theta\big(x^{(l+\frac{1}{2})}\big)\,,
\end{split}
\label{ex1-scheme-0}
\end{align}
where $\Theta(x)$ is the limit of the flow map $\varphi$, given by 
\begin{align}
  \dot{y}_1(s) = -\frac{\xi\big(y(s)\big)\,y_1(s)}{c^2}\,,\quad
  \dot{y}_2(s) = -\xi(y(s))\,y_2(s)\,, \qquad s \ge 0\,,
  \label{ex1-scheme-0-flow}
\end{align}
starting from $y(0) = x$.
  \item
    \underline{Numerical scheme (\ref{scheme-non-reversible})--(\ref{phi-map-a}) using $\Theta^{A}$}.
    Let us choose the skew-symmetric matrix  
    \begin{align}
      A=
    \begin{pmatrix}
      0 & 1/2 \\
      -1/2 & 0
    \end{pmatrix}.
      \label{ex1-scheme-1-matrix-a}
    \end{align}
    Since $U\equiv 0$ and $a=\mbox{id}$, we have 
\begin{align}
  \begin{split}
    x^{(l + \frac{1}{2})} =& x^{(l)} + \sqrt{2 \beta^{-1}h}\,  \bm{\eta}^{(l)}\,, \\
  x^{(l+1)} =& \Theta^A\big(x^{(l+\frac{1}{2})}\big)\,,
\end{split}
\label{ex1-scheme-0-a}
\end{align}
where $\Theta^A(x)$ is the limit of the flow map $\varphi^A$, given by
\begin{align}
  \dot{y}_1(s) = -\xi\big(y(s)\big)\,\Big(\frac{y_1(s)}{c^2} - \frac{y_2(s)}{2}\Big) \,,\quad
  \dot{y}_2(s) = -\xi\big(y(s)\big)\,\Big(\frac{y_1(s)}{2c^2} + y_2(s)\Big)\,, \qquad s \ge 0\,,
  \label{ex1-scheme-0-flow-a}
\end{align}
starting from $y(0) = x$.
  \item
    \underline{Numerical scheme (\ref{micro-scheme-pi}) using $\Pi$}.
Similarly, since $U\equiv 0$ and $a=\mbox{id}$, (\ref{micro-scheme-pi}) becomes  
\begin{align}
  \begin{split}
    x^{(l + \frac{1}{2})} =& x^{(l)} + \sqrt{2 \beta^{-1}h}\,  \bm{\eta}^{(l)}\,, \\
  x^{(l+1)} =& \Pi\big(x^{(l+\frac{1}{2})}\big)\,,
\end{split}
\label{ex1-scheme-1}
\end{align}
where $\Pi$ is the projection map onto $\Sigma$, defined in (\ref{projection-map-pi}). 
  \item
    \underline{Euler-Maruyama discretization of the SDE
    (\ref{dynamics-1-submanifold})}. 
    Notice that we have $Pa = P$, and it is straightforward to compute 
    \begin{align*}
      P_{11} = \frac{c^4x_2^2}{x_1^2+c^4x_2^2}\,,\quad 
      P_{12} = P_{21} = -\frac{c^2x_1x_2}{x_1^2+c^4x_2^2}\,,\quad 
      P_{22} = \frac{x_1^2}{x_1^2+c^4x_2^2}\,.
    \end{align*}
  Therefore, discretizing (\ref{dynamics-1-submanifold}), we obtain 
\begin{align}
  \begin{split}
    x^{(l + 1)}_1 =& x^{(l)}_1 + \frac{1}{\beta}
    \frac{c^4(c^2-2) x_1x_2^2-c^2x_1^3}{(x_1^2+c^4x_2^2)^2}\,h 
    + \sqrt{2 \beta^{-1}h}\, 
    \Big(\frac{c^4x_2^2}{x_1^2+c^4x_2^2}\,\eta^{(l)}_1
    -\frac{c^2x_1x_2}{x_1^2+c^4x_2^2}\,\eta^{(l)}_2\Big) \\
    x^{(l + 1)}_2 =& x^{(l)}_2 + \frac{1}{\beta}
    \frac{(1-2c^2)x_1^2x_2-c^6x_2^3}{(x_1^2+c^4x_2^2)^2}\,h 
    + \sqrt{2 \beta^{-1}h}\, 
    \Big(
    -\frac{c^2x_1x_2}{x_1^2+c^4x_2^2}\,\eta^{(l)}_1
    +\frac{x_1^2}{x_1^2+c^4x_2^2}\,\eta^{(l)}_2 \Big)\,.
\end{split}
\label{ex1-scheme-2}
\end{align}
\end{enumerate}

Based on Theorems~ \ref{thm-mu1-mu2}, \ref{thm-estimate-scheme-on-submanifold-pi},\, and
Remark~\ref{rmk-conjecture-non-reversible-scheme}, we
study the performance of the schemes (\ref{ex1-scheme-0}),  
(\ref{ex1-scheme-0-a}), and (\ref{ex1-scheme-2}) in sampling the conditional
measure $\mu_1$, as well as the performance of the scheme (\ref{ex1-scheme-1})
in sampling the measure $\mu_2$. 

In the numerical experiment, we choose $\beta=1.0$ in each of the above schemes.
For the first scheme using $\Theta$, we simulate (\ref{ex1-scheme-0}) 
for $n=2\times 10^{7}$ steps with the step-size $h=0.01$. In each constraint step, 
$\Theta(x^{(l+\frac{1}{2})})$ is computed by solving the ODE
(\ref{ex1-scheme-0-flow}) starting from $y(0) = x^{(l+\frac{1}{2})}$ until
the time when $|\xi(y(s))| < 10^{-8}$ is satisfied, using the $3$rd order (Bogacki-Shampine) Runge-Kutta
(RK) method.  
The adaptivity technique in the second point of Remark~\ref{rmk-on-numerical-scheme-about-ode} is used with $\kappa=0.5$. 
The step-size for solving the ODE is set to $\Delta s = 0.1$ initially and
is divided by $2.0$ whenever we find that the value of $|\xi|$ is not decreasing.
(The numerical error of $\Theta$ is $3.7 \times 10^{-7}$ on average, comparing to the
reference solution that is obtained by solving the ODE with $\kappa=0$ and
the fixed step-size $\Delta s= 0.001$.)
On average, we observe that $25$ iterations of the RK method are needed in
each constraint step in order to meet the criterion $|\xi(y(s))| < 10^{-8}$.

For the second scheme using $\Theta^A$, we simulate (\ref{ex1-scheme-0-a}) 
for $n=2\times 10^{7}$ steps with the step-size $h=0.005$. 
Notice that, a slightly smaller step-size $h$ is used, because in this case the non-gradient ODE flow 
(\ref{ex1-scheme-0-flow-a}) produces a drift force on the level set $\Sigma$.
In each constraint step, $\Theta^{A}(x^{(l+\frac{1}{2})})$ is computed by solving the ODE
(\ref{ex1-scheme-0-flow-a}) in the same way (with the same parameters) as we did in the first scheme.
On average, we find that $23$ iterations of the RK method are needed in
each constraint step in order to meet the criterion $|\xi(y(s))| < 10^{-8}$.

For the third scheme using $\Pi$, (\ref{ex1-scheme-1}) is simulated for $n=2\times 10^{7}$ steps with the step-size $h=0.01$. Using the parametrization (\ref{ex1-sigma-by-angle-theta}), we have $\Pi(x) = (c\cos\theta^*, \sin\theta^*)^T$, where 
\begin{align}
  \theta^* = \argmin_{\theta \in [0, 2\pi]} \Big((x_1-c\cos\theta)^2 +
  (x_2-\sin \theta)^2\Big)\,, \quad x=(x_1,x_2)^T\,.
  \label{ex1-scheme-1-opt}
\end{align}
Therefore, in each step, $\Pi(x^{(l+\frac{1}{2})})$ is computed by solving (\ref{ex1-scheme-1-opt}) 
using the simple gradient descent method. The step-size is fixed to $\Delta t
= 0.1$ and the gradient descent iteration terminates when the derivative of the objective
function in (\ref{ex1-scheme-1-opt}) has an absolute value that is less than $10^{-8}$.
On average, it requires $32$ gradient descent iterations in each step in order to meet the convergence criterion.

Let us make a comparison among the three schemes (\ref{ex1-scheme-0}), (\ref{ex1-scheme-0-a}) and (\ref{ex1-scheme-1}).
From Figure~\ref{fig-cmp-theta-pi} and Figure~\ref{fig-vector-field-theta-pi}, 
we can see that the three maps $\Theta$, $\Theta^A$ and $\Pi$ indeed have different effects.
Roughly speaking, comparing to the projection map $\Pi$, both $\Theta$
and $\Theta^A$ tend to map states towards one of the two vertices $(\pm c,
0)$, where $|\nabla\xi|$ are smaller, while $\Pi^A$ introduces a further rotational force on $\Sigma$.
Based on the states generated from these three schemes,
in Figure~\ref{fig-dist-and-traj} we show the empirical probability densities of the parameter $\theta$ in (\ref{ex1-sigma-by-angle-theta}).  
From the agreement between the empirical densities and the densities
computed from the analytical expressions in (\ref{ex1-mu1-mu2-theta}), we can
make the conclusion that 
the trajectories generated from the two schemes using $\Theta$ and $\Theta^A$ indeed sample the
probability measure $\mu_1$, while the trajectory generated from the scheme
using $\Pi$ samples $\mu_2$. 

Lastly, concerning the fourth scheme, we simulate (\ref{ex1-scheme-2}) for $n=10^7$
steps using the step-size $h = 0.0001$.
In this case, we find that it is necessary to choose a small step-size $h$ in order to keep the trajectory close to the level set $\Sigma$.
As can be seen from Figure~\ref{fig-em-scatter}, 
even with this smaller step-size $h=0.0001$,
the generated trajectory departs from the level set $\Sigma$. This
indicates the limited usefulness of the direct Euler-Maruyama
discretization of the SDE (\ref{dynamics-1-submanifold}) in long time simulations.
\begin{figure}[htpb]
  \centering
  \includegraphics[width=13cm]{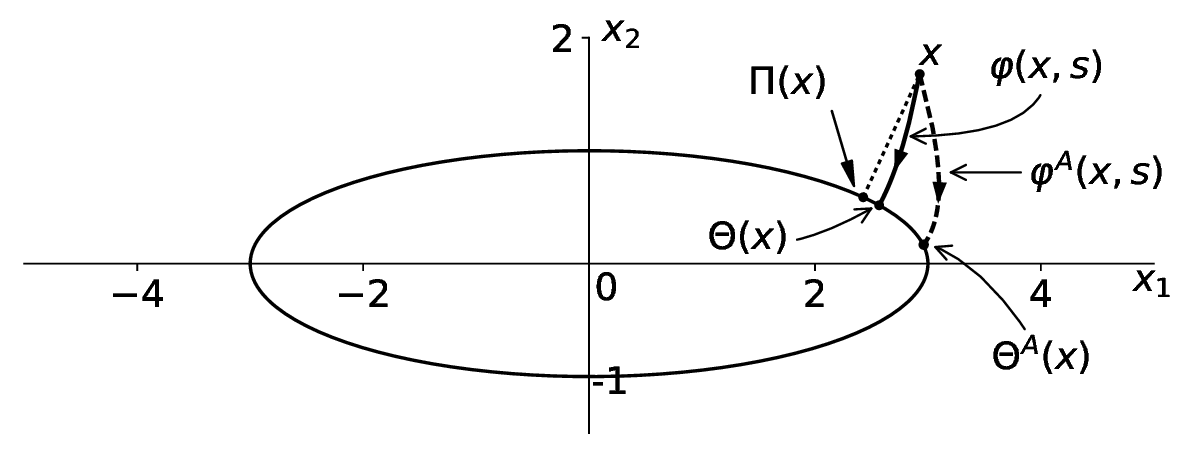}
  \caption{Example $1$. Given $x \in \mathbb{R}^2$, $\Pi(x)$ is the state on 
  the ellipse $\Sigma$ which achieves the minimal distance to $x$, while $\Theta(x)$ and
  $\Theta^A(x)$ are the limits of the ODE flows (\ref{ex1-scheme-0-flow}) and (\ref{ex1-scheme-0-flow-a}) starting from $x$, respectively. \label{fig-cmp-theta-pi}}
\end{figure}
\begin{figure}[htpb]
  \centering
  \includegraphics[width=15cm]{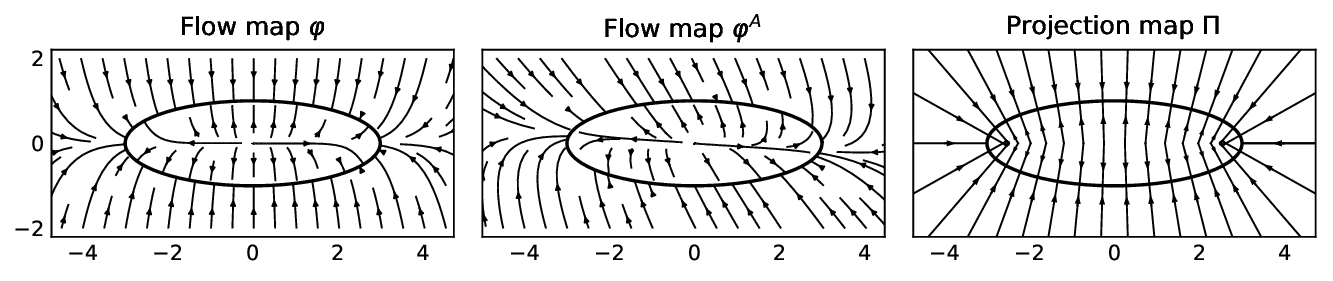}
  \caption{Example $1$. Left: the streamline of the flow map $\varphi$ in (\ref{ex1-scheme-0-flow}).
  Middle: the streamline of the flow map $\varphi^A$ (\ref{ex1-scheme-0-flow-a}) with the matrix $A$ in (\ref{ex1-scheme-1-matrix-a}).
  Right: illustration of the projection $\Pi$. Points on each straight line are
mapped to the same point on $\Sigma$. \label{fig-vector-field-theta-pi}}
\end{figure}
\begin{figure}[htp]
  \includegraphics[width=12cm, height=4cm]{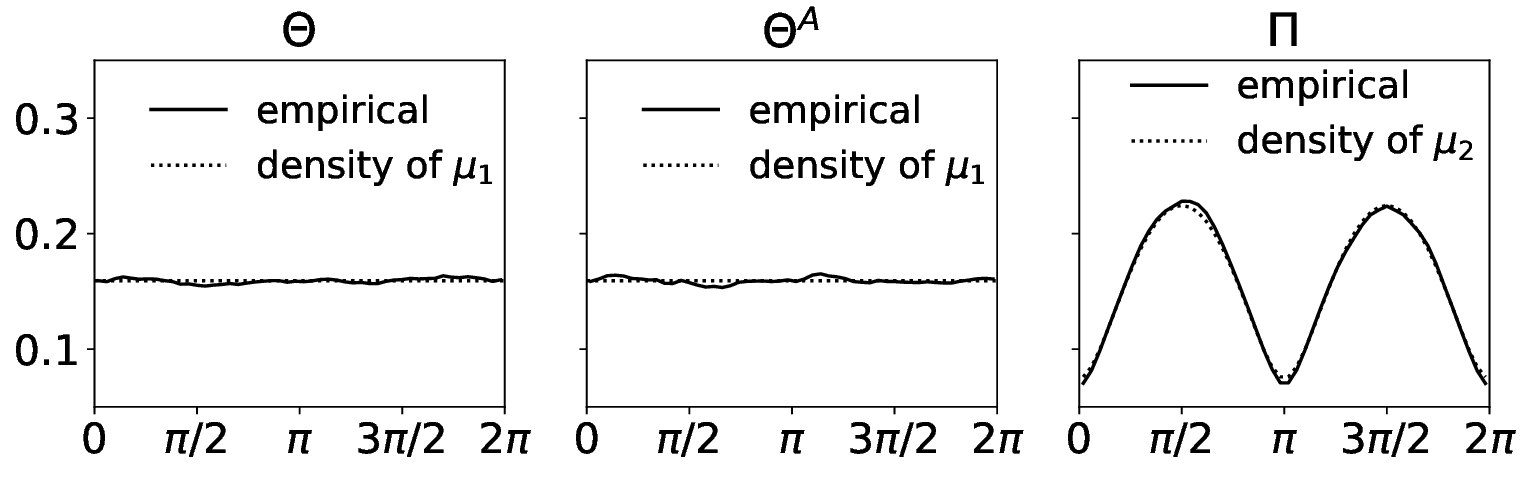}
  \centering
  \caption{Example $1$. The probability densities of the parameter
  $\theta$, computed from the scheme (\ref{ex1-scheme-0}) using $\Theta$ (left plot),
  the scheme (\ref{ex1-scheme-0-a}) using $\Theta^A$ (middle plot), and the
  scheme (\ref{ex1-scheme-1}) using $\Pi$ (right plot).
In each plot, dotted curves are the probability densities computed from the analytical expressions of $\mu_1, \mu_2$ in
(\ref{ex1-mu1-mu2-theta}), respectively. Solid lines are the empirical
probability densities of $\theta$ estimated using the states generated from
the schemes (\ref{ex1-scheme-0}), (\ref{ex1-scheme-0-a}) and (\ref{ex1-scheme-1}), respectively. 
\label{fig-dist-and-traj}}
\end{figure}
\begin{figure}[htp]
  \includegraphics[width=8cm, height=5cm]{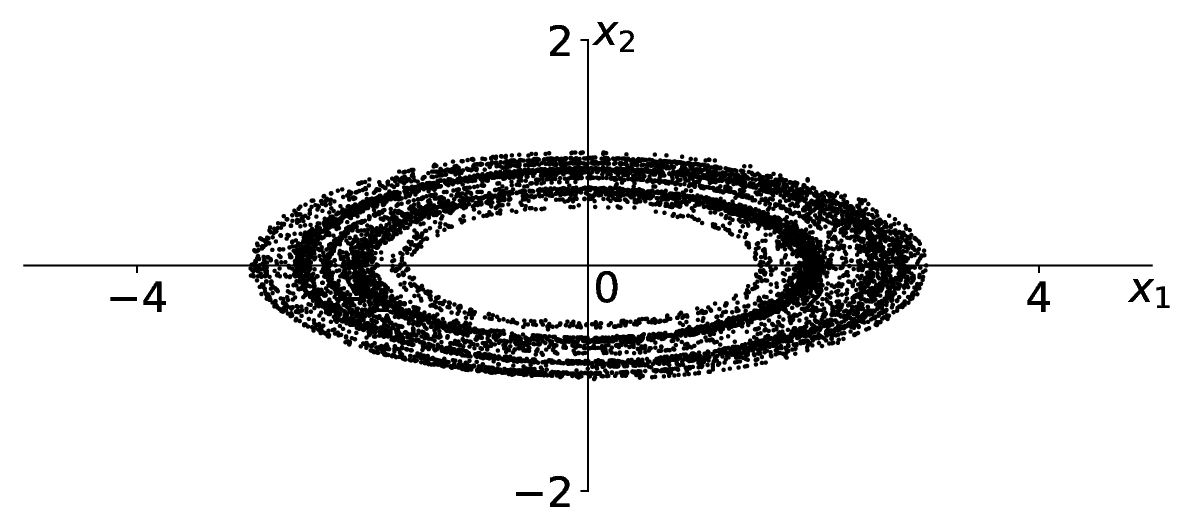}
  \centering
  \caption{Example $1$. States generated from the Euler-Maruyama discretization (\ref{ex1-scheme-2}),
where we choose $h=0.0001$ and $n=10^{7}$. In this case, the sampled states deviate from the level set $\Sigma$. 
  \label{fig-em-scatter}}
\end{figure}
\subsubsection*{Example 2: Numerical comparison with the
Metropolis-adjusted method on the special orthogonal group $SO(11)$.}
In this example, we compare the computational efficiency between our scheme 
(\ref{micro-scheme-initial-repeat}) using the flow map $\Theta$ and the Metropolis-adjusted method
introduced in~\cite{goodman-submanifold}. We consider the special orthogonal
group $SO(11)$, which consists of  orthogonal matrices of size $11 \times 11$ with determinant equals to $1$. 
This example is taken from~\cite{goodman-submanifold}. The authors there applied their method to the estimation of the mean value of the function $f(x) = \mbox{Tr}(x)$, 
i.e, the trace of the matrix $x$, where $x$ follows the surface measure of $SO(11)$.
The manifold $SO(11)$ can be viewed as (one connected component of) the level set 
of the map $\xi: \mathbb{R}^{121} \rightarrow \mathbb{R}^{66}$, which includes
all the row ortho-normality constraints. Readers are referred to the original work~\cite{goodman-submanifold}
for a detailed introduction on the example.

In this numerical study, we implement both the scheme
(\ref{micro-scheme-initial-repeat}) and the (Metropolis-adjusted) algorithm
in~\cite{goodman-submanifold} to estimate the mean value of $\mbox{Tr}(x)$.
Notice that, since $\det(\nabla\xi^T\nabla\xi)$ is constant, the conditional
measure $\mu_1$ in (\ref{mu1-intro}) coincides with the surface measure of $SO(11)$ when we choose the
potential $U\equiv 0$.
In both cases, we generate $n=10^6$ samples on the same laptop (CPU: Intel Core i5,
$2.60$GHz, $4$ cores; system: Ubuntu $18.04$).
For the scheme (\ref{micro-scheme-initial-repeat}), we choose the step-size
$h=0.022$. The map $\Theta$ is computed
by integrating the ODE~(\ref{phi-map}) with $a=\mbox{id}$, until
the condition $|\xi(\varphi(x,s))|<10^{-9}$ is satisfied.
To accelerate the ODE integration, we have applied the adaptivity technique in the second point of Remark~\ref{rmk-on-numerical-scheme-about-ode} with $\kappa=0.5$.
Starting from the initial step-size $\Delta s=0.2$, the step-size used in the ODE
integration is divided by $2.0$ whenever we find that the value of $|\xi|$ is not decreasing.
(The numerical error of $\Theta$ is $2.0 \times 10^{-4}$ on average, comparing to the
reference solution that is obtained by solving the ODE with $\kappa=0$ and
the fixed step-size $\Delta s= 0.002$.)
Furthermore, the new state will be discarded (and resampled) if its determinant equals to $-1$.
With these parameters, we observe that on average $37$ Runge-Kutta iterations
are needed for each evaluation of the map $\Theta$.
In total, it takes $2676.9$ seconds to generate $n=10^6$
samples, while the estimated mean value of $\mbox{Tr}(x)$ is $3.8 \times 10^{-3}$
with a statistical error $3.9 \times 10^{-3}$.
For the algorithm in~\cite{goodman-submanifold}, 
the maximal number of Newton steps is set to $10$ and the proposal length
scale is chosen to be $0.257$~\footnote{The roles of the proposal length scale
in~\cite{goodman-submanifold} and the step-size $h$ in the
scheme~(\ref{micro-scheme-initial-repeat}) are different. 
The proposal length scale $0.257$ used in this example corresponds to a step-size $0.033$
($\approx 0.257^2/2$) in the scheme (\ref{micro-scheme-initial-repeat}).}.
In our experiment, we find that this proposal length scale (different from the
one used in~\cite{goodman-submanifold}) leads to slightly smaller correlation time.
Within the entire computation, the success rate of the Newton's method (i.e.,
the rate that the Newton's method converges) is $67.2\%$ and 
each time it takes $5$-$6$ iterations 
on average for the Newton's method to reach convergence (the convergence
criteria is $|\xi(x)| < 10^{-9}$).
In total, it takes $7315.9$ seconds to generate $n=10^6$ samples. The
estimated mean value is $-3.8\times 10^{-3}$ and the statistical error is $3.8\times 10^{-3}$. 

The empirical density distributions of $\mbox{Tr}(x)$ using both the scheme
(\ref{micro-scheme-initial-repeat}) and the algorithm in~\cite{goodman-submanifold} are shown in the left
plot of Figure~\ref{fig-ex2-acf-density}, while the autocorrelation functions are
plotted in the right plot of Figure~\ref{fig-ex2-acf-density}.
From these results, we can conclude that in this example both approaches provide similar
statistical estimations (The autocorrelation time using Metropolis-adjusted
method is slightly smaller with the above parameters.). At the same time, the total computational time using the scheme (\ref{micro-scheme-initial-repeat})
is less than half of the computational time required by the algorithm in~\cite{goodman-submanifold}.
For this example, although the average number of Newton steps in the latter algorithm is
smaller than the average number of ODE iterations, the computational cost of
each Newton step is indeed larger. We refer to Remark~\ref{rmk-cmp-metropolis-or-not} for the
comparison of computational complexity of both approaches.
\begin{figure}[htp]
  \includegraphics[width=6.8cm, height=6cm]{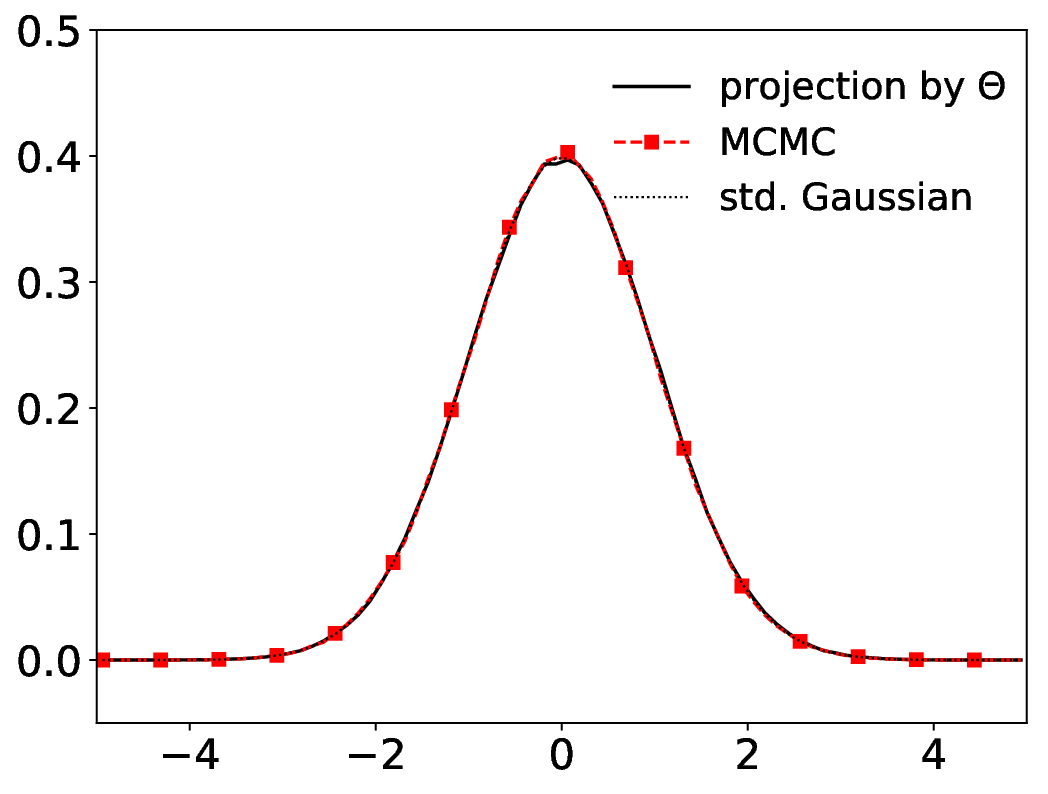}
  \includegraphics[width=6.8cm, height=6cm]{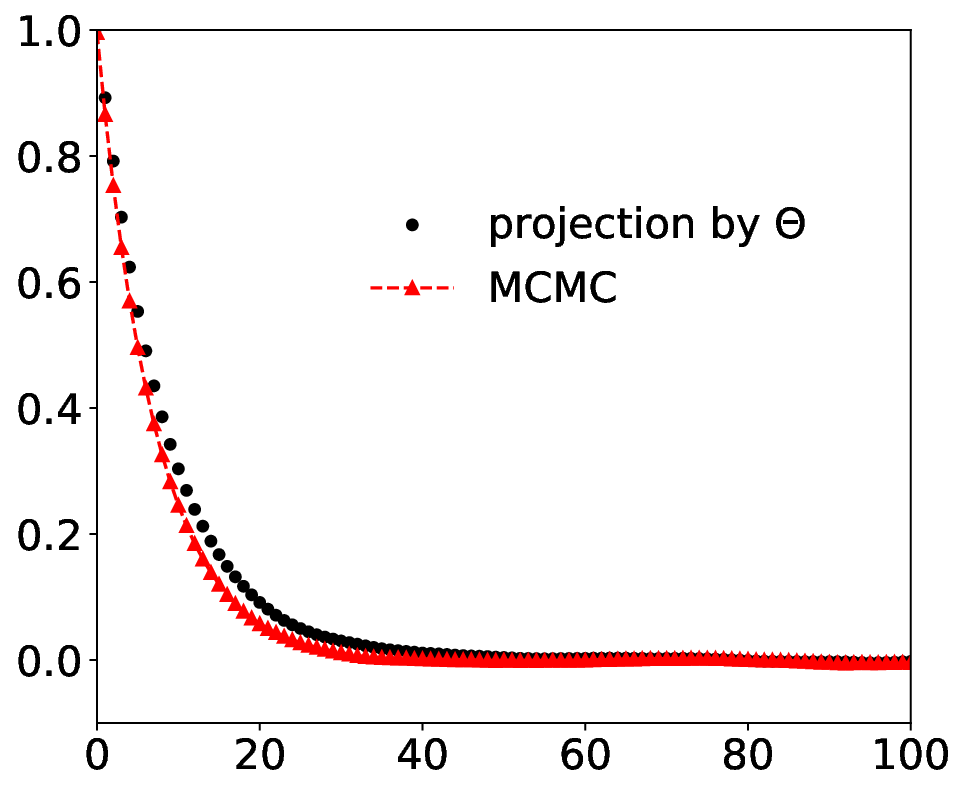}
  \centering
  \caption{Example $2$. 
  Left: empirical density plots of the statistics
  $\mbox{Tr}(x)$, where $x \in SO(11)$. Both empirical densities resemble
  the probability density of the standard Gaussian random variable.
  Right: autocorrelation functions of the sampled trajectories.
  In both plots, the curve with label ``projection by $\Theta$'' and 
  the curve with label ``MCMC'' are the results obtained using the scheme (\ref{micro-scheme-initial-repeat}) and the Metropolis-adjusted algorithm in \cite{goodman-submanifold}, respectively.
  \label{fig-ex2-acf-density}}
\end{figure}
\subsubsection*{Example 3: Removing stiffness by choosing a non-constant matrix $a$}
In this example, we choose the reaction coordinate function 
\begin{align*}
  \xi(x) = \xi(x_1, x_2,\cdots, x_d) =
\frac{1}{2} \Big(x_1^2 + x_2^2 + \cdots + x^2_d - 1\Big)\,.
\end{align*}
Correspondingly, the level set 
 \begin{align*}
\Sigma=\Big\{(x_1, x_2,\cdots, x_d) \in \mathbb{R}^d~\Big|~ x_1^2 + x_2^2 +\cdots + x_d^2 = 1\Big\}
\end{align*}
is the $(d-1)$-dimensional unit sphere, and we have 
\begin{align*}
\nabla\xi = (x_1,x_2,\cdots,x_d)^T, \quad \nabla\xi^T\nabla\xi = \sum_{i=1}^d x_i^2\,.
\end{align*}
In the following, we give an example to show that in some applications it is
helpful to use a non-constant matrix $a$ in the numerical scheme \eqref{micro-scheme-initial-repeat}. 
Briefly speaking, varying the matrix $a$ properly allows to rescale the scheme along different directions. 
It has a preconditioning effect when different time scales (stiffness) exist. 

Consider $d=3$ and the potential $U = \frac{\theta^2}{2\epsilon}$, where
$\epsilon>0$ is a small parameter, $\theta$ is the angle of the state $x = (x_1, x_2, x_3)$ under the  
spherical coordinate system 
\begin{align*}
  x_1 = \rho\cos\theta\cos\varphi\,, \quad x_2 = \rho\cos\theta\sin\varphi\,, \quad x_3 = \rho\sin\theta\,,
\end{align*}
where $\rho \ge 0$, $\theta \in [-\frac{\pi}{2},\frac{\pi}{2}]$, and $\varphi \in
[0, 2\pi]$. We can verify that 
\begin{align}
  \nabla \theta = \frac{1}{\rho^2}\Big(-\frac{x_1x_3}{(x_1^2+x_2^2)^{\frac{1}{2}}},\,
  -\frac{x_2x_3}{(x_1^2+x_2^2)^{\frac{1}{2}}},\,
  (x_1^2+x_2^2)^{\frac{1}{2}}\Big)^T\,.
  \label{ex3-nabla-theta-exp}
\end{align}
Correspondingly, with the choice of $\sigma=a=\mbox{id}$, the scheme (\ref{micro-scheme-initial-repeat}) is 
\begin{align}
  \begin{split}
    x^{(l + \frac{1}{2})} =& x^{(l)}-\frac{1}{\epsilon} (\theta \nabla \theta)(x^{(l)})\,h + \sqrt{2 \beta^{-1}h}\, \bm{\eta}^{(l)}\,, \\
  x^{(l+1)} =& \Theta\big(x^{(l+\frac{1}{2})}\big)\,,
\end{split}
\label{ex3-scheme-a-id}
\end{align}
where $\nabla\theta$ is given in (\ref{ex3-nabla-theta-exp}). 
Notice that, the coefficients in (\ref{ex3-scheme-a-id}) are $\mathcal{O}(\frac{1}{\epsilon})$ when
$\epsilon$ is small.
In particular, it implies that sampling  the invariant measure using (\ref{ex3-scheme-a-id})
will be inefficient when $\epsilon$ is small, since the step-size $h$
will be severely limited due to the large magnitude of the coefficients in (\ref{ex3-scheme-a-id}). 

On the other hand, based on the form of $U$ and the expression
(\ref{ex3-nabla-theta-exp}), we consider the orthogonal vectors
\begin{align*}
  \begin{split}
  \bm{\sigma}_1 =& (x_1 , x_2 , x_3)^T=\nabla\xi\,,\quad \bm{\sigma}_2 = (x_2 , -x_1 , 0)^T\,,\\
  \bm{\sigma}_3 =& \Big(-\frac{\sqrt{\epsilon}\,x_1x_3 }{(x_1^2+x_2^2)^{\frac{1}{2}}}, 
    -\frac{\sqrt{\epsilon}\,x_2x_3}{(x_1^2+x_2^2)^{\frac{1}{2}}}, 
    \sqrt{\epsilon} (x_1^2+x_2^2)^{\frac{1}{2}}\Big)^T=\sqrt{\epsilon}\rho^2
    \nabla\theta\,, 
  \end{split}
\end{align*}
and we define $\sigma =
(\bm{\sigma}_1,\bm{\sigma}_2, \bm{\sigma}_3) \in \mathbb{R}^{3 \times 3}$.
Direct calculation shows that 
\begin{align}
 a=\sigma \sigma^T = 
  \begin{pmatrix}
    x_1^2 + x_2^2+\frac{\epsilon x_1^2x_3^2}{x_1^2+x_2^2} & 
    \frac{\epsilon x_1x_2 x_3^2}{x_1^2+x_2^2} & (1-\epsilon) x_1x_3\\
\frac{\epsilon x_1x_2 x_3^2}{x_1^2+x_2^2}  & x_1^2+x_2^2+\frac{\epsilon
  x_2^2x_3^2}{x_1^2+x_2^2} & (1-\epsilon) x_2x_3 \\
    (1-\epsilon) x_1x_3 & (1-\epsilon) x_2x_3 & x_3^2 + \epsilon(x_1^2 + x_2^2)
  \end{pmatrix}\,.
    \label{ex3-1-id-eps-a}
\end{align}
Correspondingly, using (\ref{ex3-1-id-eps-a}), the scheme (\ref{micro-scheme-initial-repeat}) becomes 
\begin{align}
  \begin{split}
    x^{(l + \frac{1}{2})}_i =& x^{(l)}_i+\Big[-\theta \frac{\partial
    \theta}{\partial x_i}+\frac{1}{\beta} \frac{\partial a_{ij}}{\partial x_j}\Big](x^{(l)})\,h +
    \sqrt{2 \beta^{-1}h}\, \sigma_{ij}(x^{(l)})\,\bm{\eta}^{(l)}_j\,, \quad 1
    \le i \le 3\,,\\
  x^{(l+1)} =& \Theta\big(x^{(l+\frac{1}{2})}\big)\,,
\end{split}
\label{ex3-scheme-a-non-id}
\end{align}
where $\Theta(x)$ is the limit of the ODE flow 
\begin{align}
  \dot{y}(s) = - \xi\big(y(s)\big)\Big(2\xi\big(y(s)\big)+1\Big)\,y(s)\,,\quad  y(0) = x\,.
  \label{ex3-scheme-a-non-id-ode}
\end{align}
Importantly, in contrast to (\ref{ex3-scheme-a-id}), the scheme
(\ref{ex3-scheme-a-non-id})--(\ref{ex3-scheme-a-non-id-ode}) is no longer stiff when $\epsilon$ is small.

Now we compare the numerical efficiency between the
schemes (\ref{ex3-scheme-a-id}) and (\ref{ex3-scheme-a-non-id})--(\ref{ex3-scheme-a-non-id-ode}).
First of all, since the surface measure on $\Sigma$ satisfies $d\nu =
\cos\theta d\theta\,d\varphi$, we know that the target measure is 
\begin{align}
  d\mu = \frac{1}{Z} e^{-\frac{\beta \theta^2}{2\epsilon}} d\nu = 
  \frac{1}{Z} e^{-\frac{\beta \theta^2}{2\epsilon}} \cos\theta\, d\theta\,d\varphi\,.
  \label{ex3-mu-eps}
\end{align}
In the numerical study, we choose $\epsilon = 0.005$ and generate $n=10^7$ states for both schemes.
For the scheme (\ref{ex3-scheme-a-id}) which corresponds to $a=\mbox{id}$, 
we use both a small step-size $h=0.0002$ and a (relatively) larger step-size $h=0.005$, while
we choose a large step-size $h=0.01$ in the scheme (\ref{ex3-scheme-a-non-id})--(\ref{ex3-scheme-a-non-id-ode}) . The empirical probability
densities of the angles $\theta, \varphi$ for the two schemes are shown
in Figure~\ref{fig-ex3-1} and Figure~\ref{fig-ex3-2}, respectively.
From Figure~\ref{fig-ex3-1}, we see that the step-size $h$ has to be small ($h=0.0002$) in (\ref{ex3-scheme-a-id}) in order to produce the
correct probability density of the angle $\theta$ (left plot). However, with such a small
$h$, the estimated empirical density of the angle $\varphi$ (right plot) is
still noisy with $n=10^7$. On the other hand, for the scheme 
(\ref{ex3-scheme-a-non-id})--(\ref{ex3-scheme-a-non-id-ode}) which corresponds
to the matrix $a$ in (\ref{ex3-1-id-eps-a}),
Figure~\ref{fig-ex3-2} shows that the probability densities of both angles
$\theta, \varphi$ are well approximated using the large step-size $h=0.01$.
Therefore, we conclude that in this example choosing the non-constant matrix $a$ in (\ref{ex3-1-id-eps-a}) indeed 
helps improve the sampling efficiency.
\begin{figure}[htpb]
  \centering
  \includegraphics[width=15cm]{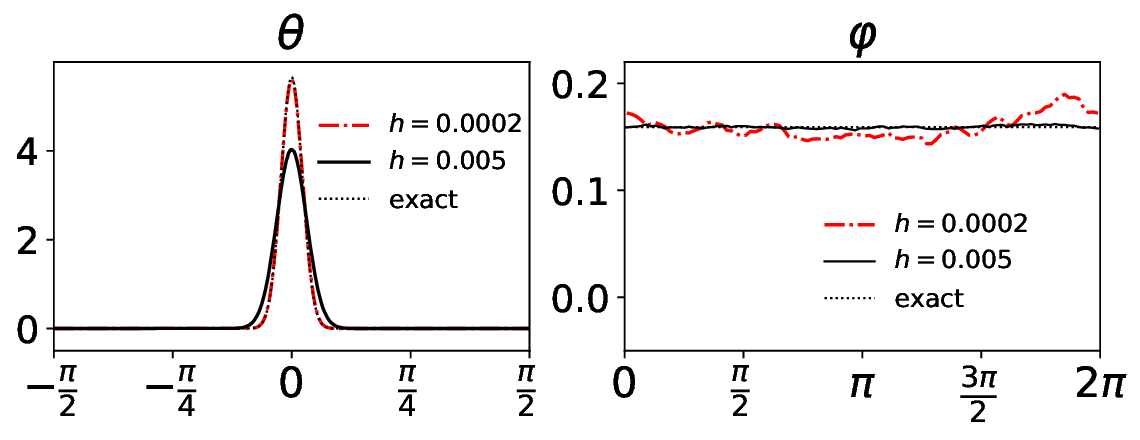}
  \caption{Example $3$. The empirical densities of the angles $\theta$ and
  $\varphi$ are estimated using the scheme (\ref{ex3-scheme-a-id}) which corresponds to $a=\mbox{id}$.
  $\epsilon = 0.005$ and $n=10^7$ states are sampled, using a small step-size
  $h=0.0002$ and a larger step-size $h=0.005$. The curves with label ``exact'' are the analytical 
  marginal densities computed from (\ref{ex3-mu-eps}).
  \label{fig-ex3-1}}
\end{figure}
\begin{figure}[htpb]
  \centering
  \includegraphics[width=15cm]{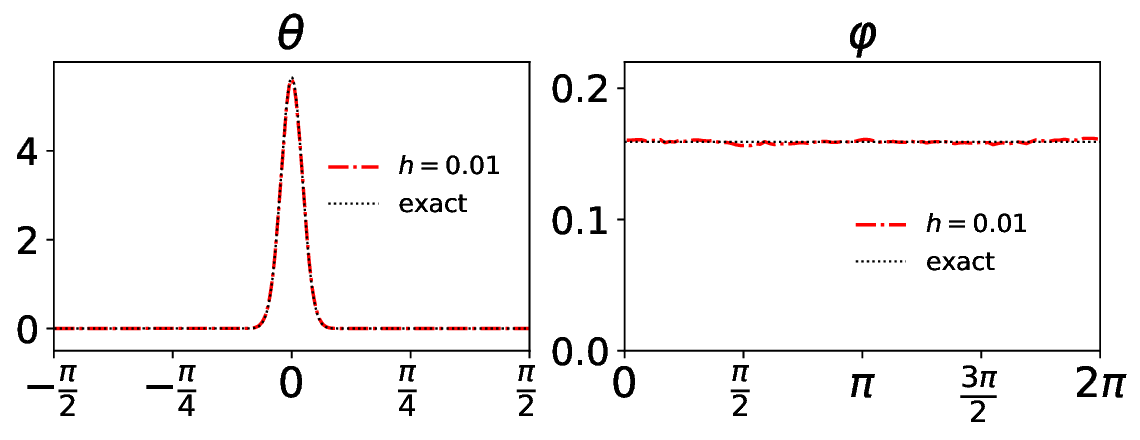}
  \caption{Example $3$. The empirical densities of the angles $\theta$ and
  $\varphi$ are estimated using the scheme (\ref{ex3-scheme-a-non-id})--(\ref{ex3-scheme-a-non-id-ode})
   which corresponds to the matrix $a$ in (\ref{ex3-1-id-eps-a}).
  $\epsilon = 0.005$ and $n=10^7$ states are sampled using a large step-size
  $h=0.01$. The curves with label ``exact'' are the analytical marginal densities
  computed from (\ref{ex3-mu-eps}). \label{fig-ex3-2}}
\end{figure}
\FloatBarrier

\section{Conclusions}
\label{sec-conclusion}
Ergodic diffusion processes on a submanifold of $\mathbb{R}^d$ and related numerical sampling schemes have been considered
in this work. A family of SDEs has been obtained whose invariant measures coincide with
the given probability measure on the submanifold. In particular, for the
conditional probability measure, we found that the corresponding SDEs
have a relatively simple form. We proposed and analyzed a consistent numerical scheme
which only requires $1$st order derivatives of the reaction coordinate function.
Different sampling schemes on the submanifolds are numerically evaluated. 

The current work extends results in the literature and may further contribute to 
both the analysis and the development of numerical methods on
related problems, in particular problems in molecular dynamics such as free energy calculation and model
reduction of high-dimensional stochastic processes. 
Closely related to the current paper, the following topics could be considered. 
First, the ``non-reversible'' scheme (\ref{scheme-non-reversible}) is supported by a simple
numerical example but theoretical justification still needs to be investigated. This will be
considered in future following the approach described in Remark~\ref{rmk-conjecture-non-reversible-scheme}. 
Second, the constrained numerical schemes in the current work do
not involve system's momentum variables. In view of the
work~\cite{Tony-constrained-langevin2012}, it is interesting to study the
Langevin dynamics under different constraints (such as certain variants of the map $\Theta$ used in this work).
Third, there is a research interest in the literature to study the effective dynamics of
molecular systems along a given reaction coordinate $\xi$. The coefficients of the
effective dynamics are usually defined as averages on the level set of
$\xi$~\cite{effective_dynamics}. As an application of the numerical scheme proposed in this work, we will study
numerical algorithms to simulate the effective dynamics. This topic is related to the
heterogeneous multiscale methods~\cite{hmm_review_2007} and the equation-free
approach~\cite{eqfree-computer,eqfree-algo} in the literature. 
\section*{Acknowledgement}
This work is funded by the Einstein Center of Mathematics (ECMath) through project CH21.
The author would like to thank Gabriel Stoltz for stimulating discussions on
constrained Langevin processes at the Institut Henri Poincar{\'e} - Centre
{\'E}mile Borel during the trimester ``Stochastic Dynamics Out of Equilibrium''. 
The author appreciates the hospitality of this institution.
The author also thanks the anonymous referees for their valuable comments and
criticism which helped improve the manuscript substantially.
\appendix
\section{Useful facts about the Riemannian manifold $\mathcal{M}$}
\label{app-sec-manifold}
In this section, we present technical details of Section~\ref{sec-n-rn} related to the Riemannian 
manifold $\mathcal{M}=(\mathbb{R}^d, g)$, where $g=a^{-1}$. 
The main result is Proposition~\ref{laplaceN-state}, where we give the expression of
the Laplacian-Beltrami operator $\Delta^{\Sigma}$ on the level set $\Sigma$ in (\ref{levelset-sigma-intro}), viewed as
a submanifold of $\mathcal{M}$. Before that, we first introduce some
notations and quantities related to $\mathcal{M}$ and $\Sigma$.
Readers are referred to \cite{do1992riemannian,bishop1964geometry,jost2008riemannian,petersen2006riemannian}
for related discussions on general Riemannian manifolds. 

Under Assumption~\ref{assump-1}, given two vectors $\bm{u} = (u_1, u_2, \cdots, u_d)^T$, $\bm{v} = (v_1, v_2, \cdots, v_d)^T$,
we consider the space $\mathbb{R}^d$ with the weighted inner product
\begin{align}
  g(\bm{u}, \bm{v}) = \langle \bm{u}, \bm{v}\rangle_g = u_i (a^{-1})_{ij} v_j\,.
  \label{ip-a}
\end{align}
The inner product in (\ref{ip-a}) defines a Riemannian metric $g$ on $\mathbb{R}^d$ and 
we denote by $\mathcal{M} = (\mathbb{R}^d, g)$
the Riemannian manifold $\mathbb{R}^d$ endowed with this metric.

Notice that $\mathcal{M}$ as a manifold is quite special
(simple), in that 
it has a natural global coordinate chart which is given by the usual Euclidean coordinate.
Since we will always work with this coordinate, we will
not distinguish between tangent vectors (operators acting on functions) and their coordinate representations
($d$-dimensional vectors). 
In particular,  $\bm{e}_i$ denotes the vector whose $i$th component equals
to $1$ while all the other $d-1$ components equal to $0$, where $1 \le i \le d$.
At each point $x \in \mathcal{M}$, vectors $\bm{e}_1, \bm{e}_2,\,\cdots, \bm{e}_d$ form a basis of the
tangent space $T_x\mathcal{M}$ and under this basis we have $g=a^{-1}$, as can be seen from (\ref{ip-a}).

Denote by $\mbox{grad}^{\mathcal{M}}$, $\mbox{div}^{\mathcal{M}}$ the
gradient and the divergence operator on $\mathcal{M}$, respectively.
For any smooth function $f: \mathcal{M} \rightarrow \mathbb{R}$, it is direct to verify that 
\begin{align*}
  \mbox{grad}^{\mathcal{M}} f = g^{ij} \frac{\partial f}{\partial x_j}
\bm{e}_i = (a\nabla f)_i\, \bm{e}_i\, ,
\end{align*} 
where $g^{ij}=(g^{-1})_{ij} = a_{ij}$, and $\nabla f$ 
denotes the ordinary gradient operator for functions on the Euclidean space $\mathbb{R}^d$.
For simplicity, we will also write $\partial_i f$ for the partial derivative with respect to $x_i$, and $(a\nabla
f)_i$ to denote the $i$th component of the vector $a\nabla f$, i.e., 
$\partial_i f = \frac{\partial f}{\partial x_i}$, and 
$(a\nabla f)_i = a_{ij}\frac{\partial f}{\partial x_j} = a_{ij}\partial_jf$.

The Laplace-Beltrami operator on $\mathcal{M}$ is defined by
$\Delta^{\mathcal{M}} f = \mbox{div}^{\mathcal{M}} (\mbox{grad}^{\mathcal{M}}
f)$. Equivalently, we have $\Delta^{\mathcal{M}} f =
\mbox{tr}(\mbox{Hess}^{\mathcal{M}}f)$, where  $\mbox{Hess}^{\mathcal{M}}$ is
the Hessian operator on $\mathcal{M}$. 
The integration by parts formula on $\mathcal{M}$ has the form
\begin{align}
  \int_{\mathcal{M}} (\Delta^{\mathcal{M}} f_1)\,f_2\,dm = - \int_{\mathcal{M}}
  \langle \mbox{grad}^{\mathcal{M}} f_1,
  \mbox{grad}^{\mathcal{M}} f_2\rangle_g\, dm = \int_{\mathcal{M}}
  (\Delta^{\mathcal{M}} f_2)f_1\,dm\,,  
  \label{integrate-by-part}
\end{align}
for $\forall f_1,\, f_2 \in C^\infty_0(\mathcal{M})$,
where $dm= (\det g)^{\frac{1}{2}} dx = (\det a)^{-\frac{1}{2}} dx$ is the volume form, and $C^{\infty}_0(\mathcal{M})$ consists of all smooth functions on $\mathcal{M}$ with compact support.

Besides the vector basis $\bm{e}_1,\, \bm{e}_2,\, \cdots, \bm{e}_d$, the vectors 
\begin{align}
\bm{\sigma}_i =
(\sigma_{1i}, \sigma_{2i},\cdots, \sigma_{di})^T, \quad 1 \le i \le d\,, 
\label{sigma-basis}
\end{align}
will also be useful. Note that $a=\sigma\sigma^T=g^{-1}$ implies $\langle \bm{\sigma}_i,
\bm{\sigma}_j\rangle_g = (a^{-1})_{rl}\sigma_{ri}\sigma_{lj} = \delta_{ij}$.
In other words, $\bm{\sigma}_1,\,\bm{\sigma}_2,\, \cdots, \bm{\sigma}_d$ form
an orthonormal basis of $T_x\mathcal{M}$ at each $x \in \mathcal{M}$.

Denote by $\nabla^{\mathcal{M}}$ the Levi-Civita connection on $\mathcal{M}$.
Given $x \in \mathcal{M}$ and a tangent vector $\bm{v} \in T_x\mathcal{M}$, 
$\nabla^{\mathcal{M}}_{\bm{v}}$
is the covariant derivative operator on $\mathcal{M}$ along the vector $\bm{v}$.
For two vectors $\bm{u} = (u_1, u_2, \cdots, u_d)^T$, $\bm{v} = (v_1, v_2, \cdots, v_d)^T$,
the Hessian of a smooth function $f: \mathcal{M} \rightarrow \mathbb{R}$ is defined as 
\begin{align}
  \mbox{Hess}^{\mathcal{M}} f(\bm{u},\bm{v}) = \bm{u}(\bm{v}f) -
  (\nabla^{\mathcal{M}}_{\bm{u}}\bm{v}) f =
  u_iv_j \mbox{Hess}^{\mathcal{M}} f\big(
  \bm{e}_i,\bm{e}_j\big) = u_iv_j \Big(\frac{\partial^2 f}{\partial
  x_i\partial x_j} - \Gamma_{ij}^l \frac{\partial f}{\partial x_l}\Big)\,,
  \label{hessian-chart}
\end{align}
  where  
  \begin{align}
    \Gamma_{ij}^l =& \frac{1}{2} g^{lr}\Big(\frac{\partial g_{ir}}{\partial x_j}
  + \frac{\partial g_{jr}}{\partial x_i} - \frac{\partial g_{ij}}{\partial
  x_r}\Big) 
  = \frac{1}{2} a_{lr}\Big(\frac{\partial (a^{-1})_{ir}}{\partial x_j}
  + \frac{\partial (a^{-1})_{jr}}{\partial x_i} - \frac{\partial (a^{-1})_{ij}}{\partial
  x_r}\Big)\,, 
  \quad 1 \le i,j,l \le d
  \label{gamma-det-exp}
\end{align}
  are the Christoffel's symbols defined by  
  $\nabla^{\mathcal{M}}_{\bm{e}_i} \bm{e}_j =
\Gamma^{l}_{ij} \bm{e}_l$, for $1 \le i,j\le d$.

Now let us consider the level set 
\begin{align}
  \Sigma=\xi^{-1}(\bm{0})=\Big\{ x \in \mathcal{M}=\mathbb{R}^d ~\Big|~ \xi(x)
  = \bm{0} \in \mathbb{R}^k\Big\}
\label{levelset-sigma}
\end{align}
of the $C^2$ function $ \xi : \mathbb{R}^d \rightarrow \mathbb{R}^k$ with $\xi = (\xi_1, \xi_2, \cdots, \xi_k)^T$, $1 \le k < d$. 
Applying regular value theorem~\cite{banyaga2004lectures}, we know that $\Sigma$ is a $(d-k)$-dimensional submanifold of $\mathcal{M}$, under Assumption~\ref{assump-2}.

Given $x \in \Sigma$ and a vector $\bm{v} \in T_{x}\mathcal{M}$, the orthogonal
projection operator ($d \times d$ matrix) $P : \mathbb{R}^d \rightarrow
T_x\Sigma$ is defined such that $\langle \bm{v}-P\bm{v}, \bm{u}\rangle_g = 0$,
for $\forall\, \bm{u} \in T_x\Sigma\,$.
It is straightforward to verify that $P=\mbox{id}-
a\nabla\xi\Psi^{-1}\nabla\xi^T$, or entry-wise, 
\begin{align}
  P_{ij} = \delta_{ij} -
  (\Psi^{-1})_{\alpha\gamma}\,(a\,\nabla\xi_{\alpha})_i
  \partial_{j}\xi_{\gamma} \,, \quad 1 \le i,j \le d\,,
\label{p-ij}
\end{align}
where $\Psi$ is the invertible $k \times k$ symmetric matrix at each point $x
\in \Sigma$, given by
\begin{align}
  \Psi_{\alpha\gamma} = \langle\mbox{grad}^{\mathcal{M}}\xi_\alpha, \mbox{grad}^{\mathcal{M}}\xi_\gamma\rangle_g 
  =(\nabla\xi^T a \nabla \xi)_{\alpha\gamma}\,, \quad 1 \le \alpha, \gamma \le
  k\,.
  \label{psi-ij}
\end{align}
In the above, $\nabla\xi$ denotes the $d \times k$ matrix with entries
$(\nabla\xi)_{i\alpha} =\partial_i\xi_\alpha$, for $1 \le \alpha \le k$, $1 \le i
\le d$. We can verify that 
\begin{align}
 \quad aP^T = Pa, \quad P^2 = P\,, \quad P^T\nabla\xi_\alpha = 0\,, \quad
 1 \le \alpha \le k\,.
 \label{p-fact}
\end{align}
  Let us further assume that $\bm{v} \in T_x\Sigma$ is a tangent vector of $\Sigma$ at $x$. 
  Since $\{\bm{\sigma}_i\}_{1 \le i \le d}$ forms an orthonormal basis of the
  tangent space $T_x \mathcal{M}$,
we have $\bm{v} =
\langle \bm{v}, \bm{\sigma}_i\rangle_g \bm{\sigma}_i$. 
Using the fact that $P\bm{v} = \bm{v}$, we obtain
$\bm{v}
= \langle \bm{v},
\bm{p}_i\rangle_g \bm{p}_i$,
where $\bm{p}_i = P\bm{\sigma}_i \in T_x\Sigma$. If we denote $\bm{p}_i = P_{i,j} 
\bm{e}_j$, then it follows from (\ref{p-ij}) and (\ref{p-fact}) that
\begin{align}
  \begin{split}
    & P_{i,j} = (P\sigma)_{ji} = \sigma_{ji} -
    (\Psi^{-1})_{\alpha\gamma} (a \nabla\xi_\alpha)_j
  (\sigma^T\nabla\xi_\gamma)_i  \\ 
    &P_{l,i}P_{l,j} = (Pa)_{ij}=(aP^T)_{ij} = a_{ij} - 
     (\Psi^{-1})_{\alpha\gamma} (a\nabla\xi_\alpha)_i
  (a\nabla\xi_\gamma)_j\,,
\end{split}
\label{p-i-j}
  \end{align}
for $1\le i,j \le d$.

Let $\mbox{grad}^{\Sigma}$, $\mbox{div}^{\Sigma}$, $\Delta^{\Sigma}$, $\mbox{Hess}^{\Sigma}$ denote the 
gradient operator, the divergence operator, the Laplace-Beltrami operator and the Hessian operator on $\Sigma$,
respectively. It is direct to check that
the Levi-Civita connection and the gradient operator on $\Sigma$ 
are given by $\nabla^{\Sigma}=P\nabla^{\mathcal{M}}$ and
$\mbox{grad}^{\Sigma}=P\,\mbox{grad}^{\mathcal{M}}$, respectively. 
 In particular, for $f \in C^\infty(\Sigma)$ and let $\widetilde{f}$ be its
 extension to $\mathcal{M}$ such that $\widetilde{f} \in
 C^\infty(\mathcal{M})$ and $\widetilde{f}|_{\Sigma} = f$, we have 
\begin{align*}
  \mbox{grad}^{\Sigma} f =
  P\mbox{grad}^{\mathcal{M}}\widetilde{f} = Pa \nabla\widetilde{f}\,.
  \end{align*}
Let $\nu_g$ be the surface measure on $\Sigma$ induced from
the metric $g$ on $\mathcal{M}$.
We recall that the mean curvature vector $H$ on $\Sigma$ is defined such
that~\cite{ambrosio1996,projection_diffusion}
\begin{align}
  \int_{\Sigma} \mbox{div}^{\Sigma} \bm{v}\, d\nu_g = - \int_{\Sigma} \langle H, \bm{v}\rangle_g \,
  d\nu_g \,, 
  \label{h-div-thm}
\end{align}
for all vector fields $\bm{v}$ on $\mathcal{M}$. 

We have the following lemma, concerning the operators on $\Sigma$. 
\begin{lemma}
  Let $f \in C^{\infty}(\Sigma)$ and $\widetilde{f} \in
  C^{\infty}(\mathcal{M})$ be its extension
  to $\mathcal{M}$. $\bm{u} \in \Gamma(T\mathcal{M})$ is a tangent vector field
  on $\mathcal{M}$ and we recall the vectors $\bm{p}_i = P\bm{\sigma}_i$, $1
  \le i \le d$.
We have  
  \begin{enumerate}
\item
  $\mbox{\textnormal{div}}^{\Sigma} \bm{u} = \langle \nabla^{\mathcal{M}}_{\bm{p}_i} \bm{u},
  \bm{p}_i\rangle_g$.
\item
  $(\mbox{\textnormal{div}}^{\Sigma} \bm{p}_i)\bm{p}_i+ P\nabla^{\mathcal{M}}_{\bm{p}_i}
  \bm{p}_i= 0 \,$.
%  \label{div-identity}
\item
  $\Delta^{\Sigma} f = \sum\limits_{i=1}^d\bm{p}_i^2 \widetilde{f} + (\mbox{\textnormal{div}}^{\Sigma} \bm{p}_i) \bm{p}_i
    \widetilde{f}
    = \sum\limits_{i=1}^d\bm{p}_i^2 \widetilde{f} - \big(P
    \nabla^{\mathcal{M}}_{\bm{p}_i}
    \bm{p}_i\big) \widetilde{f}$.
\item
  $\Delta^{\Sigma} f  
  = \mbox{\textnormal{Hess}}^{\mathcal{M}} \widetilde{f}( \bm{p}_i, \bm{p}_i) + H\widetilde{f}$,
where $H$ is the mean curvature vector of the submanifold $\Sigma$.
\item
  In the special case when $g=a=\mbox{\textnormal{id}}$,
  we have $(\mbox{\textnormal{div}}^{\Sigma} \bm{p}_i)\bm{p}_i=
  P\nabla^{\mathcal{M}}_{\bm{p}_i}
  \bm{p}_i= 0$, and $\Delta^{\Sigma} = \sum\limits_{i=1}^d \bm{p}_i^2$.
  \end{enumerate}
  \label{lemma-n-1}
\end{lemma}
\begin{proof}
  The first two assertions can be directly verified. Let us prove the last three assertions.
Let $x \in \Sigma$ and assume that $\bm{v}_i$, $1 \le i \le d-k$, is an orthonormal basis of $T_x\Sigma$. 
We have $\bm{v}_i = \langle \bm{v}_i, \bm{p}_j\rangle_g \bm{p}_j$.
For the third assertion, by definition, 
\begin{align*}
    \Delta^{\Sigma} f =& \mbox{div}^{\Sigma}(\mbox{grad}^{\Sigma} f) =
  \mbox{div}^{\Sigma}\big(P\,\mbox{grad}^{\mathcal{M}}
  \widetilde{f}\,\big) = \mbox{div}^{\Sigma} \big(\langle
  \mbox{grad}^{\mathcal{M}}\widetilde{f},
    \bm{p}_i\rangle_g\, \bm{p}_i\big) \\
    =& \mbox{div}^{\Sigma} \big((\bm{p}_i\widetilde{f}\,)\bm{p}_i\big) =
    \sum_{i=1}^d\bm{p}_i^2 \widetilde{f} + (\mbox{div}^{\Sigma} \bm{p}_i) \bm{p}_i
    \widetilde{f}
    = \sum_{i=1}^d\bm{p}_i^2 \widetilde{f} - \big(P
    \nabla^{\mathcal{M}}_{\bm{p}_i}
    \bm{p}_i\big) \widetilde{f}\,,
\end{align*}
where the second assertion has been used in the last equality.

For the fourth assertion, 
starting from the third assertion, using the definition of
$\mbox{Hess}^{\mathcal{M}}$ in (\ref{hessian-chart}),
and applying Proposition~\ref{prop-mean} below, we obtain
\begin{align*}
  \Delta^{\Sigma} f = &
   \sum_{i=1}^d \bm{p}_i^2 f - P \nabla^{\mathcal{M}}_{\bm{p}_i} \bm{p}_i f\\
   =& \mbox{Hess}^{\mathcal{M}} \widetilde{f}( \bm{p}_i, \bm{p}_i) + 
  \big[(I-P)\nabla^{\mathcal{M}}_{\bm{p}_i} \bm{p}_i\big] \widetilde{f}
= \mbox{Hess}^{\mathcal{M}} \widetilde{f}( \bm{p}_i, \bm{p}_i) + H\widetilde{f}\,.
\end{align*}

For the last assertion, when $g=\mbox{\textnormal{id}}$,
we have $\Gamma_{ij}^l \equiv 0$, for $\forall 1\le i,j,l \le d$.  
  Also, it follows from (\ref{p-i-j}) that $P_{i,j}=P_{ij} = P_{ji}$ and $P_{il}
  P_{lj} = P_{ij}$. 
  We obtain  
  \begin{align*}
    &(\mbox{div}^{\Sigma} \bm{p}_i) \bm{p}_i\\
    =& \langle \nabla^{\mathcal{M}}_{\bm{p}_j}\bm{p}_i,
    \bm{p}_j\rangle_g\, \bm{p}_i\\
    =& P_{j, l}P_{j,j'} P_{i,i'} \langle
    \nabla^{\mathcal{M}}_{\bm{e}_l}
    \big(P_{i, r}\bm{e}_r\big),
    \bm{e}_{j'} \rangle_g\, \bm{e}_{i'}\\
    =& P_{lr} \frac{\partial P_{ir}}{\partial x_l}  P_{ii'}\,\bm{e}_{i'}\\
  =&\Big[\frac{\partial P_{il}}{\partial x_l}  P_{ii'} -  \frac{\partial
P_{lr}}{\partial x_l} P_{ri'}\Big]\,\bm{e}_{i'} \\
=& 0\,,
  \end{align*}
  and the other assertions follow accordingly.
\end{proof}
\begin{prop}
  Let $H$ be the mean curvature vector defined in (\ref{h-div-thm}) on the submanifold $\Sigma$. 
  We have 
  \begin{align}
    \begin{split}
    H = &(I-P)\nabla^{\mathcal{M}}_{\bm{p}_i} \bm{p}_{i} \\
    =& 
      -(\Psi^{-1})_{\alpha\gamma}
\bigg[\frac{1}{2}
 (Pa)_{ij} (a\nabla\xi_\alpha)_l\frac{\partial (a^{-1})_{ij}}{\partial x_l} +
 P_{il}\frac{\partial (a\nabla\xi_\alpha)_l}{\partial x_i} \bigg] 
       a\nabla\xi_\gamma \,.
\end{split}
    \label{prop-mean-eqn-1}
  \end{align}
  In the special case when $g=a=\mbox{\textnormal{id}}$, we have 
  \begin{align}
    H = P_{jl}\frac{\partial P_{il}}{\partial x_j}\bm{e}_i= 
    -\big[(\Psi^{-1})_{\alpha\gamma}  P_{ij}\, \partial^2_{ij}\xi_\alpha\big]
    \nabla\xi_\gamma\,.
    \label{prop-mean-eqn-2}
  \end{align}
  \label{prop-mean}
\end{prop}
\begin{proof}
  Given a tangent vector field $\bm{v}$ on $\mathcal{M}$, 
  from the definition of $P$ we have $\bm{v} = P\bm{v} + (\Psi^{-1})_{\alpha\gamma} \langle \bm{v},
  a\nabla\xi_\gamma\rangle_g\, a\nabla\xi_\alpha$.
  Since $P\bm{v}$ is a tangent vector field on $\Sigma$, using (\ref{h-div-thm}) and the
  divergence theorem on $\Sigma$, we know 
  \begin{align}
    \int_{\Sigma} \langle H, \bm{v}\rangle_g\, d\nu_g 
    = -\int_{\Sigma} \mbox{div}^{\Sigma} \big[(I - P)\bm{v}\big]\, d\nu_g = 
    -\int_{\Sigma} \mbox{div}^{\Sigma} \big[(\Psi^{-1})_{\alpha\gamma}
    \langle \bm{v}, a\nabla\xi_\gamma\rangle_g\, a\nabla\xi_\alpha\big]\, d\nu_g\,.
  \label{mean-curvature-eqn}
  \end{align}
  For the first expression, we notice that $\langle 
  (I - P)\bm{v}, \bm{p}_i\rangle_g \equiv 0$, $1 \le i
  \le d$. Applying Lemma~\ref{lemma-n-1}, we have 
  \begin{align*}
    &-\int_{\Sigma} \mbox{div}^{\Sigma} \big[(I - P)\bm{v}\big] d\nu_g \\
    =& -\int_{\Sigma} \langle \nabla^{\mathcal{M}}_{\bm{p}_i} \big[(I - P)\bm{v}\big], \bm{p}_i\rangle_g\, d\nu_g \\
    =& -\int_{\Sigma} \bm{p}_i\langle (I -
  P)\bm{v}, \bm{p}_i\rangle_g\, d\nu_g +
  \int_{\Sigma} \langle (I - P)\bm{v}, \nabla^{\mathcal{M}}_{\bm{p}_i}\bm{p}_i
  \rangle_g\, d\nu_g 
    \\
    =& \int_{\Sigma} \langle \bm{v}, (I -
    P)\nabla^{\mathcal{M}}_{\bm{p}_i}\bm{p}_i\rangle_g\, d\nu_g \,.
  \end{align*}
  Comparing the last equality above with (\ref{mean-curvature-eqn}), we
  conclude that $H=(I - P)\nabla^{\mathcal{M}}_{\bm{p}_i}\bm{p}_i$.
  
    For the second expression, we notice that $\langle
    a\nabla\xi_\alpha, \bm{p}_i\rangle_g=0$, 
and also recall the expressions (\ref{gamma-det-exp}), (\ref{psi-ij}) and (\ref{p-i-j}).
Applying Lemma~\ref{lemma-n-1}, integrating by parts, and noticing the cancellation
of some terms, we can derive 
  \begin{align*}
    &\mbox{div}^{\Sigma} \big[(\Psi^{-1})_{\alpha\gamma}
    \langle \bm{v}, a\nabla\xi_\gamma\rangle_g\, a\nabla\xi_\alpha\big]\\
    =& \langle \nabla^{\mathcal{M}}_{\bm{p}_i} 
    \big[(\Psi^{-1})_{\alpha\gamma}
      \langle \bm{v}, a\nabla\xi_\gamma\rangle_g\, a\nabla\xi_\alpha\big], \bm{p}_i\rangle_g\, \\
=& (\Psi^{-1})_{\alpha\gamma}
      \langle \bm{v}, a\nabla\xi_\gamma\rangle_g\, 
       \langle
       \nabla^{\mathcal{M}}_{\bm{p}_i} (a\nabla\xi_\alpha), \bm{p}_i\rangle_g\, \\
=& (\Psi^{-1})_{\alpha\gamma}
      \langle \bm{v}, a\nabla\xi_\gamma\rangle_g\, 
      P_{i, j}P_{i,l} \langle \nabla^{\mathcal{M}}_{\bm{e}_j}
      \big((a\nabla\xi_\alpha)_r \bm{e}_r\big), \bm{e}_l\rangle_g\, \\
  =& (\Psi^{-1})_{\alpha\gamma}
      \langle \bm{v}, a\nabla\xi_\gamma\rangle_g\, 
      (Pa)_{jl}
      \Big[(a\nabla\xi_\alpha)_r \Gamma^i_{jr}(a^{-1})_{il} +
\frac{\partial (a\nabla\xi_\alpha)_r}{\partial x_j} (a^{-1})_{lr}\Big] \\
=& (\Psi^{-1})_{\alpha\gamma}
      \langle \bm{v}, a\nabla\xi_\gamma\rangle_g\, 
      (Pa)_{ij}
\Big[\frac{1}{2}
  (a\nabla\xi_\alpha)_l\frac{\partial (a^{-1})_{ij}}{\partial x_l} +
\frac{\partial (a\nabla\xi_\alpha)_l}{\partial x_i} (a^{-1})_{lj}\Big]\,.
  \end{align*}
  The second identity in (\ref{prop-mean-eqn-1}) is obtained after comparing the above expression with (\ref{mean-curvature-eqn}) .

In the case $g=a=\mbox{\textnormal{id}}$, we have $\Gamma_{il}^r \equiv 0$, 
$1 \le i,l,r \le d$. 
It follows that
\begin{align*}
  \mbox{div}^{\Sigma}
  \big[(\Psi^{-1})_{\alpha\gamma}
    \langle \bm{v}, a\nabla\xi_\gamma\rangle_g\, a\nabla\xi_\alpha\big]
  = (\Psi^{-1})_{\alpha\gamma}
      \langle \bm{v}, a\nabla\xi_\gamma\rangle_g\,
      P_{ij}\,\partial^2_{ij}\xi_\alpha 
\end{align*}
and we obtain that 
$H=-\big[(\Psi^{-1})_{\alpha\gamma}  P_{ij} \, \partial^2_{ij}\xi_\alpha\big] \nabla\xi_\gamma$.
Using (\ref{p-ij}) and (\ref{p-fact}), we have 
    \begin{align*}
      P_{jl}\frac{\partial P_{il}}{\partial x_j}\bm{e}_i
      = -P_{jl}\frac{\partial\big((\Psi^{-1})_{\alpha\gamma}
	\partial_l\xi_\alpha
      \partial_i\xi_\gamma\big)}{\partial x_j} \bm{e}_i
      =- \big[(\Psi^{-1})_{\alpha\gamma}  P_{jl}\,
      \partial^2_{jl}\xi_\alpha\, \partial_i\xi_\gamma\big]
      \bm{e}_i
      =H\,,
    \end{align*}
and therefore the first expression in (\ref{prop-mean-eqn-2}) holds as well.
\end{proof}
Next, we study the Laplace-Beltrami operator $\Delta^{\Sigma}$ on the submanifold
$\Sigma$. Clearly, $\Delta^{\Sigma}$ is self-adjoint and, similar to (\ref{integrate-by-part}), 
we have the integration by parts formula on $\Sigma$ with respect to the
measure $\nu_g$, as 
\begin{align}
  \int_{\Sigma} (\Delta^{\Sigma} f_1)\,f_2\,d\nu_g = - \int_{\Sigma}
  \langle \mbox{grad}^{\Sigma} f_1, \mbox{grad}^{\Sigma} f_2\rangle_g\, d\nu_g
  = \int_{\Sigma} (\Delta^{\Sigma} f_2)f_1\,d\nu_g\,,  
  \label{integrate-by-part-sigma}
\end{align}
for $\forall f_1,\, f_2 \in C^\infty(\Sigma)$.
The expression of $\Delta^{\Sigma}$ can be computed explicitly and this is the
content of the following proposition. 
\begin{prop}
  Let $\Sigma$ be the submanifold of $\mathcal{M}$ defined in
  (\ref{levelset-sigma}), $P$ be the projection matrix in (\ref{p-ij}),
  and $\Delta^{\Sigma}$ be the Laplace-Beltrami operator on $\Sigma$.   We have 
  \begin{align}
    \Delta^{\Sigma}
    =& 
    (Pa)_{ij} \frac{\partial^2 }{\partial x_i\partial x_j} 
    + \Big[\frac{\partial (Pa)_{ij}}{\partial x_j} + \frac{1}{2} (Pa)_{ij} 
      \frac{\partial \ln \big((\det a)^{-1}\det(\nabla\xi^T a \nabla\xi)\big)}{\partial
x_j}\Big] \frac{\partial }{\partial x_i}\,.
\label{prop-laplace-a}
  \end{align}
  In the special case when $g=a=\mbox{\normalfont{id}}$, we have 
  \begin{align}
    \begin{split}
       \Delta^{\Sigma} =& \sum_{i=1}^d \bm{p}^2_i=  
    P_{ij} \frac{\partial^2 }{\partial x_i\partial x_j}
+ P_{lj}\frac{\partial P_{li}}{\partial x_j}
\frac{\partial}{\partial x_i}\\
=& P_{ij} \frac{\partial^2 }{\partial x_i\partial x_j}
      + H_i \frac{\partial}{\partial x_i}\,,
\end{split}
\label{prop-laplace-rn-eqn}
  \end{align}
  where $H = H_i \bm{e}_i$ is the mean curvature vector of the submanifold
  $\Sigma$.  
  \label{laplaceN-state}
\end{prop}
\begin{proof}
  Let $f \in C^{\infty}(\Sigma)$ and $\widetilde{f} \in
  C^{\infty}(\mathcal{M})$ be its extension to $\mathcal{M}$.
  Using Lemma~\ref{lemma-n-1} and Proposition~\ref{prop-mean}, we have 
  \begin{align*}
    \Delta^{\Sigma} f  
    =& \mbox{\textnormal{Hess}}^{\mathcal{M}} \widetilde{f}( \bm{p}_r, \bm{p}_r) + H\widetilde{f}\\
  =& P_{r, j}P_{r,l} \mbox{\textnormal{Hess}}^{\mathcal{M}} \widetilde{f}(
    \bm{e}_j,
    \bm{e}_l) + H\widetilde{f}\\
    =& (Pa)_{jl} \Big(\frac{\partial^2 \widetilde{f}}{\partial
x_j\partial x_l} - \Gamma_{jl}^i \frac{\partial \widetilde{f}}{\partial
x_i}\Big)\\
&- 
(\Psi^{-1})_{\alpha\gamma}
\Big[\frac{1}{2}
      (Pa)_{jl} (a\nabla\xi_\gamma)_r\frac{\partial (a^{-1})_{jl}}{\partial x_r} +
      P_{lr}
\frac{\partial (a\nabla\xi_\gamma)_r}{\partial x_l} \Big] 
(a\nabla\xi_\alpha)_i \frac{\partial \widetilde{f}}{\partial x_i} \,.
  \end{align*}
  Notice that we have already obtained the coefficients of the second order
  derivative terms. For the terms of the first order derivatives, let us denote 
  \begin{align}
    \begin{split}
  I_1=  &-(Pa)_{jl} \Gamma_{jl}^i\\
    I_2 =& 
-\frac{1}{2}
    (\Psi^{-1})_{\alpha\gamma}
      (Pa)_{jl}
(a\nabla\xi_\alpha)_i (a\nabla\xi_\gamma)_r\frac{\partial (a^{-1})_{jl}}{\partial x_r}\\
I_3 = &- (\Psi^{-1})_{\alpha\gamma}
P_{lr} (a\nabla\xi_\alpha)_i \frac{\partial
      (a\nabla\xi_\gamma)_r}{\partial x_l}\,. 
    \end{split}
      \label{i-123}
    \end{align}
    Using the expression of $Pa$ in (\ref{p-i-j}), the property $Pa\nabla\xi_\gamma = 0$, and integrating by parts,
    we easily obtain
    \begin{align}
      \begin{split}
      I_2=& \frac{1}{2}\big((Pa)_{ir} - a_{ir}\big) (Pa)_{jl} \frac{\partial
  (a^{-1})_{jl}}{\partial x_r} \\
  I_3 =& \frac{\partial P_{lr}}{\partial x_l}\big(a_{ir} - (Pa)_{ir}\big)\,.
\end{split}
  \label{i-23}
    \end{align}
    For $I_1$, direct calculation using (\ref{gamma-det-exp}) gives 
    \begin{align}
      \begin{split}
      I_1 =& - \frac{1}{2}
    (Pa)_{jl}
    a_{ir}\Big(\frac{\partial (a^{-1})_{lr}}{\partial x_j}
    + \frac{\partial (a^{-1})_{jr}}{\partial x_l} - \frac{\partial
    (a^{-1})_{jl}}{\partial x_r}\Big)\\
  = &- (Pa)_{jl}  a_{ir}\frac{\partial (a^{-1})_{lr}}{\partial x_j}
    + \frac{1}{2}
    (Pa)_{jl}
  a_{ir} \frac{\partial (a^{-1})_{jl}}{\partial x_r} \\
  =&  \frac{\partial (Pa)_{ij}}{\partial x_j} 
- \frac{\partial P_{jr}}{\partial x_j} a_{ir}
    + \frac{1}{2}
    (Pa)_{jl}
  a_{ir} \frac{\partial (a^{-1})_{jl}}{\partial x_r} \,.
\end{split}
\label{i-1}
    \end{align}
    Therefore, 
    \begin{align*}
      I_1 + I_2 + I_3
      = &
\frac{\partial (Pa)_{ij}}{\partial x_j} 
-\frac{\partial P_{lr}}{\partial x_l}(Pa)_{ir}
    + \frac{1}{2}
    (Pa)_{jl}
  a_{ir} \frac{\partial (a^{-1})_{jl}}{\partial x_r} + 
\frac{1}{2}\big((Pa)_{ir} - a_{ir}\big) (Pa)_{jl} \frac{\partial
  (a^{-1})_{jl}}{\partial x_r} \\
=  & \frac{\partial (Pa)_{ij}}{\partial x_j} 
-\frac{\partial P_{lr}}{\partial x_l}(Pa)_{ir}
+ \frac{1}{2}(Pa)_{ir} (Pa)_{jl} \frac{\partial (a^{-1})_{jl}}{\partial x_r} \,.
    \end{align*}
    Applying Lemma~\ref{lemma-identity-in-i12} below to handle the last term above, we conclude
  \begin{align*}
I_1 + I_2 + I_3 =&
 \frac{\partial (Pa)_{ij}}{\partial x_j} 
 -\frac{1}{2}(Pa)_{ir}\frac{\partial \ln\mbox{\normalfont{det}}\,a}{\partial x_r}
 +\frac{1}{2}(Pa)_{ir}\frac{\partial \ln\mbox{\normalfont{det}}\,\Psi}{\partial x_r}\,.
\end{align*}

Finally, when $g=a=\mbox{\normalfont{id}}$, applying Lemma~\ref{lemma-n-1}, we can obtain
  \begin{align*}
\Delta^{\Sigma} f 
=& \mbox{\textnormal{Hess}}^{\mathcal{M}} \widetilde{f}( \bm{p}_i, \bm{p}_i) +
    H\widetilde{f}\\
  =& P_{l, i}P_{l,j}
\Big(\frac{\partial^2 \widetilde{f}}{\partial x_i\partial x_j} - \Gamma_{ij}^l \frac{\partial \widetilde{f}}{\partial x_l}\Big)
  + H\widetilde{f}\\
  =& P_{ij}
  \frac{\partial^2 \widetilde{f}}{\partial x_i\partial x_j} + H_i
  \frac{\partial\widetilde{f}}{\partial x_i}\,.
\end{align*}
The other equality in (\ref{prop-laplace-rn-eqn}) follows from Proposition~\ref{prop-mean}.
\end{proof}
We point out that the proof of Proposition~\ref{laplaceN-state}
is indeed valid for a general Riemannian manifold $\mathcal{M}$ and its level set $\Sigma$ as well. In this case, 
(\ref{prop-laplace-a}) holds true on a local coordinate of the manifold $\mathcal{M}$. 

The following identity has been used in the above proof, and will be
useful in Appendix~\ref{app-sec-numerical} as well.
\begin{lemma}
  \begin{align*}
     \frac{1}{2}(Pa)_{ir}  (Pa)_{jl} \frac{\partial (a^{-1})_{jl}}{\partial x_r} 
     = -\frac{1}{2}(Pa)_{ir}\frac{\partial \ln\mbox{\normalfont{det}}\,a}{\partial x_r}
+ (Pa)_{ir}  
\frac{\partial P_{lr}}{\partial x_l}
+ \frac{1}{2}(Pa)_{ir}  \frac{\partial \ln\mbox{\normalfont{det}}\Psi}{\partial x_r} \,.
    \end{align*}
    \label{lemma-identity-in-i12}
\end{lemma}
\begin{proof}
  Using the expression of $Pa$ in (\ref{p-i-j}), the relations 
  \begin{align*}
Pa\nabla\xi_\gamma = 0\,,\quad
  \frac{\partial \ln \det a}{\partial x_r} = (a^{-1})_{jl}
\frac{\partial a_{jl}}{\partial x_r}\,, \quad 
  \frac{\partial \ln \det \Psi}{\partial x_r} = (\Psi^{-1})_{\alpha\gamma}
\frac{\partial \Psi_{\alpha\gamma}}{\partial x_r}\,, 
\end{align*}
and the integration by parts formula, we can compute
  \begin{align*}
    & \frac{1}{2}(Pa)_{ir}  (Pa)_{jl} \frac{\partial (a^{-1})_{jl}}{\partial x_r} \\
  = &\frac{1}{2}(Pa)_{ir} \Big(a_{jl} - 
  (\Psi^{-1})_{\alpha\gamma} (a\nabla\xi_\alpha)_j (a\nabla\xi_\gamma)_l\Big)
  \frac{\partial (a^{-1})_{jl}}{\partial x_r} \\
  =&
-\frac{1}{2}(Pa)_{ir}\frac{\partial \ln\mbox{det}\,a}{\partial x_r}
- \frac{1}{2}(Pa)_{ir}  
(\Psi^{-1})_{\alpha\gamma} (a\nabla\xi_\alpha)_l\,\partial^2_{lr}\xi_\gamma 
+ \frac{1}{2}(Pa)_{ir}  
(\Psi^{-1})_{\alpha\gamma} \partial_l\xi_\alpha\,
\frac{\partial (a\nabla\xi_{\gamma})_{l}}{\partial x_r} \\
  =&
-\frac{1}{2}(Pa)_{ir}\frac{\partial \ln\mbox{det}\,a}{\partial x_r}
- (Pa)_{ir}  
(\Psi^{-1})_{\alpha\gamma} (a\nabla\xi_\alpha)_l\,\partial^2_{lr}\xi_\gamma 
+ \frac{1}{2}(Pa)_{ir}  \frac{\partial \ln\mbox{det}\Psi}{\partial x_r} \\
  =&
-\frac{1}{2}(Pa)_{ir}\frac{\partial \ln\mbox{det}\,a}{\partial x_r}
+ (Pa)_{ir}  
\frac{\partial P_{lr}}{\partial x_l}
+ \frac{1}{2}(Pa)_{ir}  \frac{\partial \ln\mbox{det}\Psi}{\partial x_r} \,.
    \end{align*}
  \end{proof}
  We conclude this section with the proof of Corollary~\ref{corollary-on-non-reversible} in Section~\ref{sec-n-rn}.
\begin{proof}[Proof of Corollary~\ref{corollary-on-non-reversible}]
  Notice that the infinitesimal generator of (\ref{dynamics-1-submanifold-j})
  can be written as 
  \begin{align*}
    \mathcal{L}^{\bm{J}} = J_i\frac{\partial}{\partial x_i} + \mathcal{L}\,,
  \end{align*}
  where $\mathcal{L}$ is the infinitesimal generator of (\ref{dynamics-1-submanifold}).
  Using the fact $\bm{J} \in T_x \Sigma$, the same argument of Proposition~\ref{sde-inv-mu} implies that (\ref{dynamics-1-submanifold-j}) evolves on $\Sigma$ as well. 
  Since $\mu_1$ is invariant with respect to $\mathcal{L}$, to show
  the SDE (\ref{dynamics-1-submanifold-j}) has the same invariant measure, it is enough to verify that 
  \begin{align}
    \mbox{div}^{\Sigma}\Big\{\bm{J} \exp\Big[-\beta \Big(U+\frac{1}{2\beta} \ln
    \frac{\det (\nabla\xi^T a \nabla\xi)}{\det a}\Big)\Big]\Big\} = 0\,, \quad
    \forall~x \in \Sigma\,,
    \label{div-j-to-check}
  \end{align}
  where we have used the expression of $\mu_1$ in (\ref{mu1-mu2-g}). Applying
  the formula of $\mbox{div}^{\Sigma}$ in Lemma~\ref{lemma-n-1}, 
  we can compute the right hand side of (\ref{div-j-to-check}), as 
  \begin{align*}
    \begin{split}
    &\mbox{div}^{\Sigma}\Big\{\bm{J} \exp\Big[-\beta \Big(U+\frac{1}{2\beta} \ln
    \frac{\det (\nabla\xi^T a \nabla\xi)}{\det a}\Big)\Big]\Big\} \\
    =&\langle \nabla^{\mathcal{M}}_{\bm{p}_j}
\Big\{\exp\Big[-\beta \Big(U+\frac{1}{2\beta} \ln
    \frac{\det (\nabla\xi^T a \nabla\xi)}{\det a}\Big)\Big]J_i\bm{e}_i\Big\},
    \bm{p}_j\rangle_g \\
      =& P_{j,l}P_{j,r}\langle \nabla^{\mathcal{M}}_{\bm{e}_l}
\Big\{\exp\Big[-\beta \Big(U+\frac{1}{2\beta} \ln \frac{\det (\nabla\xi^T a \nabla\xi)}{\det a}\Big)\Big]J_i\bm{e}_i\Big\},
      \bm{e}_r\rangle_g\,. 
    \end{split}
  \end{align*}
  which implies that (\ref{div-j-to-check}) is equivalent to 
  \begin{align*}
    0 =& (Pa)_{lr} \frac{\partial J_i}{\partial x_l} (a^{-1})_{ir} +
    (Pa)_{lr}J_i\, \Gamma_{li}^{r'} (a^{-1})_{r'r} - \beta (Pa)_{lr} J_i (a^{-1})_{ir} \frac{\partial U}{\partial x_l} \\
    & - \frac{1}{2} (Pa)_{lr} J_i (a^{-1})_{ir} \frac{\partial}{\partial
    x_l}\Big[ \ln \frac{\det (\nabla\xi^T a \nabla\xi)}{\det a} \Big]\\
=& P_{li} \frac{\partial J_i}{\partial x_l} +
    P_{lr'}J_i\, \Gamma_{li}^{r'} - \beta P_{li} J_i \frac{\partial U}{\partial x_l} 
     - \frac{1}{2} P_{li} J_i \frac{\partial}{\partial x_l}\Big[ \ln
     \frac{\det (\nabla\xi^T a \nabla\xi)}{\det a} \Big]\,,
  \end{align*}
  where $\Gamma_{li}^{r'}$ are the Christoffel's symbols satisfying $\nabla^{\mathcal{M}}_{\bm{e}_l} \bm{e}_i =
  \Gamma^{r'}_{li} \bm{e}_{r'}$. Using the expression (\ref{gamma-det-exp}) of $\Gamma_{li}^{r'}$, the fact $J_i = P_{ij} J_j$, and
  Lemma~\ref{lemma-identity-in-i12}, we can
  further simplify the above equation 
  and obtain
  \begin{align*}
    &P_{lr'}J_i\, \Gamma_{li}^{r'} - \frac{1}{2} P_{li} J_i
    \frac{\partial}{\partial x_l}\Big[ \ln \frac{\det (\nabla\xi^T a \nabla\xi)}{\det a} \Big]\\
    =& \frac{1}{2} J_i(Pa)_{lr} \frac{\partial (a^{-1})_{lr}}{\partial x_i} -
    \frac{1}{2} P_{li} J_i \frac{\partial}{\partial x_l}\Big[ \ln \frac{\det
    (\nabla\xi^T a \nabla\xi)}{\det a} \Big]\\
    =& J_j \frac{\partial P_{ij}}{\partial x_i}\,.
  \end{align*}
  Therefore, we see that 
  (\ref{div-j-to-check}) is equivalent to the condition (\ref{j-assumption}).
\end{proof}

\section{Proofs in Section~\ref{sec-numerical-scheme-mu1}}
\label{app-sec-numerical}
In this section, we collect proofs of the various results in Section~\ref{sec-numerical-scheme-mu1}. 

First, we prove Proposition~\ref{prop-map-phi-1st-2nd-derivative}, which concerns the properties of the flow map $\Theta$ defined in 
(\ref{phi-map}), (\ref{fun-cap-f}), and (\ref{theta-map}).
While the approach of the proof is similar to the one in~\cite{fatkullin2010},
here we consider the specific function $F$ in (\ref{fun-cap-f}) 
and we will provide full details of the derivations.

\begin{proof}[Proof of Proposition~\ref{prop-map-phi-1st-2nd-derivative}]
      In this proof, we will always assume $x \in \Sigma$. 
      For a function which only depends on the state and is evaluated at $x$,
      we will often omit its argument in order to keep the notations simple.
      Also notice that, repeated indices other than $l$ and $l'$
      indicate that they are summed up, while for the indices $l$, $l'$
      we assume that they are fixed by default unless the summation operator is used explicitly.

      Since $\nabla
      F = 0$ on $\Sigma$, from the equation (\ref{phi-map}) we know that
      $\varphi(x,s) \equiv x$, $\forall s \ge 0$. 
      Let us Denote by $\nabla^2 F$ the Hessian matrix (on the standard Euclidean space)
      of the function $F$ in (\ref{fun-cap-f}), i.e., 
     $\nabla^2 F =(\partial^2_{ij} F)_{1 \le i, j \le d}$.
      Since $\xi(x) = \bm{0} \in \mathbb{R}^k$, direct calculation gives 
    \begin{align}
      (a\nabla^2 F)_{ij} = a_{ir}\frac{\partial^2 F}{\partial
      x_{r}\partial x_{j}} = 
      (a\nabla\xi\nabla \xi^T)_{ij} \,, \qquad 1 \le i, j \le d\,.
      \label{a-d2f}
    \end{align}
    Meanwhile, it is straightforward to verify that $a\nabla^2 F$ satisfies 
    \begin{align*}
      \begin{split}
	& \langle a \nabla^2 F \bm{u}, \bm{v}\rangle_g = \langle \bm{u}, a
      \nabla^2 F \bm{v}\rangle_g  \,,\quad \forall~\bm{u}, \bm{v} \in \mathbb{R}^d\,, \\
      & \langle a \nabla^2F \bm{u}, \bm{u}\rangle_g = |\nabla\xi^T\bm{u}|^2 \ge
      0\,, \quad \forall~\bm{u}  \in \mathbb{R}^d\,,\\
      &(a\nabla^2 F) \bm{u}= a\nabla\xi\nabla\xi^T \bm{u} = 0\,,
      \quad\,\forall~\bm{u} \in T_x\Sigma\,.
\end{split}
    \end{align*}
    Therefore, we can assume that $a\nabla^2F$ has real (non-negative) eigenvalues 
	  \begin{align}
	  \lambda_1=\lambda_2= \cdots=\lambda_{d-k} = 0 <
    \lambda_{d-k+1}\le \cdots \le\lambda_d\,, 
	    \label{eigenvalue-sequence}
	  \end{align}
     and the corresponding eigenvectors, denoted by  
    $\bm{v}_i = (v_{i1}, v_{i2}, \cdots, v_{id})^T$, $1 \le i \le d$, 
    are orthonormal with respect to the inner product
    $\langle\cdot,\cdot\rangle_g$ in (\ref{ip-a}), such that
$\bm{v}_1, \bm{v}_2, \cdots, \bm{v}_{d-k} \in T_x\Sigma$.

The projection matrix $P$ in (\ref{p-ij}) can be expressed using the vectors 
$\bm{v}_i$ as
	  \begin{align}
	    P_{ij} = \sum_{l=1}^{d-k} v_{li} (a^{-1})_{jr} v_{lr}\,, \quad 1 \le i, j \le d\,,
	    \label{p-by-v}
	  \end{align}
	  and we have 
	  \begin{align}
	    \sum_{l=1}^{d-k} v_{li}v_{lj} = (Pa)_{ij}\,,\quad  a_{ij} - (Pa)_{ij}
	    =\sum_{l=d-k+1}^{d} v_{li}v_{lj}\,.
	    \label{v-pa-and-remain}
	  \end{align}
    It is also a simple fact that the eigenvalues of the $k\times k$ matrix
    $\Psi=\nabla\xi^T a\nabla\xi$ are $\lambda_{d-k+1}$, $\lambda_{d-k+2}$,
    $\cdots$, $\lambda_{d}$, 
    with the corresponding eigenvectors given by $\nabla\xi^T \bm{v}_{d-k+1}$,
    $\nabla\xi^T\bm{v}_{d-k+2}$, $\cdots$, $\nabla\xi^T\bm{v}_d$. In
    particular, this implies
      \begin{align}
\prod\limits_{i=d-k+1}^d \lambda_i=\det (\nabla\xi^T a\nabla\xi) = \det \Psi\,.
	\label{det-equals-eigenvalue-product}
      \end{align}

      In the following, we study the ODE (\ref{phi-map}) using the eigenvectors $\bm{v}_i$.
    Differentiating the ODE (\ref{phi-map}) twice, using the facts that 
    $\varphi(x,s)\equiv x$, $\forall s \ge 0$, and $\nabla F = 0$ on $\Sigma$, we obtain     
	  \begin{align}
      \begin{split}
	\frac{d}{ds} \frac{\partial \varphi_i}{\partial x_j}(x,s) =&
	-\Big(a_{ir'}\frac{\partial^2 F}{\partial x_{r'}\partial x_{i'}}\Big) \,
      \frac{\partial\, \varphi_{i'}}{\partial x_j} (x,s)\\
      \frac{d}{ds} \frac{\partial^2 \varphi_i}{\partial x_j\partial
      x_r}(x,s) =&
	-\bigg(2\frac{\partial a_{ir'}}{\partial x_{i'}} \frac{\partial^2
	  F}{\partial x_{r'} \partial x_{j'}} + a_{ir'}\frac{\partial^3
	  F}{\partial x_{r'}\partial x_{i'}\partial
	    x_{j'}}\bigg) 
      \frac{\partial\, \varphi_{i'}}{\partial x_j} (x,s)
      \frac{\partial\, \varphi_{j'}}{\partial x_r} (x,s)\\
      &-\Big(a_{ir'}\frac{\partial^2 F}{\partial x_{r'}\partial x_{i'}}\Big)
      \frac{\partial^2\, \varphi_{i'}}{\partial x_j \partial x_r}(x,s)\,,
    \end{split}
    \label{derivative-of-flow-map}
    \end{align}
for $s \ge 0$ and $1 \le i, j, r \le d$.
      \begin{enumerate}
	\item
    The first equation of (\ref{derivative-of-flow-map}) implies 
    \begin{align}
      \frac{d}{ds}\Big(v_{lj}\frac{\partial \varphi_i}{\partial
      x_j}(x,s)\Big) =&
	-\Big(a_{ir'}\frac{\partial^2 F}{\partial x_{r'}\partial x_{i'}}\Big)
      \Big( v_{lj} \frac{\partial\, \varphi_{i'}}{\partial
      x_j}(x,s)\Big)\,,\quad 1 \le l \le d\,.
      \label{1st-derivative-ode-1}
    \end{align}
    Since $\varphi(\cdot, 0)$ is the identity map, we have 
    \begin{align}
     v_{lj} \frac{\partial\varphi_{i}}{\partial
    x_j}(x,0) = v_{li}, \quad \mbox{at}~ s=0\,. 
      \label{1st-derivative-ode-2}
  \end{align}
  Because $\bm{v}_l$ is the eigenvector of $a\nabla^2 F$, we can 
  directly solve the solution of (\ref{1st-derivative-ode-1})-(\ref{1st-derivative-ode-2}) and obtain
    \begin{align}
       v_{lj} \frac{\partial\, \varphi_{i}}{\partial x_j}(x,s) =
      e^{-\lambda_l s} v_{li}\, \Longleftrightarrow 
\frac{\partial\, \varphi_{i}}{\partial x_j}(x,s) =
\sum_{l=1}^d e^{-\lambda_l s} v_{li} (a^{-1})_{jr}v_{lr}\,, \qquad \forall~s \ge 0\,,
\label{solution-of-1st-derivative-ode}
    \end{align}
    for $1 \le i, j \le d$.
	  Sending $s\rightarrow +\infty$, using (\ref{eigenvalue-sequence}) and (\ref{p-by-v}), we obtain 
    \begin{align}
      \frac{\partial \Theta_{i}}{\partial x_j} = \lim_{s\rightarrow +\infty} 
      \frac{\partial\, \varphi_{i}}{\partial x_j}(x,s) = \sum_{l=1}^{d-k} v_{li} 
      (a^{-1})_{jr}v_{lr} =P_{ij}\,.
      \label{theta-1st-equals-p-in-proof}
    \end{align}
	\item
	  We proceed to compute $a_{jr}\frac{\partial^2\Theta_i}{\partial
	  x_j\partial x_r}$, $1 \le i\le d$. For this purpose,
    let us define
    \begin{align*}
      \begin{split}
	& A_{l}(x,s) =  (a^{-1})_{ij'} v_{lj'} a_{jr}
	\frac{\partial^2 \varphi_{i}}{\partial x_j\partial x_r}(x,s) \,,\quad 1 \le l
      \le d\,,\\
      \Longleftrightarrow\quad &
       a_{jr}\frac{\partial^2 \varphi_{i}}{\partial x_j\partial
       x_r}(x,s) = \sum_{l=1}^d v_{li} A_{l}(x,s) \,. 
	\end{split}
    \end{align*}
    Using the second equation of (\ref{derivative-of-flow-map}), 
    the solution (\ref{solution-of-1st-derivative-ode}), and the orthogonality
    of the eigenvectors, we can obtain
    \begin{align*}
      \frac{d A_{l}}{ds}(x,s)
      =
      - \sum_{l'=1}^d\bigg[
	2\frac{\partial a_{ir}}{\partial x_{i'}}\frac{\partial^2 F}{\partial
	x_{r}\partial x_{j}}(a^{-1})_{ir'} +       
      \frac{\partial^3 F}{\partial x_{r'}\partial x_{i'}\partial x_{j}}\bigg]
      v_{l'i'}v_{l'j} v_{lr'}\, e^{-2\lambda_{l'}s} 
      - \lambda_l A_l(x,s)\,, 
    \end{align*}
for $1 \le l \le d$, from which we get 
    \begin{align*}
      &P_{ii'} 
      a_{jr}\frac{\partial^2 \varphi_{i'}}{\partial x_j\partial x_r}(x,s) \\
      =& \sum_{l=1}^{d-k} v_{li}A_l(x,s) \\
      = &
      -\sum_{l=1}^{d-k} \sum_{l'=1}^d\bigg[
	2\frac{\partial a_{jr}}{\partial x_{i'}}\frac{\partial^2 F}{\partial
	x_r\partial x_{j'}}(a^{-1})_{jr'} +       
      \frac{\partial^3 F}{\partial
    x_{r'}\partial x_{i'}\partial x_{j'}}\bigg]
    v_{l'i'}v_{l'j'}v_{lr'}v_{li}\,
    e^{-\lambda_l s} \int_0^s e^{(\lambda_l -2\lambda_{l'})u}\,du \\
      = &
\sum_{l=1}^{d-k}\sum_{l'=1}^d
\bigg[2\lambda_{l'}\frac{\partial (a^{-1})_{i'r}}{\partial x_{j}}
      - 
      \frac{\partial^3 F}{\partial
    x_{r}\partial x_{i'}\partial x_{j}}\bigg]
      v_{l'i'}v_{l'j}v_{lr}v_{li}\,
      \int_0^s e^{-2\lambda_{l'}\,u}\,du\,. 
    \end{align*}
To further simplify the last expression above, 
 we differentiate the identity
    \begin{align*}
       \frac{\partial^2 F}{\partial x_{i'}\partial x_{j}}
      v_{l'i'}v_{l'j} = \lambda_{l'}\,,
    \end{align*}
where $l'$ is fixed, $1 \le l' \le d$, along the eigenvector $\bm{v}_l$, which
gives
    \begin{align*}
      & \frac{\partial^3 F}{\partial x_{r}\partial x_{i'}\partial x_{j}}
      v_{l'i'}v_{l'j}v_{lr}\\
      =& -2 \frac{\partial^2 F}{\partial x_{i'}\partial x_{j}}
      \frac{\partial v_{l'i'}}{\partial x_{r}} v_{l'j}
      v_{lr} +
      \frac{\partial \lambda_{l'}}{\partial x_{r}} v_{lr}\\
      =&-2 \lambda_{l'}
       (a^{-1})_{i'r'} v_{l'r'}
      \frac{\partial v_{l'i'}}{\partial x_{r}} 
      v_{lr} +
      \frac{\partial \lambda_{l'}}{\partial x_{r}} v_{lr}\,.
    \end{align*}
    Therefore, taking the limit $s\rightarrow +\infty$, using 
    the relations (\ref{v-pa-and-remain}),
    (\ref{det-equals-eigenvalue-product}), and
    Lemma~\ref{lemma-identity-in-i12} in Appendix~\ref{app-sec-manifold}, we can compute
    \begin{align}
      &P_{ii'} a_{jr}\frac{\partial^2 \Theta_{i'}}{\partial x_j\partial x_r}\notag \\
      =&\lim_{s\rightarrow +\infty} P_{ii'} 
      a_{jr}\frac{\partial^2 \varphi_{i'}}{\partial x_j\partial x_r}(x,s)\notag \\
      = &\lim_{s\rightarrow +\infty} 
      \sum_{l'=d-k+1}^d\sum_{l=1}^{d-k}
      \bigg[2\lambda_{l'}\frac{\partial (a^{-1})_{i'r}}{\partial x_{j}}
v_{l'i'}v_{l'j}
+2 \lambda_{l'}
       (a^{-1})_{jr'} v_{l'r'}
      \frac{\partial v_{l'j}}{\partial x_{r}} -
      \frac{\partial \lambda_{l'}}{\partial x_{r}} 
\bigg]
      v_{lr}v_{li}\,
      \int_0^s e^{-2\lambda_{l'}\,u}\,du\notag \\
      =&\frac{\partial (a^{-1})_{i'r}}{\partial x_{j}} \big(a_{i'j} -
      (Pa)_{i'j}\big) (Pa)_{ir} - 
      \frac{1}{2} \frac{\partial (a^{-1})_{jr'}}{\partial x_{r}} \big(a_{jr'} -
      (Pa)_{jr'}\big) (Pa)_{ir}
      - \frac{1}{2} (Pa)_{ir} \frac{\partial \ln\det \Psi}{\partial x_r}\notag\\
      =&
      -P_{ir} \frac{\partial a_{jr}}{\partial x_{j}} - \frac{\partial
	P_{jr}}{\partial x_{j}} (Pa)_{ir} + P_{ir} 
	\frac{\partial (Pa)_{jr}}{\partial x_{j}} 
+ \frac{\partial P_{jr}}{\partial x_j} (Pa)_{ir}\notag \\
=& 
      -P_{ir} \frac{\partial a_{jr}}{\partial x_{j}} + P_{ir} 
	\frac{\partial (Pa)_{jr}}{\partial x_{j}} \,.
	\label{p-a-d2-theta}
    \end{align}

    On the other hand, differentiating the relation $\xi(\Theta(x)) \equiv
    \bm{0}$ twice and using (\ref{theta-1st-equals-p-in-proof}), we get
    \begin{align*}
      \frac{\partial\xi_{\gamma}}{\partial x_{i'}} \frac{\partial^2
      \Theta_{i'}}{\partial x_j\partial x_r}
      =-
      \frac{\partial^2
      \xi_{\gamma}}{\partial x_{i'} \partial x_{j'}} 
      \frac{\partial \Theta_{i'}}{\partial x_j} 
      \frac{\partial \Theta_{j'}}{\partial x_r}  = 
      - \frac{\partial^2
      \xi_{\gamma}}{\partial x_{i'} \partial x_{j'}} P_{i'j}P_{j'r}\,,
    \end{align*}
       for $1 \le \gamma \le k$.
    Therefore, using $PaP^T=P^2a=Pa$ and $Pa\nabla\xi_\gamma=0$, we
    can compute
    \begin{align}
      \begin{split}
	& (\delta_{ii'}-P_{ii'})  
      a_{jr}\frac{\partial^2 \Theta_{i'}}{\partial x_j\partial x_r} \\
      = &   (\Psi^{-1})_{\alpha\gamma}  (a\nabla \xi_\alpha)_i
      \frac{\partial\xi_{\gamma}}{\partial x_{i'}} 
      a_{jr}\frac{\partial^2 \Theta_{i'}}{\partial x_j\partial x_r} \\
      = &
- (\Psi^{-1})_{\alpha\gamma}
(a\nabla \xi_\alpha)_i
\frac{\partial^2 \xi_{\gamma}}{\partial x_{i'} \partial x_{j'}} 
P_{i'j}P_{j'r} a_{jr} \\
 =&
- (\Psi^{-1})_{\alpha\gamma}
(a\nabla \xi_\alpha)_i
(\partial^2_{i'j'} \xi_{\gamma})(Pa)_{i'j'} \\
      = &
(Pa)_{i'j'}\frac{\partial P_{ii'}}{\partial x_{j'}} \,.
    \end{split}
 \label{1-p-delta-phi}
    \end{align}
    Summing up (\ref{p-a-d2-theta}) and (\ref{1-p-delta-phi}), we conclude
    that
    \begin{align*}
      a_{jr}\frac{\partial^2 \Theta_{i}}{\partial x_j\partial x_r} 
      = 
      \frac{\partial (Pa)_{ij}}{\partial x_{j}} - 
      P_{ir} \frac{\partial a_{rj}}{\partial x_{j}}\,. 
    \end{align*}
      \end{enumerate}
    \end{proof}
    Now, we prove Theorem~\ref{thm-estimate-scheme-on-submanifold}.
\begin{proof}[Proof of Theorem~\ref{thm-estimate-scheme-on-submanifold}]
  Since we follow the approach in~\cite{conv-time-averaging}, we will only sketch the proof and will mainly focus on the differences.

  First of all, we introduce some notations. Let $x^{(l)}$, $l = 0, 1, \cdots$, be the states generated from the
  numerical scheme (\ref{micro-scheme-initial-repeat}) and let 
  $\psi$ be a function on $\Sigma$. We will adopt the abbreviations $\psi^{(l)} = \psi(x^{(l)})$, $P^{(l)} = P(x^{(l)})$, etc. 
For $j\ge 1$, $D^j\psi[\bm{u}_1, \bm{u}_2, \cdots, \bm{u}_j]$ denotes the 
  $j$th order directional derivatives of $\psi$ along the vectors $\bm{u}_1$, $\bm{u}_2$,
  $\cdots$, $\bm{u}_j$, and $|D^j\psi|_{\infty}$ is the supremum norm of $D^j\psi$ on $\Sigma$.
  Similarly, $D^j\Theta[\bm{u}_1,
  \bm{u}_2, \cdots, \bm{u}_j]$ denotes the $d$-dimensional vector whose $i$th component is 
$D^j\Theta_i[\bm{u}_1, \bm{u}_2, \cdots, \bm{u}_j]$, for $1 \le i \le d$.

Define the vector $\bm{b}^{(l)}=(b^{(l)}_1, b^{(l)}_2, \cdots, b^{(l)}_d)^T$ by 
  \begin{align}
    b_i^{(l)} = \Big(-a_{ij}\frac{\partial U}{\partial x_j}   +
  \frac{1}{\beta}\frac{\partial a_{ij}}{\partial x_j}\Big)(x^{(l)})\,, \quad 1
  \le i \le d\,,
  \label{b-vector}
\end{align} 
  for $l = 0,\,1,\,\cdots$, and set 
  \begin{align}
    \bm{\delta}^{(l)} = \bm{b}^{(l)}h + \sqrt{2\beta^{-1} h}\,\sigma^{(l)}
    \bm{\eta}^{(l)}\,.
    \label{delta-vector-by-b}
\end{align}
 We have 
\begin{align}
  \bm{\delta}^{(l)} = x^{(l+\frac{1}{2})}- x^{(l)}\,,\qquad\mbox{and}\quad
  x^{(l+1)} = \Theta\big(x^{(l+\frac{1}{2})}\big) =
\Theta\big(x^{(l)}+\bm{\delta}^{(l)}\big)\,. 
\label{delta-and-x}
\end{align}
Let $\mathcal{L}$ be the infinitesimal generator of the SDE
  (\ref{dynamics-1-submanifold}) in Theorem~\ref{thm-mu1-mu2}, given by  
\begin{align}
  \begin{split}
  \mathcal{L} =& - (Pa)_{ij} \frac{\partial U}{\partial
  x_j}\frac{\partial}{\partial x_i} + \frac{1}{\beta} \frac{\partial
  (Pa)_{ij}}{\partial x_j}\frac{\partial}{\partial x_i} + \frac{1}{\beta}
  (Pa)_{ij} \frac{\partial^2}{\partial x_i\partial x_j}\\
  =& \frac{e^{\beta U}}{\beta}  \frac{\partial}{\partial x_i}\Big(e^{-\beta
  U}(Pa)_{ij}\frac{\partial}{\partial x_j}\Big)\,, 
\end{split}
\label{l-of-sde-1}
\end{align}
in Remark~\ref{rmk-3}.
We consider the Poisson equation on $\Sigma$
\begin{align}
  \mathcal{L} \psi = f - \overline{f}\,.
  \label{poisson-eqn}
\end{align}
The existence and the regularity of the solution $\psi$ can be established under
  Assumption~\ref{assump-1}--\ref{assump-2}, and the Bakry-Emery condition in
  Section~\ref{sec-n-rn}.
Applying Taylor's theorem and using the fact that $\Theta(x^{(l)}) = x^{(l)}$ since $x^{(l)} \in \Sigma$, we have
\begin{align}
  \begin{split}
  &\psi^{(l+1)} \\
    =& (\psi\circ\Theta)\big(x^{(l)} + \bm{\delta}^{(l)}\big)\\
    =& \psi^{(l)} + D(\psi\circ\Theta)^{(l)}[\bm{\delta}^{(l)}]
    + \frac{1}{2}D^2(\psi\circ\Theta)^{(l)}[\bm{\delta}^{(l)}, \bm{\delta}^{(l)}]
    +
    \frac{1}{6}D^3(\psi\circ\Theta)^{(l)}[\bm{\delta}^{(l)},\bm{\delta}^{(l)},\bm{\delta}^{(l)}]
    + R^{(l)}\,,
  \end{split}
  \label{psi-taylor}
\end{align}
where the reminder is given by 
\begin{align*}
  R^{(l)} = 
  \frac{1}{6} \Big(\int_0^1 s^3\,D^4(\psi\circ\Theta)\big(x^{(l)}+(1-s)\bm{\delta}^{(l)} \big)
  ds\Big)\big[\bm{\delta}^{(l)},\bm{\delta}^{(l)},\bm{\delta}^{(l)},\bm{\delta}^{(l)}\big]\,.
\end{align*}

Now we apply Proposition~\ref{prop-map-phi-1st-2nd-derivative} to 
simplify the expression in (\ref{psi-taylor}). Using the chain rule, the expressions (\ref{delta-vector-by-b})--(\ref{l-of-sde-1}), we can derive 
\begin{align}
  &    \psi^{(l+1)}\notag \\
  =&\, \psi^{(l)} + D\psi^{(l)}\big[P^{(l)}\bm{\delta}^{(l)}+\frac{1}{2}
D^2\Theta^{(l)}[\bm{\delta}^{(l)},\bm{\delta}^{(l)}]\big]
  + \frac{1}{2}D^2\psi^{(l)}[P^{(l)}\bm{\delta}^{(l)}, P^{(l)}\bm{\delta}^{(l)}]\notag\\
  &+ \frac{1}{6}D^3(\psi\circ\Theta)^{(l)}[\bm{\delta}^{(l)},\bm{\delta}^{(l)},\bm{\delta}^{(l)}] + R^{(l)}\notag \\
  =&\, \psi^{(l)} + (\mathcal{L}\psi)^{(l)} h+ \sqrt{2\beta^{-1}h}\, D\psi^{(l)}[(P\sigma)^{(l)}\bm{\eta}^{(l)}]
  + 
  \frac{h^2}{2} D\psi^{(l)}\big[D^2\Theta^{(l)}[
  \bm{b}^{(l)}, \bm{b}^{(l)}]\big]\notag \\
  & + \sqrt{2\beta^{-1}} h^{\frac{3}{2}} D\psi^{(l)}\big[D^2\Theta^{(l)}[
\bm{b}^{(l)}, \sigma^{(l)}\bm{\eta}^{(l)}]\big] +
\frac{h^2}{2}D^2\psi^{(l)}[P^{(l)}\bm{b}^{(l)},P^{(l)}\bm{b}^{(l)}]\notag \\
&+ \sqrt{2\beta^{-1}} h^{\frac{3}{2}} D^2\psi^{(l)}[P^{(l)}\bm{b}^{(l)}, (P\sigma)^{(l)} \bm{\eta}^{(l)}] 
  + \frac{hD\psi^{(l)}}{\beta}
    \big[D^2\Theta^{(l)}[\sigma^{(l)}\bm{\eta}^{(l)},\sigma^{(l)}\bm{\eta}^{(l)}]-a^{(l)}:\nabla^2 \Theta^{(l)}\big]\notag\\
    &\, + \frac{h}{\beta} \Big(D^2\psi^{(l)}[(P\sigma)^{(l)}\bm{\eta}^{(l)},
  (P\sigma)^{(l)}\bm{\eta}^{(l)}] -
  (Pa)^{(l)}:D^2\psi^{(l)}\Big)
 + \frac{1}{6}D^3(\psi\circ\Theta)^{(l)}[\bm{\delta}^{(l)},\bm{\delta}^{(l)},\bm{\delta}^{(l)}]
+ R^{(l)}\,,
\label{psi-k-taylor}
\end{align}
where in the last equation we added and subtracted some terms, and 
we used the identity 
\begin{align}
  D\psi^{(l)}\big[P^{(l)}\bm{b}^{(l)}+\frac{1}{\beta}
  a^{(l)}:\nabla^2 \Theta^{(l)}\big] + \frac{1}{\beta} (Pa)^{(l)}:\nabla^2\psi^{(l)} =
  (\mathcal{L}\psi)^{(l)}\,,
  \label{identity-l-psi}
\end{align}
which can be verified using Proposition \ref{prop-map-phi-1st-2nd-derivative},
(\ref{b-vector}) and (\ref{l-of-sde-1}).
In (\ref{identity-l-psi}), $a:\nabla^2\Theta$ is the vector whose $i$th component is given by
$a_{jr}\frac{\partial^2 \Theta_{i}}{\partial x_j\partial x_r}$, and
$(Pa):\nabla^2 \psi$ is defined in a similar way. 

Summing up (\ref{psi-k-taylor}) for $l=0, 1, \cdots, n-1$, dividing both
sides by $T$,
and using the Poisson equation (\ref{poisson-eqn}), gives
\begin{align}
  \widehat{f}_n - \overline{f} = 
\frac{1}{n}\sum_{l=0}^{n-1} f(x^{(l)}) -
  \overline{f} 
  =
  \frac{\psi^{(n)}-\psi^{(0)}}{T} + \frac{1}{T} \sum_{i=1}^5 M_{i,n} +
  \frac{1}{T} \sum_{i=1}^4 S_{i,n} \,,
  \label{diff-hatfn-f-by-psi}
\end{align}
where 
\begin{align}
  M_{1,n} =& -\sqrt{2\beta^{-1}h}
  \sum_{l=0}^{n-1}D\psi^{(l)}[(P\sigma)^{(l)}\bm{\eta}^{(l)}] \,,\notag \\
    M_{2,n} =& -\sqrt{2\beta^{-1}}
  h^{\frac{3}{2}} \sum_{l=0}^{n-1} D\psi^{(l)}\big[D^2\Theta^{(l)}[\bm{b}^{(l)},
  \sigma^{(l)}\bm{\eta}^{(l)}]\big]\,, \notag\\
  M_{3,n} =& -\frac{h}{\beta}
  \sum_{l=0}^{n-1}
  D\psi^{(l)} \big[D^2\Theta^{(l)}[\sigma^{(l)}\bm{\eta}^{(l)},\sigma^{(l)}\bm{\eta}^{(l)}]-
    a^{(l)}:\nabla^2 \Theta^{(l)}\big]\,,\notag\\
     M_{4,n} = &-
  \sqrt{2\beta^{-1}} h^{\frac{3}{2}} \sum_{l=0}^{n-1}
  D^2\psi^{(l)}\big[P^{(l)}\bm{b}^{(l)}, (P\sigma)^{(l)} \bm{\eta}^{(l)}\big]\,, \label{m-terms-submanifold}\\
    M_{5,n} =&
  -\frac{h}{\beta} \sum_{l=0}^{n-1} \Big(D^2\psi^{(l)}\big[(P\sigma)^{(l)}\bm{\eta}^{(l)},
  (P\sigma)^{(l)}\bm{\eta}^{(l)}\big] - (Pa)^{(l)}:D^2\psi^{(l)}\Big)\,, \notag
\end{align}
and
\begin{align}
  S_{1,n} =& -\frac{h^2}{2} \sum_{l=0}^{n-1}
  D\psi^{(l)}\big[D^2\Theta^{(l)}[\bm{b}^{(l)},\bm{b}^{(l)}]\big] \,, \notag \\
    S_{2,n} =& -\frac{h^2}{2} \sum_{l=0}^{n-1}
  D^2\psi^{(l)}\big[P^{(l)}\bm{b}^{(l)}, P^{(l)}\bm{b}^{(l)}\big]\,,
  \label{s-terms-submanifold}\\
  S_{3,n} =& -\sum_{l=0}^{n-1} R^{(l)}\,, \qquad
  S_{4,n} = -\frac{1}{6}
  \sum_{l=0}^{n-1}
  D^3(\psi\circ\Theta)^{(l)}\big[\bm{\delta}^{(l)},\bm{\delta}^{(l)},\bm{\delta}^{(l)}\big]\,.
  \notag
\end{align}
Using (\ref{delta-vector-by-b}),  the last term $S_{4,n}$ above 
can be further decomposed as
\begin{align*}
  S_{4,n} = M_{0,n} + S_{0,n}\,,
\end{align*}
where 
\begin{align}
  \begin{split}
  M_{0,n} =& -\frac{\sqrt{2\beta^{-1}}}{6} h^{\frac{3}{2}} 
    \sum_{l=0}^{n-1} \bigg(\frac{2}{\beta}
    D^3(\psi\circ\Theta)^{(l)}\big[\sigma^{(l)}\bm{\eta}^{(l)},\sigma^{(l)}\bm{\eta}^{(l)},\sigma^{(l)}\bm{\eta}^{(l)}\big]+ 
    3 h D^3(\psi\circ\Theta)^{(l)}\big[\bm{b}^{(l)},\bm{b}^{(l)},
    \sigma^{(l)}\bm{\eta}^{(l)}\big]\bigg)\,, \\
    S_{0,n} =& - \frac{h^2}{6} \sum_{l=0}^{n-1} \bigg(\frac{6}{\beta}
    D^3(\psi\circ\Theta)^{(l)}\big[\bm{b}^{(l)},
    \sigma^{(l)}\bm{\eta}^{(l)},\sigma^{(l)}\bm{\eta}^{(l)}\big] + hD^3(\psi\circ\Theta)^{(l)}\big[
    \bm{b}^{(l)},\bm{b}^{(l)},\bm{b}^{(l)}\big]\bigg)\,.
  \end{split}
  \label{s4-m0-s0-submanifold}
\end{align}
Notice that the terms $M_{i,n}$, $0 \le i \le 5$, are all martingales and in
particular we have $\mathbf{E} M_{i,n} = 0$. Therefore, since the level set
$\Sigma$ is compact (Assumption~\ref{assump-2}), the first conclusion follows from the estimates 
\begin{align}
  \begin{split}
  &
   |S_{1,n}| \le C|D\psi|_{\infty}\,h\,T\,, \qquad |S_{2,n}| \le
   C|D^2\psi|_{\infty}\, h\,T\,,\\
   &
   \mathbf{E}|S_{0,n}| \le C|D^3\psi|_{\infty}\, h\,T \,,\qquad
     \mathbf{E}|S_{3,n}| \le C|D^4\psi|_{\infty}\,h\,T \,,
  \end{split}
  \label{m-s-bound-1}
\end{align}
while the second conclusion follows by squaring both sides of (\ref{diff-hatfn-f-by-psi}) and using
the estimates 
\begin{align}
  \begin{split}
  & \mathbf{E}|S_{0,n}|^2 \le C h^2 T^2 |D^3\psi|^2_{\infty}  \,, \qquad
    \mathbf{E}|S_{3,n}|^2 \le C h^2 T^2 |D^4\psi|^2_{\infty} \,,    \\
    & \mathbf{E} |M_{0,n}|^2 \le Ch^2 T|D^3\psi|^2_{\infty}\,, \qquad 
     \mathbf{E} |M_{1,n}|^2 \le CT|D\psi|^2_{\infty}\,,\qquad
     \mathbf{E} |M_{2,n}|^2 \le Ch^2 T|D\psi|^2_{\infty}\,,\\
     & \mathbf{E} |M_{3,n}|^2 \le ChT|D\psi|^2_{\infty}\,,  \qquad 
      \mathbf{E} |M_{4,n}|^2 \le Ch^2 T|D^2\psi|^2_{\infty} \,,\qquad 
      \mathbf{E} |M_{5,n}|^2 \le ChT|D^2\psi|^2_{\infty} \,.
  \end{split}
  \label{thm-estimate-square-mn-sn}
\end{align}

As far as the third conclusion (pathwise estimate) is concerned, notice that (\ref{diff-hatfn-f-by-psi}) implies
\begin{align}
  \begin{split}
\big| \widehat{f}_n - \overline{f} \big|
  \le &
  \frac{|\psi^{(n)}-\psi^{(0)}|}{T} + \frac{1}{T}
  \sum_{i=0}^5 |M_{i,n}| + \frac{1}{T} \sum_{i=0}^3 |S_{i,n}| \\
  \le &
  C\Big(h + \frac{1}{T}\Big) + \frac{1}{T} \sum_{i=0}^5 |M_{i,n}|\,, 
\end{split}
\label{diff-hatfn-f-by-psi-pathwise}
\end{align}
where we have used the estimates (\ref{m-s-bound-1}) for $|S_{1,n}|$, $|S_{2,n}|$, and the upper bounds
\begin{align*}
  \begin{split}
  |S_{0,n}|\le& Ch^2 \sum_{l=0}^{n-1} |\bm{\eta}^{(l)}|^2 + Ch^3 n \le ChT\,,
    \quad a.s. \\
  |S_{3,n}|\le& Ch^2 \sum_{l=0}^{n-1} |\bm{\eta}^{(l)}|^4 + Ch^4 n \le ChT\,,
    \quad a.s.
  \end{split}
\end{align*}
which are implied by the strong law of large numbers for $\frac{1}{n}
\sum\limits_{l=0}^{n-1} |\bm{\eta}^{(l)}|^4$, when $n \rightarrow +\infty$.
Finally, we estimate the martingale terms $M_{i,n}$ in (\ref{diff-hatfn-f-by-psi-pathwise}). Notice that, for any $r \ge 1$, we can deduce the following upper bounds
(see \cite{conv-time-averaging})
\begin{align*}
  \begin{split}
    & \frac{1}{T^{2r}} \mathbf{E}|M_{1,n}|^{2r} \le \frac{C}{T^r}\,,\quad
    \frac{1}{T^{2r}} \mathbf{E}|M_{2,n}|^{2r} \le \frac{Ch^{2r}}{T^r}\,, \\
    & \frac{1}{T^{2r}} \mathbf{E}|M_{3,n}|^{2r} \le \frac{Ch^r}{T^r}\,,\quad
    \frac{1}{T^{2r}} \mathbf{E}|M_{4,n}|^{2r} \le \frac{Ch^{2r}}{T^r}\,, \\
    & \frac{1}{T^{2r}} \mathbf{E}|M_{5,n}|^{2r} \le \frac{Ch^{r}}{T^r}\,,\quad
    \frac{1}{T^{2r}} \mathbf{E}|M_{0,n}|^{2r} \le \frac{Ch^{2r}}{T^r}\,,
  \end{split}
\end{align*}
which give 
\begin{align}
  \mathbf{E} 
  \Big(\frac{1}{T} \sum_{i=0}^5 |M_{i,n}|\Big)^{2r} \le 
  \frac{C}{T^{2r}} \sum_{i=0}^5 \mathbf{E}|M_{i,n}|^{2r} \le \frac{C}{T^r}\,.
  \label{bound-sum-of-mi-2r}
\end{align}
Now, for any $0 < \epsilon < \frac{1}{2}$, the Borel-Cantelli lemma 
implies that there is an almost surely bounded random variable $\zeta(\omega)$, such that
\begin{align}
  \frac{1}{T} \sum_{i=0}^5 |M_{i,n}| \le \frac{\zeta(\omega)}{T^{\frac{1}{2} -
\epsilon}}\,.
\label{sum-of-mi-pathwise}
\end{align}
Therefore, the third conclusion follows readily from 
(\ref{diff-hatfn-f-by-psi-pathwise}) and (\ref{sum-of-mi-pathwise}).
\end{proof}
Next, we prove Corollary~\ref{corollary-on-spectral-gap}.
\begin{proof}[Proof of Corollary~\ref{corollary-on-spectral-gap}]
  From the estimates in (\ref{thm-estimate-square-mn-sn}), we know that it is only
  necessary to consider the term $M_{1,n}$ in (\ref{m-terms-submanifold}).
  Recall that $\psi$ solves the Poisson equation (\ref{poisson-eqn}) and we
  can assume $\int_{\Sigma} \psi\,d\mu_1 = 0$ without loss of generosity.
  Applying the Poisson equation, the Poincar{\'e} inequality, and the
  Cauchy-Schwarz inequality, we have the standard estimates
    \begin{align*}
    \int_{\Sigma} \psi^2 d\mu_1 \le& -\frac{1}{K} \int_{\Sigma}
    (\mathcal{L}\psi) \psi\,d\mu_1 \\
    \le& \frac{1}{K} \Big[\int_{\Sigma} (\mathcal{L}\psi)^2
    d\mu_1\Big]^{\frac{1}{2}} \Big( \int_{\Sigma}
    \psi^2\,d\mu_1\Big)^{\frac{1}{2}} \\
    =& \frac{1}{K} \Big[\int_{\Sigma} (f-\overline{f})^2 d\mu_1\Big]^{\frac{1}{2}} \Big( \int_{\Sigma}
    \psi^2\,d\mu_1\Big)^{\frac{1}{2}}\,, 
  \end{align*}
  which implies
  \begin{align}
    \begin{split}
        \Big(\int_{\Sigma} \psi^2\,d\mu_1\Big)^{\frac{1}{2}} \le \frac{1}{K} 
\Big[\int_{\Sigma} (f-\overline{f}\,)^2\,d\mu_1\Big]^{\frac{1}{2}} \,,
      \quad \mbox{and}\hspace{0.2cm}
- \int_{\Sigma} (\mathcal{L}\psi) \psi\,d\mu_1 
\le \frac{1}{K} \int_{\Sigma} (f-\overline{f}\,)^2\,d\mu_1\,.
    \end{split}
    \label{estimate-followed-from-poincare}
  \end{align}
  Since the term $M_{1,n}$ in (\ref{m-terms-submanifold}) is a martingale, we
  have 
  \begin{align*}
    \frac{1}{T^2}\mathbf{E} |M_{1,n}|^2 = \frac{2\beta^{-1}}{T}
    \frac{1}{n} \sum_{l=0}^{n-1}
    \mathbf{E} \Big[\big((Pa)^{(l)}\nabla \psi^{(l)}\big)
    \cdot \nabla \psi^{(l)}\Big]\,.
  \end{align*}
  Applying the first estimate in the conclusion of
  Theorem~\ref{thm-estimate-scheme-on-submanifold}, using
  (\ref{integration-by-part-mu1}) in Remark~\ref{rmk-3}, as well as the
  estimate (\ref{estimate-followed-from-poincare}), we obtain
  \begin{align*}
    \frac{1}{T^2}\mathbf{E} |M_{1,n}|^2 \le& \frac{2\beta^{-1}}{T}\int_{\Sigma} (Pa\nabla \psi)\cdot
    \nabla \psi\, d\mu_1 + C\Big(\frac{h}{T}+\frac{1}{T^2}\Big)\\
    \le & \frac{2\int_{\Sigma} (f-\overline{f}\,)^2\,d\mu_1}{KT} 
    + C\Big(\frac{h}{T}+\frac{1}{T^2}\Big)\,.
  \end{align*}
  The conclusion follows by squaring both sides of
  (\ref{diff-hatfn-f-by-psi}), applying Young's inequality, and using the same argument of Theorem~\ref{thm-estimate-scheme-on-submanifold}.
\end{proof}

Finally, we prove Proposition~\ref{prop-map-pi}, which concerns the properties of the projection map $\Pi$ defined in (\ref{projection-map-pi}).

    \begin{proof}[Proof of Proposition~\ref{prop-map-pi}]
  For $1 \le l \le d$, recall that
  $\bm{p}_{l} = (P\sigma)_{i'l}\bm{e}_{i'}$
  is the tangent vector field defined in Appendix~\ref{app-sec-manifold}
  such that 
  $ \bm{p}_{l} \in T_{x}\Sigma$ at each $x \in \Sigma$.
  Since $\Pi_i(x) = x_i$ for
  $x \in \Sigma$, $1 \le i \le d$, taking derivatives along $\bm{p}_{l}$
  twice, we obtain
  \begin{align}
    \begin{split}
      &\frac{\partial \Pi_i}{\partial x_j} (P\sigma)_{jl}  = (P\sigma)_{il} \,,
    \\
    &\frac{\partial^2 \Pi_i}{\partial x_j\partial x_r} 
    (P\sigma)_{jl}(P\sigma)_{rl} = (P\sigma)_{rl} \frac{\partial
      (P\sigma)_{il}}{\partial x_r} - \frac{\partial \Pi_i}{\partial x_j} 
      (P\sigma)_{rl} \frac{\partial (P\sigma)_{jl}}{\partial x_r} \,.
    \end{split}
 \label{eqn-lemma-pi-0}
  \end{align}
Notice that, for a function which only depends on the state and is evaluated at $x \in \Sigma$,
      we will often omit its argument in order to keep the notations simple.

 On the other hand, the vector
 $\bm{\sigma}_{l} - \bm{p}_{l} = ((I-P)\sigma)_{i'l}\bm{e}_{i'} \in (T_x\Sigma)^{\perp}$ (the complement
 of the subspace $T_x\Sigma$ in $T_x\mathcal{M}$). Let
 $\phi(s)$ be the geodesic curve in $\mathcal{M}$ such that $\phi(0) = x$ and $\phi'(0) = 
 \bm{\sigma}_{l} - \bm{p}_{l}$.
 We have $\Pi_i(\phi(s)) = x_i$, $\forall s \in [0, \epsilon)$ for some
   $\epsilon > 0$. Taking derivatives with respect to $s$ twice, we obtain
   \begin{align*}
     \begin{split}
     &\frac{\partial \Pi_i}{\partial x_j}(\phi(s))\frac{d\phi_j(s)}{ds} =
     0\,, \\
     &\frac{\partial^2 \Pi_i}{\partial x_j\partial x_r}(\phi(s))
     \frac{d\phi_j(s)}{ds}\frac{d\phi_r(s)}{ds}=-\frac{\partial
     \Pi_i}{\partial x_j}(\phi(s)) \frac{d^2\phi_j(s)}{ds^2} 
     =\frac{\partial
     \Pi_i}{\partial x_j}(\phi(s))\,\Gamma^j_{rr'}(\phi(s)) \frac{d\phi_r(s)}{ds}
     \frac{d\phi_{r'}(s)}{ds}\,,
   \end{split}
   \end{align*}
   for $1 \le i \le d$, where $\phi_j$ denotes the $j$th component of $\phi$, and
   the geodesic equation of the curve $\phi$ has been used to obtain the last expression above. In
   particular, setting $s=0$, we obtain 
   \begin{align}
     \begin{split}
     &\frac{\partial \Pi_i}{\partial x_j}
     \big(\sigma_{jl} - (P\sigma)_{jl}\big) = 0\,, \\
     &
     \frac{\partial^2 \Pi_i}{\partial x_j\partial x_r} 
     \big(\sigma_{jl} - (P\sigma)_{jl}\big)
\big(\sigma_{rl} - (P\sigma)_{rl}\big)
      =\frac{\partial \Pi_i}{\partial x_j} \Gamma^j_{rr'}
\big(\sigma_{rl} - (P\sigma)_{rl}\big)
     \big(\sigma_{r'l} - (P\sigma)_{r'l}\big)\,.
   \end{split}
     \label{eqn-lemma-pi-1}
   \end{align}
   Combining the first equations in both (\ref{eqn-lemma-pi-0}) and
   (\ref{eqn-lemma-pi-1}), we can conclude that $\frac{\partial \Pi_i}{\partial x_j}
   = P_{ij}$ at $x \in \Sigma$. 
   Since (\ref{eqn-lemma-pi-1}) holds at any $x \in \Sigma$, taking
   the derivative in the first equation of (\ref{eqn-lemma-pi-1}) along the tangent vector 
   $\bm{p}_{l}\in T_x\Sigma$, we obtain 
   \begin{align}
     \frac{\partial^2 \Pi_i}{\partial x_j\partial x_r}
\big(\sigma_{jl} - (P\sigma)_{jl}\big) (P\sigma)_{rl} = - 
     \frac{\partial \Pi_i}{\partial x_j} 
     (P\sigma)_{rl} 
     \frac{\partial \big(\sigma_{jl} - (P\sigma)_{jl}\big) }{\partial x_r}\,.
     \label{eqn-lemma-pi-2}
   \end{align}
   Combining (\ref{eqn-lemma-pi-0}), (\ref{eqn-lemma-pi-1}) and (\ref{eqn-lemma-pi-2}),
   using Lemma~\ref{lemma-identity-in-i12} in Appendix~\ref{app-sec-manifold},
   the expression in (\ref{gamma-det-exp}),
   the relations
   \begin{align*}
     &(P\sigma)_{jl} (P\sigma)_{rl} = (PaP^T)_{jr} = (Pa)_{jr}\,, \\
     &\big(\sigma_{rl} - (P\sigma)_{rl}\big) \big(\sigma_{r'l} -
(P\sigma)_{r'l}\big)= a_{rr'} - (Pa)_{rr'} \,,
   \end{align*}
   and the integration by parts formula, we can compute
   \begin{align*}
     &\frac{\partial^2 \Pi_i}{\partial x_j\partial x_r} a_{jr}  \\
     &\frac{\partial^2 \Pi_i}{\partial x_j\partial x_r} 
     \big(P\sigma + (\sigma - P\sigma)\big)_{jl}\big(P\sigma + (\sigma - P\sigma)\big)_{rl}
       \\
     = &
     \frac{\partial^2 \Pi_i}{\partial x_j\partial x_r} 
     (P\sigma)_{jl} (P\sigma)_{rl}
     + 
     2 \frac{\partial^2 \Pi_i}{\partial x_j\partial x_r} 
     (\sigma - P\sigma)_{jl} (P\sigma)_{rl}
     + \frac{\partial^2 \Pi_i}{\partial x_j\partial x_r}
(\sigma - P\sigma)_{jl}(\sigma - P\sigma)_{rl} \\
     =& 
(P\sigma)_{rl} \frac{\partial (P\sigma)_{il}}{\partial x_r} - P_{ij} 
      (P\sigma)_{rl} \frac{\partial (P\sigma)_{jl}}{\partial x_r} 
      - 2P_{ij} (P\sigma)_{rl} 
     \frac{\partial \big(\sigma_{jl} - (P\sigma)_{jl}\big) }{\partial x_r}\\
     &
     + P_{ij}\Gamma^j_{rr'} \big(\sigma_{rl} - (P\sigma)_{rl}\big) \big(\sigma_{r'l} - (P\sigma)_{r'l}\big) \\
     =& 
     \bigg[(P\sigma)_{rl} \frac{\partial (P\sigma)_{il}}{\partial x_r} + P_{ij} 
      (P\sigma)_{rl} \frac{\partial (P\sigma)_{jl}}{\partial x_r} 
      - 2P_{ij} (P\sigma)_{rl} \frac{\partial \sigma_{jl}}{\partial x_r}\bigg] 
      + P_{ij}\Gamma^j_{rr'} \big(a_{rr'} - (Pa)_{rr'}\big)\\
     =& 
     \bigg[2(P\sigma)_{rl} \frac{\partial (P\sigma)_{il}}{\partial x_r} -  
(P\sigma)_{jl}
      (P\sigma)_{rl} \frac{\partial P_{ij}}{\partial x_r} 
      - 2 (P\sigma)_{rl} \frac{\partial (P\sigma)_{il}}{\partial x_r} 
      + 2 (P\sigma)_{rl} \sigma_{jl} \frac{\partial P_{ij}}{\partial x_r}\bigg] \\
    &+ (Pa)_{ij}\frac{\partial (a^{-1})_{lj}}{\partial x_r}
    \big(a_{lr}-(Pa)_{lr}\big)
      -\frac{1}{2} (Pa)_{ij} \frac{\partial (a^{-1})_{lr}}{\partial x_j}
      \big(a_{lr}-(Pa)_{lr}\big)\\
     =& 
  (Pa)_{lj} \frac{\partial P_{ij}}{\partial x_l} 
+\bigg[ - P_{ij} \frac{\partial a_{jl}}{\partial x_l}
       - (Pa)_{ij}\frac{\partial P_{lj}}{\partial x_l}
      + P_{il}\frac{\partial (Pa)_{lj}}{\partial x_j}\bigg]
      +(Pa)_{ij}\frac{\partial P_{lj}}{\partial x_l} + 
	\frac{1}{2} (Pa)_{ij} \frac{\partial \ln \mbox{det}\Psi}{\partial
	x_j}\\
     =& 
 - P_{ij} \frac{\partial a_{jl}}{\partial x_l}
      + \frac{\partial (Pa)_{ij}}{\partial x_j}
      +
	\frac{1}{2} (Pa)_{ij} \frac{\partial \ln \mbox{det}\Psi}{\partial x_j}\,.
   \end{align*}
 \end{proof}

\renewcommand{\bibsection}{\subsection*{\large References}}  
\bibliographystyle{plainnat}
\bibliography{subref}

\end{document}